\documentclass[12pt,a4paper]{amsart}
\usepackage{amsmath, amssymb}
\newfont {\cyr} {wncyr10}
\renewcommand{\labelenumi}{{(\roman{enumi})}}

\usepackage{amsfonts}
\usepackage{amsmath}
\usepackage{amssymb}
\usepackage{setspace}
\usepackage{multicol}

\newtheorem{theorem}{Theorem}[section]
\newtheorem{lemma}[theorem]{Lemma}

\newtheorem{corollary}[theorem]{Corollary}
\newtheorem{definition}[theorem]{Definition}

\newcounter{claim}[theorem]
\renewcommand{\theclaim}{\noindent{(\thesection.\arabic{theorem}.\arabic{claim})}}

\newcounter{cclaim}[theorem]

\def \lcm{\mathrm {lcm}}
\def \Spin{\mathrm {Spin}}
\def \udot {{}^{\textstyle .}}

\newcommand{\E}{\mathrm{E}}\newcommand{\SU}{\mathrm{SU}}
\newcommand{\F}{\mathrm{F}}\newcommand{\A}{\mathrm{A}}\newcommand{\B}{\mathrm{B}}\newcommand{\M}{\mathcal{M}}
\newcommand{\G}{\mathrm{G}}

\newcommand{\Aut}{\mathrm{Aut}}

\newcommand{\Syl}{\mathrm{Syl}}

\newcommand{\GF}{\mathrm{GF}}
\newcommand{\GL}{\mathrm{GL}}
\newcommand{\Sp}{\mathrm{Sp}}
\newcommand{\SL}{\mathrm{SL}}
\newcommand{\0}{\emptyset}

\newcommand{\PSL}{\mathrm{PSL}}\newcommand{\PSp}{\mathrm{PSp}}
\newcommand{\Sym}{\mathrm{Sym}}
\newcommand{\Alt}{\mathrm{Alt}}

\newcommand{\U}{\mathrm{U}}

\newcommand{\ATLAS}{\textsc{Atlas\ }}

\def \L {\hbox {\rm L}}

\def \POmega {\mathrm {P}\Omega}

\def \GU {\mbox {\rm GU}}

\def \Syl {\hbox {\rm Syl}}

\def \ov {\overline}

\def \wt {\widetilde}

\def \Aut{ \mathrm {Aut}}

\def \Fi {\mbox {\rm Fi}}

\def \J{\mbox {\rm J}}

\def \Ly{\mbox {\rm Ly}}

\def \B{\mbox {\rm B}}

\def \M{\mbox {\rm M}}

\def \HN{\mbox {\rm HN}}

\def \HS{\mbox {\rm HS}}

\def \Th{\mbox {\rm Th}}

\def \ON {\mbox {\rm O'N}}

\def \Co {\mbox {\rm Co}}

\def \Ru {\mbox {\rm Ru}}

\def \Suz{\mbox {\rm Suz}}

\def \McL{\mbox {\rm McL}}

\def \He {\mbox {\rm He}}\def \PSU {\mbox {\rm PSU}}

\begin{document}
\renewcommand{\labelenumi}{(\roman{enumi})}

\title  {Generation of finite simple groups with an application to groups acting on Beauville surfaces}
\author{Ben Fairbairn} \author{Kay Magaard}\author{Christopher Parker}

\address{
School of Mathematics\\
University of Birmingham\\
Edgbaston\\
Birmingham B15 2TT\\
United Kingdom} \email{c.w.parker@bham.ac.uk}

\email {}

\date{\today}

\maketitle \pagestyle{myheadings}

\markright{{\sc Generation of simple groups: Beauville Groups}} \markleft{{\sc Fairbairn, Magaard and Parker}}

\def \Chev {\mathrm {Chev}}\def \Spor {\mathrm {Spor}}
\def \K {\mathcal K}
\def \I {\mathcal I}
\def \LCp {\mathcal {LC}_p}
\def \LTp {\mathcal {LT}_p}
\def \diag{\mathrm {diag}}
\def \LC {\mathcal {LC}}
\section{Introduction}

Recently,
 Bauer, Catanese and Grunewald \cite{2, 3, 5} have  initiated
the study of Beauville surfaces.
They are 2-dimensional complex algebraic varieties which
are rigid, in the sense of admitting no deformations. These surfaces  are defined over the field $\ov {\mathbb{Q}}$ of
algebraic numbers, and provide a geometric action of the absolute Galois group $\mathrm {Gal}(\bar{\mathbb{Q}}/\mathbb{Q})$.
By generalizing Beauville's original example \cite[p.159]{4}, such surfaces can be constructed from finite
groups acting on suitable pairs of algebraic curves. Formally, following Catanese \cite{5},  we make the following definition.

\begin{definition}\label{deff} Suppose that a
finite group $G$  acts by holomorphic transformations  on two  algebraic curves $\mathcal{C}_1$ and $\mathcal{C}_2$  of genus at least $2$. Let $G$ act diagonally on $\mathcal C_1 \times \mathcal C_2$ and assume that
\begin{enumerate}
\item [(a)] $G$ acts effectively on  $\mathcal{C}_1$ and $\mathcal C_2$ so that both $\mathcal{C}_1/G$ and $\mathcal C_2/G$  are isomorphic to the projective line and the coverings $\mathcal{C}_i\rightarrow\mathcal{C}_i/G$. $i=1,2$,  are ramified over at most three points; and
\item[(b)]  $G$ acts freely on  $\mathcal{C}_1\times\mathcal{C}_2$.
\end{enumerate} Then the surface $(\mathcal C_1 \times \mathcal C_2)/G$ is a \emph{Beauville surface of unmixed type}.
\end{definition}

 Condition (a) in Definition~\ref{deff} is equivalent to $\mathcal C_1$ and $\mathcal C_2$
 admitting a regular dessin in the sense of Grothendieck's theory of dessins
d'enfants \cite{6, 11, 23}, or equivalently an orientably regular hypermap \cite{14}, with $G$ acting as
the orientation-preserving automorphism group.

A  particularly attractive feature of this class of surfaces  is the fact that the above definition can be translated into more finitary combinatorial terms that ``internalize" the structure of the surface  within the group $G$ in the following way (see \cite{3}).

\begin{definition}\label{maindef}
Let $G$ be a group. An \emph{unmixed Beauville structure}, or simply a \emph{Beauville structure} of $G$ is a pair triples $(x_1,y_1,z_1), (x_2,y_2,z_2)\in G\times G \times G$ such that for $i=1,2$  the following hold.
\begin{enumerate}
\item $G = \langle x_i,y_i,z_i\rangle$ and $x_iy_iz_i = 1$;
\item $1/o(x_i) + 1 /o(y_i) + 1/o(z_i) < 1$; and
\item no non-identity power of $x_1$, $y_1$ or $z_1$ is conjugate in $G$ to a power of $x_2$, $y_2$ or $z_2$.
\end{enumerate}
The  Beauville structure then has \emph{type} $$((o(x_1),o(y_1),o(z_1)),(o(x_2),o(y_2),o(z_2))).$$ We call a group possessing a Beauville structure a \emph{Beauville group}.
\end{definition}

Property (i) is equivalent to condition ($a$), with $x_i$, $y_i$ and $z_i$ representing the ramification over the three points, property (ii) is equivalent to each of the curves $\mathcal{C}_i$ having genus at least 2 (arising as a smooth quotient of the hyperbolic plane), and property (iii) is equivalent to $G$ acting freely on the product $\mathcal C_1 \times \mathcal C_2$. Note that this last condition is always satisfied if $l_1m_1n_1$ is coprime to $l_2m_2n_2$ -- a useful observation we shall use many times.

For a group $G$, a triple of elements $(x_1,y_1,z_1)$ which generate $G$ and satisfy (i) and (ii) of Definition~\ref{maindef}  is  called a \emph{hyperbolic triple.}

%

In this paper we address a  conjecture of
Bauer, Catanese and Grunewald which asserts  that all non-abelian finite simple  groups except for the alternating group $\Alt(5)$ are Beauville groups \cite[Conjecture 7.17]{3}.

Our main theorem confirms their judgement. Recall that a finite group $G$ is \emph{quasisimple} provided $G/Z(G)$ is a non-abelian simple group and $G=[G,G]$.

\begin{theorem}\label{main}
With the exceptions of $\SL_2(5)$ and $\PSL_2(5)(\cong\Alt(5) \cong \SL_2(4)$), every finite quasisimple group is a Beauville group.
\end{theorem}

Before continuing, we mention that our notation for group extensions is consistent with that in the {\sc {Atlas}} \cite{ATLAS}. The notation for the simple groups is mostly self-explanatory.

A number of special cases of the conjecture were verified in \cite{3}. Since then there have been a number of contributions towards a proof of the conjecture. The primary contributions are as follows.
\begin{itemize}
\item The alternating groups $\Alt(n)$ for $n\geq6$ (and the symmetric groups $\Sym(n)$ for $n\geq5$) were shown to be Beauville groups in \cite{2,5} and \cite{FG} for extensions of the results.
\item The groups $\PSL_2(q)$ were shown to be Beauville groups by Fuertes and Jones in \cite[Section 2]{FJ} (who also considered the groups $\SL_2(q)$) and using very different methods by Garion and Penegini in \cite[Section 3.3]{GP}.
\item The Suzuki groups $^2\B_2(2^{2n+1})$ as well as the small Ree groups $^2\mathrm G_2(3^{2n+1})$ were shown to be Beauville groups by Fuertes and Jones in \cite[Section 6]{FJ}.
\item The groups $\mathrm G_2(q)$, $^3\mathrm D_4(q)$, $\PSL_3(q)$ and $\PSU_3(q)$ were shown to be Beauville groups when $q$ is sufficiently large by Garion and Penegini in \cite[Section 3.4]{GP}.
\end{itemize}

A recent article by  Garion,  Larsen and  Lubotzky \cite{GLL} in which they prove the conjecture for sufficiently large groups.
We mention that a the conjecture of
Bauer, Catanese and Grunewald has also been confirmed by Guralnick and Malle and they have also recently released a preprint presenting the proof \cite{GMal}.

The approach we take to the conjecture builds on two fundamental results. First, a beautiful result of Gow \cite{Gow} (see Theorem~\ref{G} below). This result shows us that if $C_1$ and $C_2$ are conjugacy classes of  regular semisimple elements in a finite Lie type group $G$, then the set $C_1C_2$ contains every non-central semisimple element of $G$.
The second is a theorem of Guralnick, Pentilla, Praeger and Saxl \cite{GPPS} concerning subgroups of $\GL_d(q)$, $q=p^a$ containing elements which they call primitive prime divisors elements,  ${\mathrm{ppd }}(e,q)$, for some $d/2< e \le d$.  The investigations of the classical groups (including the spin groups)  leads us to develop various consequences of the main theorem of \cite{GPPS}. We do this in Section~\ref{SS3}, to provide a collection of consequences of the main theorem in \cite{GPPS} which we believe to be of independent interest and motivates the first part of the title of the paper. These result go beyond what is required for the application we have in  this paper.
We recall the definition of a Zsigmondy prime in Section~\ref{SS2} and then, following Feit \cite{Feit}, define large Zsigmondy divisors. The Zsigmondy primes are examples of primitive prime divisors. Thus we consider what happens if a subgroup of $\GL_d(q)$ contains two Zsigmondy primes (Theorem~\ref{twoZsig}), a large Zsigmondy divisor  (Lemma~\ref{3.2}), and two large Zsigmondy divisors  (Theorem~\ref{exotictwolarge}). The main result is  then our  generation theorem Theorem~\ref{maingeneration} which, for example, guarantees that certain pairs of elements of $\GL_d(q)$  generate $\GL_d(q)$.
We present here the statement of Theorem~\ref{maingeneration} for subgroups of $\GL_d(q)$.

\begin{theorem}\label{twoMMMMM}
Suppose that $q= p^a$, $d \ge 4$, $H \le G=\GL_d(q)$ is irreducible and  $d/2< e<f\le d$. Let $r= \lambda_{ea,p}$, $s= \lambda_{fa,p}$ and assume that $H$ has elements of order $r$ and elements of order $s$.
If $\gcd(d,e,f)=1$ and at least one of $e$ and $f$ is odd, then either $H \ge \SL_d(q)$ or $d=4$, $q=2$  and $H \cong \Alt(7)$.
\end{theorem}
We use Theorem~\ref{twoMMMMM} to produce hyperbolic triples as follows. Let $G= \SL_d(q)$ with $d\ge  5$ and $d/2 < e< f \le d$ be as in the theorem. Select an element $x$ of order $\lambda_{ea,p}$ contained in the subgroup $\SL_{e}(q)$ of $G$. We extend it to a regular semisimple element of $\SL_d(q)$ by multiplying it by a element $x_1$ from $\SL_{d-e}(q)$ (which we take to commute with $\SL_e(q)$) which acts irreducibly on the natural module from $\SL_{d-e}(q)$.
 We then use Gow's Theorem to see that there exists a conjugate $y$ of $xx_1$ such that the product $yxx_1$  has order $\lambda_{fa,p}$. Then the group $H= \langle xx_1,y\rangle$ is irreducible and contains elements of order $\lambda_{ea,p}$ and $\lambda_{ea,p}$. By Theorem~\ref{twoMMMMM}, $H=G$ and we have a hyperbolic triple of a specified type. Of course $x$ can be adjusted by any semisimple element from $\GL_{d-f}(q)$ and we still have generation of $G$. Thus we interpret Theorem~\ref{twoMMMMM} to mean that $G$ has a multitude of hyperbolic triples and as such Theorem~\ref{main} is manifestly true for large $d$.
We have similar generational results for all the classical groups. We also investigate subgroup of $\GL_d(q)$ which contain a large Zsigmondy divisors and have prime degree (Theorem~\ref{Pdegree}). Section~\ref{SS3} closes with two lemmas which provide our first hyperbolic triples in the classical groups.

The proof of Theorem~\ref{main} starts in Section~\ref{SS4}  where we consider the classical groups. The investigation is divided in to four subsections dealing with $\SL_d(q)$, $\SU_d(q)$, $\Sp_d(q)$ and $\Spin_d^\varepsilon(q)$ and their central quotients respectively. Here we must make a remark concerning the spin groups. We always consider their non-faithful action on the  natural orthogonal module. Hence we need to take special care that when we lift our hyperbolic triples for the orthogonal groups, that they actually satisfy condition (iii) of Definition~\ref{maindef}. This is one place where we use the version of Gow's Theorem presented in Section~\ref{SS2} which is a modest generalization of the theorem stated in \cite{Gow}.   When the classical groups are defined in low dimensions, there are not enough conjugacy classes of semisimple elements  to demonstrate that these groups are Beauville groups just using such elements. Hence in Lemmas~\ref{lineardim3}, \ref{U41}, \ref{U3} and \ref{Sp42} we present hyperbolic triples for $\SL_3(q)$, $\SU_4(q)$, $\SU_3(q)$ and $\Sp_4(q)$ which involve unipotent elements. We mention here that the case of $\SL_2(q)$ is  covered by Fuertes and Jones in \cite[Theorem 2.2]{FJ}.  This still leaves various groups defined in small dimensions over small fields where we rely upon computer calculations. We mention that we have used both GAP \cite{GAP} and Magma \cite{Magma} with no particular preference  to carry out such computations.

The exceptional groups of Lie type are the focus of Section~\ref{SS5}. Again Gow's Theorem is critical. The over-groups of maximal tori are known by the work of Liebeck, Saxl and Seitz \cite{LSS} and these results replace \cite{GPPS}  for the analysis of these groups of the larger rank groups $\F_4(q)$, ${}^2\E_6(q)$, $\E_6(q)$, $\E_7(q)$ and $\E_8(q)$. For the smaller rank groups, we resort to the complete lists of maximal subgroups.  Again some small groups are dealt with by computer.  At this stage the bulk of the work has been completed. In Section~\ref{SS6} we consider the alternating groups and their double covers. Even though the work of \cite{FG} shows that the alternating groups are Beauville groups, they do not investigate the double covers of these groups. The sporadic groups and their double covers are dealt with in Section~\ref{SS7} and the resulting hyperbolic triples for the sporadic simple  groups and their covers other than the baby monster, its double cover and the monster are given in terms of standard generators \cite{WilsonStandardGenerators} to guarantee that our results are replicable. Finally we are left with the exceptional covers of the alternating group and of the groups of Lie type. These are the subject of Section~\ref{exceptionalsec}. Again the hyperbolic triples  are presented as words in the standard generators for each of the groups.  Our final section gathers the various pieces of the proof together and contains the proof of Theorem~\ref{main} which of course depends on the classification of the finite simple groups.

\bigskip

\centerline {\sc { Acknowledgement}}

\medskip

The authors wish to express their gratitude to Professor Gareth Jones for introducing them to  Beauville surfaces and Beauville structures and for making them aware of the main conjecture that we resolve in this article.

\section{Preliminary results}\label{SS2}

In this section we collect some background results. In particular, we state Zsigmondy's theorem and its lesser known extension due to Feit. After that we present Gow's Theorem, a result which lies at the centre of our investigations.

\begin{theorem}[Zsigmondy \cite{Z}  (or rather Bang \cite{Bang})]\label{Zig}
 For any natural numbers $a> 1$ and $n > 1$ there is a prime number   that divides $a^n-1$ and does not divide $a^k - 1$ for any natural number $k < n$, with the following exceptions:
 \begin{enumerate}
\item  $ a = 2$ and $n = 6$; and
\item  $a + 1$ is a power of two, and $n = 2$.\end{enumerate}
\end{theorem}

A prime with the property described in Theorem~\ref{Zig} is called a \emph{Zsigmondy prime} for $\langle a,n\rangle$ or a \emph{Zsigmondy } $\langle a,n \rangle$-prime.
If $p$ is a Zsigmondy prime for $\langle a,n\rangle$, then $n$ divides $p-1$. In particular, $p \ge n+1$.

 A Zsigmondy prime $p$ for $\langle a, n\rangle$  is called a \emph{large Zsigmondy prime} for
$\langle a, n \rangle$ if $p > n + 1$ or $p^2$  divides $a^n - 1$. A Zsigmondy prime that is not large, is \emph{small}.

\begin{theorem}[Feit\cite{Feit}] \label{LZ}
If  $a$  and  $n$  are  natural numbers greater than $1$,  then there exists  a  large
Zsigmondy prime for  $\langle a,  n\rangle$  except in the following cases.
\begin{enumerate}
\item $n=2$,   and $a  =   2^s3^t -     1$ for some natural number $s$,  and $t =  0$ or $1$.
\item $a =  2$ and $n =  4$, $6$, $10$, $12$  or $18$.
\item  $a =  3$ and $n =  4$ or $6$.
\item $\langle a,  n\rangle  =  \langle 5, 6\rangle$.
\end{enumerate}
\end{theorem}

\begin{lemma}\label{div}Suppose that $\zeta$ is a Zsigmondy $\langle e,a\rangle$-prime. Then $\zeta$ does not divide $a^j-1$ for all $1\le j < 2e$ with $j \not= e$.
\end{lemma}

\begin{proof} This is true for $j < e$ as $\zeta$ is a Zsigmondy prime. Suppose that $e<j < 2e$. If $\zeta$ divides $a^{j}-1$, then $\zeta$ divides $(a^j-1)- (a^e-1)= a^e(a^{j-e}-1)$. Since $\zeta$ and $a$ are coprime, we have $\zeta $ divides $a^{j-e}-1$ with $j-e< e$ which is a contradiction. Thus the lemma holds.
\end{proof}

We will also need the following well known lemma.

\begin{lemma} \label{gcd} Suppose that $a$ and $b$ are positive integers and the $q$ is a positive power of a prime. Then the following are true:
\begin{enumerate}
\item $\gcd(q^a-1,q^b-1) = q^d-1$ where $d = \gcd(a,b)$.
\item $\gcd(q^a-1,q^a+1) = 1+ j$ where $ j = q \pmod 2$.
\item $\gcd(q-1,(q^n-1)/(q-1))= \gcd(q-1,n)$.
\end{enumerate}
\end{lemma}
\begin{proof}  Evidently $q^d-1$ is a divisor of  $t=\gcd(q^a-1,q^b-1)$. Conversely
$t$ is a divisor of any integral combination of $q^a-1$ and $q^b-1$. Without loss we may assume that
$a > b$ and thus $t$ is a divisor of $q^a - q^b = (q^{a-b}-1)q^b$. As $t$ is coprime to $q$ we see that $t$ divides
$q^{a-b}-1$. If $a = jb+r$, then, by induction on $j$, we see that $t$ divides $q^r-1$. Thus $t$ is a divisor of
$\gcd(q^b-1,q^r-1)$. Continuing as in the Euclidian division algorithm we see that $t$ divides $q^d-1$. Thus the first part follows.
The second part is trivial. For the third part set $ s = \gcd(q-1,(q^n-1)/(q-1))$ and note that $s$ divides
$ 2q^{n-2} + q^{n-3} + ...+ q + 1$ and by induction $ iq^{n-i} + q^{n -i - 1} + ...+1$. Thus $s$ is a divisor of $nq$, hence $n$. So
$ s $ is a divisor of $\gcd(q-1,n)$. The converse is clear.
\end{proof}

For integers  $k \ge 1$, let $\Phi_k(x)$ denote the $k$'th cyclotomic polynomial.

The  proof of the following lemma is extracted from the proof of \cite[10-1]{GL}.
\begin{lemma} \label{cyclo}Suppose that $r$ is a prime and $q$ is a prime power. The following hold:
\begin{enumerate}
\item There is a unique integer $m_0$ which is coprime to $r$ such that $r$ divides $\Phi_{m_0}(q)$.
\item For any positive integer $a$ and any positive integer $m$ coprime to $r$, $r$ divides $\Phi_{r^am}(q)$ if and only if $m=m_0$.
\item If $b$ and $c$ are positive integers such that $b<c$ and $r$  divides both $\Phi_b(q)$ and $\Phi_c(q)$, then $b$ divides $c$, $c/b$ is a power of $r$,   $b= r^dm_0$ where $d$ is a positive integer and $m_0 $ divides $r-1$.
\end{enumerate}
\end{lemma}

\begin{proof} Let $m_0$ be the order of $q$ modulo $r$. Then $m_0$ divides $r-1$.  Suppose that $m_1$ is also coprime to $r$ and that $r$ divides $\Phi_{m_1}(q)$. By the choice of $m_0$, we have $m_0$ divides $m_1$.  Then, by Lemma~\ref{gcd} (iii), $r$ divides $$\gcd(q^{m_0}-1,((q^{m_0})^{m/m_0}-1)/(q^{m_0}-1))= \gcd (q^{m_0}-1,m_1/m_0)$$
 which contradicts $r$ and $m_1$ being coprime. This proves (i).

 Now suppose that $r$ and $m$ are coprime.
Then, as $\Phi_{pn}(x) = \Phi_n(x^p)$ when $n$ divides $p$ and $\Phi_n(x^p) = \Phi_n(x)\Phi_{np}(x)$ when $p$ does not divide $n$, we have $$\Phi_{r^am}(x) = \Phi_m(x^{r^a})/\Phi_m(x^{r^{(a-1)}} )\equiv  \Phi_m(x) ^{\phi(r^a)}\pmod r,$$ where $\phi$ is the totient function. Hence $r$ divides the righthand side if and only if $r$ divides $\Phi_m(q)$ which is if and only if $m=m_0$. This proves (ii).

Part (iii) follows from (i) and  (ii).
\end{proof}
%
%
%
%
%
%
%
%
%

For $G$ a finite simple group of Lie type of characteristic $p$, an element is \emph{semisimple} if its order is relatively prime to $p$  and is  \emph{regular semisimple element} if its centralizer in $G$
has order relatively prime to $p$.
We need a modest extension of Gow's Theorem. The proof is almost identical to that presented by Gow in \cite{Gow}.

\begin{theorem}[Gow]\label{G}
 Let $G$ be a finite quasisimple group of Lie type of characteristic $p$, and let
$s$ be a non-central semisimple element in $G$. Assume that  $R_1$  and $R_2$ are conjugacy classes of
$G$  consisting of regular semisimple elements of $G$. Then there exist $x \in R_1$ and
$ y \in R_2$ such that
$s=xy$.
\end{theorem}
 \def \Irr{\mathrm {Irr}}
\begin{proof} From Humphreys \cite[Theorem 8.5]{Humphreys},   $G$ has a unique $p$-block of defect $0$  and all the remaining $p$-blocks of $G$ have full defect. The $p$-block of defect $0$ consists of a single character namely  the Steinberg character $\mathrm {St}$. For characters $\chi$ of $G$, the character values $\chi(g)$ of $G$ and the ratios $\frac{|g^G|\chi(g)}{\chi(1)}$ are contained in the ring $\mathcal O$ of algebraic integers. Let $\mathfrak P$ be a maximal ideal of $\mathcal O$ containing the  prime $p$. Select representatives  $r_1 \in L_1$ and $r_2 \in L_2$. Note that, as only central semisimple elements centralize a Sylow $p$-subgroup of $G$, the choices of $r_1$, $r_2$ and $s$ imply that $|C_G(r_1)|_p=|C_G(r_2)|_p=1$ and that $|C_G(s)|_p\neq |G|_p$. In particular, $|s^G|\equiv 0 \pmod p$.

By \cite[Exercise 3.9]{Isaacs}, it suffices to show that the structure constant equation $$\sum_{\chi\in \Irr(G)}\chi(r_1)\chi(r_2) \frac{|s^G|\ov{\chi(s)}}{\chi(1)} \not \equiv 0 \pmod {\mathfrak P}.$$
For $\chi$ a non-projective character we have that $\chi $ is in a block of full defect and so, as $p$ divides $|s^G|$,  $\frac{|s^G|\chi(s)}{\chi(1)} \in \mathfrak P$ by \cite[15.41]{Isaacs}. Hence

$$\sum_{\chi\in \Irr(G)}\chi(r_1)\chi(r_2) \frac{|s^G|\ov{\chi(s)}}{\chi(1)}\equiv
\mathrm {St}(r_1)\mathrm {St}(r_2)\frac{ |s^G|\ov{\mathrm {St}(s})}{\mathrm {St}(1)} \pmod {\mathfrak P}.$$

Now from \cite[Theorem 6.4.7 (ii)]{Carter}, we have that $\mathrm{St}(g)= \pm |C_G(g)|_p$ for semisimple elements $g \in G$. Since $r_1$ and $r_2$ are regular semisimple elements of $G$, we have $$\mathrm {St}(r_1)\mathrm {St}(r_2)\frac{ |s^G|\ov{\mathrm {St}(s)}}{\mathrm {St}(1)} =\pm \frac{ |s^G||C_G(s)|_p}{|G|_p}= \pm \frac {|G|_{p'}}{|C_G(s)|_{p'}} \not \in \mathfrak P.$$
This completes the proof.
\end{proof}

\begin{lemma} \label{structureconstant} Suppose that $G$ is a group, $x \in G$,  $Z =\langle x \rangle$ and   $N= N_G(Z)$.
 Assume that \begin{enumerate}
 \item $N$ is the unique maximal subgroup of $G$ containing $Z$;
 \item $|N/C_G(Z)|=k$, $|Z\setminus Z(G)| > \left({k+1} \atop 2\right) $;
 \item  $x^G \cap N\subset Z$; and
 \item for all non-trivial $z \in Z\setminus Z(G)$, the $(x^G,x^G,z^G)$ structure constant
is non-zero.
\end{enumerate}
Then there exists $z \in Z$,  such that there is a hyperbolic triple for $G$ in $x^G\times x^G \times z^G$.
\end{lemma}

\begin{proof} Set $\mathcal X = \{xy\mid x,y \in x^N\}$. Then $|\mathcal X|= |\mathcal X^N|\le \left({k+1} \atop 2\right) $.
 Since $|Z\setminus Z(G)| >\left({k+1} \atop 2\right)$ by (ii),  there
exits a non-trivial element $z \in Z\setminus Z(G)$ which does not lie in $\mathcal X$, Thus the $(x^N,x^N,z^N)$ structure constant is zero. However, by hypothesis (iv), the
structure constant $(x^G,x^G, z^G)$ is non-zero.  Therefore  we find  $y \in x^G$ and $z' \in z^G$ such that
$xyz' = 1$. Set $K= \langle x,y\rangle$. Then $K \not \le Z$. Furthermore
$K \ge \langle x \rangle =Z$ and so, by (i), $K \le N$ or $K= G$. In the former case, we have $y' \in x^G \cap N \subseteq Z$ by (iii), which is impossible.  Hence $G= K$ and the lemma is proved.
\end{proof}

\section{Generation of Classical Groups}\label{SS3}

For a natural number $n>2$ and a prime  $p$ such that $(n,p ) \neq (6,2)$, we let  $\zeta_{n,p}$ be a  Zsigmondy prime for $\langle n,p\rangle$ chosen  maximally from among all Zsigmondy $\langle n,p\rangle$ primes  and then  let $\lambda_{n,p}$ be the largest power of $\zeta_{n,p}$ which divides $p^n-1$. As in Theorem~\ref{LZ} if $\lambda_{n,p} > n+1$, we say that $\zeta_{n,p}$ is \emph{large}.
Of course a Zsigmondy prime which is not large, is \emph{small} and in such cases we have $\lambda_{n,p}=n+1$. We let $\lambda_{6,2}= 9$ and treat it as a large Zsigmondy prime.
Note that by Theorem~\ref{LZ}, if $\lambda_{n,p} = \zeta_{n,p}$ is small then $$(n,p)\in \{(4,2), (6,2), (10,2), (12,2), (18,2), (4,3), (6,3), (6,5)\}$$ is very limited.

In \cite{GPPS} Guralnick, Pentilla, Praeger and Saxl examine subgroups of $\GL_d(p^a)$ that contain so-called \emph{primitive prime divisor elements} of $p^{ea}-1$ for $d/2<e\le d$. The orders of these elements  are Zsigmondy primes for $\langle e, p^{a}\rangle$. Our first result considers the consequences the maximal subgroups of $\GL_d(p^a)$ which have elements of order $\zeta_{e,p^a}$ and $\zeta_{f,p^a}$ for $d/2<e<f\le d$. To do this we exploit the main theorem in \cite{GPPS} and adopt the division of the examples given there.  Indeed we recommend that the reader have a copy of this paper to hand.

Before continuing, however, we recall some terminology related to the maximal subgroups of the classical groups. We follow \cite{KL}.  Suppose that $G$ is a classical group.  We will always denote the \emph{natural module} of $G$ by $V$.
A subgroup $H$ of $G$ is \emph{reducible} if $V$ is reducible as a $H$-module. The subgroup $H$ is \emph{primitive} if it does not preserve a subspace decomposition of $V$. A subgroup which is not primitive is \emph{imprimitive}. Thus the imprimitive subgroups of $\GL_d(q)$ are contained in the wreath products $\GL_r(q)\wr \Sym(d/r)$ and the other classical groups have similar subgroups determined by orthogonal decompositions of $V$. However, in addition in the unitary, symplectic and orthogonal cases there is the subgroup which preserves a decomposition in to two opposite isotropic/singular subspaces is also a maximal subgroup and so we get groups $\GL_d(q):2 \le \SU_{2d}(q)$, $\GL_d(q):2 \le \Sp_{2d}(q)$ and $\GL_{d}(q):2 \le \mathrm O_{2d}^+(q)$. Finally for $\mathrm O_{2d}^\varepsilon(q)$ with $dq$ odd, we may preserve two non-isometric non-degenerate spaces and so we obtain $\mathrm O_{d}(q) \times \mathrm O_d(q)$.

We shall also come across  subfield subgroups maximum amongst these we have. These are $\GL_d(p^{a_0}) \le \GL_d(p^a)$ and $\Sp_d(p^{a_0}) \le \Sp_{d}(p^a)$, where $r=a/a_0$  a prime,  $\mathrm O_d^{\eta} (p^{a_0}) \le \mathrm O^\varepsilon _{d}(p^a)$ where $r=a/a_0$  a prime and $r \eta = \varepsilon$,  $\GU_d(p^{a_0}) \le \GU_{d}(p^a)$ where $r=a/a_0$  is an odd  prime, $\mathrm O_d^\varepsilon (p^a) \le \GU_d(q)$, $q$ odd, $\Sp_d(q) \le \GU_d(q)$, $d$ even.

Finally, we have the extension field subgroups. Here we have $d= mr$ and the containments are as follows: $\GL_m(q^r) \le \GL_d(q)$, $\GU_m(q^r) \le \GU_m(q)$, $\Sp_m(q^r) \le \Sp_d(q)$,  $\GU_{d/2}(q) \le \Sp_d(q)$ ($q$ odd),
$\mathrm O_m^\varepsilon (q^r) \le \mathrm O_d^\varepsilon (q)$,
$\mathrm O_{n/2} (q^2) \le \mathrm O_d^\varepsilon (q)$ ($qn/2$ odd), and, finally,  $\mathrm \GU_{n/2}(q) \le \mathrm O_d^\varepsilon (q)$ where $\varepsilon = (-1)^{n/2}$.

\begin{theorem}\label{twoZsig} Assume that $p$ is a prime, $q=p^a$, $d \ge 3$, $d/2< e< f \le d$ and  $H$ is a primitive, irreducible subgroup of $\GL_d(q)$ with natural module $V$. Assume that $r$ is a Zsigmondy prime for $\langle e,q\rangle$ and $s$ is a Zsigmondy prime for $\langle f,q\rangle$ and set $X= F^*(H)$. If $H$ contains elements of order $r$ and of order $s$, then one of the following holds.
\begin{enumerate}
\item There exists $a_0 $ dividing $a$ such that $X \cong \SL_d(p^{a_0})$, $\Sp_d(p^{a_0})$, $X \cong \Omega_d^\epsilon (p^{a_0})$ or $\SU_d(p^{a_0/2})$.
\item There exists $b>1$ dividing $\gcd(e,d)$ and $\gcd(f,d)$ such that $X$ is a classical group over a field of order $q^b$ in dimension $d/b$.
\item  $X/Z(X)\cong \Alt(n)$, and either $Z(X) =1$, $V$ is the irreducible section of the $n$-dimensional permutation for $X$,  and $r= e+1$ and $s= f+1$   or $X$, $d$, $p$, $(r,e)$ and $(s,f)$ are as in the first  section
of Table~\ref{tot}.

\item $X/Z(X)$ is a sporadic simple group with $X$, $d$, $p$, $(r,e)$ and $(s,f)$ as  in the middle section of Table~\ref{tot}.

\item  $X/Z(X)$ is a simple group of Lie type which is not defined in  characteristic $p$ and either $X/Z(X) \cong \PSL_2(s)$ with $s= 2f+1$, $r=\frac{s-1}{2}=e+1$ and $d = \frac{s\pm 1}{2}$, or $X/Z(X)$, $d$, $p$, $(r,e)$ and $(s,f)$ are as in the last section of Table~\ref{tot}.
\begin{table}[h]
\caption{Exotic examples in Theorem~\ref{twoZsig}}\label{tot}\vskip 5mm
\begin{tabular}{|l|ccccc|}
\hline
&$X$&$d$&$p$&$(r,e)$&$(s,f)$\\
\hline
\hline
1&$2\udot\Alt(7)$&$4$&$p \equiv 1,2,4 \pmod 7$& $(7,3)$&$(5,4)$\\
2&$3\udot\Alt(7)$&$6$&$p\equiv 1 \pmod 6$&$(5,4)$&$(7,6)$\\
3&$6\udot\Alt(7)$&$6$&$p\equiv 1 \pmod {24}$&$(5,4)$&$(7,6)$\\

4&$\Alt(7)$&$4$&$2$&$(7,3)$&$(5,4)$\\
\hline
\hline
5&$\M_{11}$&$5$&$3$&$(5,4)$& $(11,5)$\\
6&$2\udot \M_{12}$&$6$&$3$&$(5,4)$& $(11,5)$\\
7&$\M_{23}$&$11$&$2$&$(11,10)$& $(23,11)$\\
8&$\M_{24}$&$11$&$2$&$(11,10)$& $(23,11)$\\
9&$3\udot \J_3$&$18$&$p\equiv 1, 4 \pmod {15}$&$(17,16)$&$(19,18)$\\
10&$6\udot \Suz$&$12$&$p\equiv 1\pmod {6}$&$(11,10)$&$(13,12)$\\
\hline
\hline
11&$\Sp_6(2)$&$7$&$> 2$&$(5,4)$& $(7,6)$\\
12&$6\udot \PSU_4(3)$&$6$&$p\equiv 1 \pmod 6$&$(5,4)$& $(7,6)$\\
13&$2\udot \PSL_3(4)$&$6$&$3$&$(5,4)$& $(7,6)$\\
14&$6\udot \PSL_3(4)$&$6$&$p\equiv 1 \pmod 6$&$(5,4)$& $(7,6)$\\
\hline
\end{tabular}

\end{table}
\end{enumerate}
\end{theorem}

\begin{proof} We use the main theorem of \cite{GPPS}. Then $H$ is as in one of the Examples 2.1 to 2.9 listed there. The Examples is 2.1 and 2.4 of \cite {GPPS} provide our cases (i) and (ii). Examples 2.2 and 2.3 of \cite{GPPS} cannot arise as by assumption $H$ is both irreducible and primitive.

The groups described in \cite[Example 2.5]{GPPS} normalize symplectic type $2$-groups and $d = 2^c$ for some $c$. Now, in this case we require $\{r,s\} = \{d-1, d+1\}$. As Fermat primes require $c$ to be a power of $2$ and Mersenne primes require $c$ to be a prime, we have $c=2$ and $d=4$. Then $e= d-2=2$ and $f=d= 4$ and we have a contradiction to the assumption that $e> d/2=2$.
 At this stage we know that $X$ is a quasisimple group which acts absolutely irreducibly on $V$. The possibilities for these groups are the subject of Examples 2.6 to 2.9 in \cite{GPPS}. Examining the configurations in Tables~2 to 8 of \cite{GPPS} yield the examples listed in our parts (iii), (iv) and (v) as well as the possibility that $X/Z(X) \cong \PSL_2(s)$,  with $s= 2f+1$, $r=\frac{s-1}{2}=e+1$ and $d = \frac{s\pm 1}{2}$.
\end{proof}

\begin{lemma}\label{3.2}
Suppose that $q=p^a$, $d \ge 4$, $G \le \GL_d(q)$ is irreducible, $d/2 < e\le d$ and $G$ has an element of order $\lambda_{ea,p}$. Then either
\begin{enumerate}
\item $q=p \in \{2,3,5\}$ and $\lambda_{e,p}= e+1$ is small;
\item $F^*(G)$ is as in \cite[Examples 2.1, 2.4 and 2.8]{GPPS}; or
\item one of the following holds:
\begin{enumerate}
\item $d=4$, $q=p^2$, $p \in \{3,5\}$, $\lambda_{6,p}=2e+1= 7$ and $G \cong 2 \udot\Alt(7)$.
\item $d=6$, $q=4$, $\lambda_{10,2}=2e+1=11$ and $G \cong 3\udot \M_{22}$.
\item $d=9$, $q=4$, $\lambda_{18,2}=2e+1=19$ and $G \cong 3\udot \J_3$.
\item $d=4$, $q=9$, $\lambda_{6,3}= 2e+1=7$ and $G \cong 4 \udot \PSL_3(4)$.
\item $d=\frac{s\pm1 }{2}$, $q=p$, $\lambda_{e,p}= 2e+1$ and $F^*(G)/Z(F^*(G)) \cong \PSL_2(s)$.
\end{enumerate}
\end{enumerate}
\end{lemma}

\begin{proof} We scrutinize the cases which arise in \cite{GPPS}. We may suppose that \cite[Examples~2.1, 2.4 and 2.8]{GPPS} do not occur. If in any of the examples we have $\lambda_{ea,p}= e+1$, we infer that $\lambda_{ea,p}$ is small and we have $q\in \{2,3,5\}$ by Theorem~\ref{LZ}. Thus we may assume that $\lambda_{ea,p} \ge 2e+1$ or else (i) holds.
If \cite[Example 2.3]{GPPS} holds, then $G \le \GL_1(q)\wr \Sym(d)$ and $\zeta_{ae,p}= ae+1\le  d$ which means that $a=1$. Since $\zeta_{ae,p}^2$ does not divide $|G|$, we now have that $\lambda_{ea,p}$ is small and so (i) holds in this case.  Suppose that    $G$ is contained in one of the groups listed in \cite[Example 2.4]{GPPS}. Then we have $\zeta_{ae,p} = ae+1 = d \pm 1= 2^m\pm 1$. So again $a=1$. Furthermore, we have that $|\Sp_{2m}(2)|$ is also divisible by $\zeta_{ae,p}= 2^m\pm 1$. Since $\Sp_{2m}(2)$ has no elements of order $(2^{m}\pm 1)^2$, we have that $\lambda_{ae,p}$ is small and so (i) holds in this case as well. The examples listed in \cite[Example 2.5]{GPPS} all have $F^*(G)$ a cover of an alternating group. The only possibilities are that $d=4$  and $\lambda_{ea,p}=7$. There are two configurations  listed in \cite[Table 2]{GPPS}. In the first case we have that the field has order $p^2$. Then $\lambda_{6,p}=7$ (with $p$ odd)  if and only if $p \in \{3,5\}$. In the second case we have $p \equiv 2,4 \pmod 7$ and is odd. Now we have that $\lambda_{3,p}$ divides $p^2+p+1$. Since $p^2+p+1$ is not divisible by either $5$ or $9$, we see that there must be a prime divisor of $p^2+p+1$ which is either greater than $7$ or $7^2$ must divide $p^2+p+1$. It follows that $\lambda_{3,p}>7$ and we see that the second case cannot arise. From \cite[Example 2.7]{GPPS}, as  we have to consider six cases. The first is $\M_{11}$, here we would require  $d=5$, $q=p=3$ and $\lambda_{5,3}= 11$; however, in fact $\lambda_{5,3}= 11^2$. The same argument eliminates $2\udot \M_{12}$. Similarly, we have $2^{11}-1= 23\cdot 89$ and so $\lambda_{11,2}= 89$ and so $\M_{22}$ and $\M_{23}$ are removed as potential examples. This leaves examples  (iii) (a) and (iii)(b). The examples in \cite[Example 2.9, Table 7]{GPPS} throw up only once possibility which accordingly is listed in (iii)(c). Finally \cite[Example 2.9, Table 8]{GPPS} gives us our final example.
\end{proof}

\begin{lemma}\label{subfield} Suppose that $q= p^a$, $d \ge 3$, $H \le \GL_d(q)$  and  $d/2< e \le d$. Assume that $r =\zeta_{ea,p}$. If $x\in H$ has order $r$, then $H$ is not contained in a proper subfield subgroup of $G$.
\end{lemma}

\begin{proof} Suppose that $H$ is contained in subfield subgroup of $\GL_d(q)$. Then $H$ normalizes $X \cong \SL_d(p^{a_0})$ where $a_0 $ divides $a$.
Furthermore, $|HX/X|$  divides $(p^a-1)$.
Therefore $x\in X$.  Since $a_0$ divides $ a$,  $p^{a_0g}-1 \le p^{ad/2}-1$ for all $g \le d$ and so using the standard formulae for $|X|$, we have a contradiction.
\end{proof}

\begin{lemma}\label{imprimitive} Suppose that $q= p^a$, $d \ge 4$, $H \le G=\GL_d(q)$ is  irreducible, $d/2< e<f\le d$. Let $r=\lambda_{ea,p}$ and $s=\lambda_{fa,p}$ and assume that $H$ contains elements of order $r$ and elements of order $s$.  If $H$ is imprimitive, then $p^a=3$, $(r,e)= (5,4)$, $(s,f)= (7,6)$, $d \in \{7,8\}$ and $X \le \GL_1(q)\wr \Sym(d)\cong 2\wr \Sym(d)$.
\end{lemma}

\begin{proof} Since $H$ is imprimitive, it appears in Example~{2.3} of \cite{GPPS}. In particular, we have $\lambda_{ea,p}$ and $\lambda_{fa,p}$ are small. Therefore, $a=1$ and $p \in \{2,3\}$ by Theorem~\ref{LZ}. If $p=2$, then, as $a=1$, $H$ is contained in a subgroup isomorphic to $\Sym(d)$. This is a contradiction as $\Sym(d)$ does not act irreducibly on $V$.
Hence $p=3$, $e=4$ and $f= 6$. It follows that $d \in \{7,8\}$ as claimed.
\end{proof}

\begin{theorem}\label{exotictwolarge} Suppose that $q= p^a$, $d \ge 4$, $H \le G=\GL_d(q)$ is primitive and irreducible and  $d/2< e<f\le d$. Let $r=\lambda_{ea,p}$ and $s=\lambda_{fa,p}$ and assume that $H$ contains elements of order $r$ and elements of order $s$.  Then either \begin{enumerate}
\item there exists $b>1$ dividing $\gcd(e,f,d)$  such that $H$ is a classical group over a field of order $q^b$ in dimension $d/b$; or \item at least one of $\lambda_{ea,p}$ or $\lambda_{fa,p}$ is small and we have  $F^*(H)$, $d$, $q$, $(r,e)$ and $(s,f)$ are as in Table~\ref{exotic}.
\begin{table}[h]\caption{The exotic examples of Theorem~\ref{exotictwolarge}}\label{exotic}\vskip 5mm\begin{tabular}{|c|ccccc|}
\hline
&$F^*(X)$&$d$&$q$&$(r,e)$&$(s,f)$\\
\hline
1&$\Alt(7)\le \GL_4(2)$&$4$& $2$&$(3,7)$&$(4,5)$\\
2&$\Sp_6(2)\le \Omega_7(3)$& $7$& $3$ &$(5,4)$&$(7,6)$\\
3&$ 2\udot\PSL_3(4)\le \Omega_6^-(3)$&$6$&$3$&$(5,4)$&$(7,6)$\\
4&$\Alt(7) \le \Omega_6^-(3)$&$6$&$3$&$(5,4)$&$(7,6)$\\
5&$\Alt(n)\le \Omega_7(3)$, $n=8,9$&$7$&$3$&$(5,4)$&$(7,6)$\\
6&$\Alt(n)\le \Omega_{10}^-(2)$, $n= 11,12$& $10$&$2$&$(7,6)$&$(11,10)$\\
7&$\Alt(13)\le \Omega_{12}^-(2)$& $12$&$2$&$(11,10)$&$(13,12)$\\
8&$\Alt(14)\le \Sp_{12}(2)$& $12$&$2$&$(11,10)$&$(13,12)$\\
9&$\Alt(n)\le \Omega_{14}^+(2)$, $n= 15,16$& $14$&$2$&$(11,10)$&$(13,12)$\\
10&$\Alt(17)\le \Omega_{16}^+(2)$& $16$&$2$&$(11,10)$&$(13,12)$\\
11&$\Alt(18)\le \Sp_{16}(2)$& $16$&$2$&$(11,10)$&$(13,12)$\\
12&$\Alt(n)\le \Omega_{18}^-(2)$, $n= 19,20$& $18$&$2$&$(11,10)$&$(13,12)$\\
13&$\Alt(n)\le \Omega_{18}^-(2)$, $n= 19,20$& $18$&$2$&$(11,10)$&$(19,18)$\\
14&$\Alt(n)\le \Omega_{18}^-(2)$, $n= 19,20$& $18$&$2$&$(13,12)$&$(19,18)$\\
15&$\Alt(21)\le \Omega_{20}^-(2)$& $20$&$2$&$(13,12)$&$(19,18)$\\
16&$\Alt(22)\le \Sp_{20}(2)$& $20$&$2$&$(13,12)$&$(19,18)$\\
17&$\Alt(n)\le \Omega_{22}^+(2)$, $n= 23,24$& $22$&$2$&$(13,12)$&$(19,18)$\\
\hline\end{tabular}\end{table}\end{enumerate}
\end{theorem}

\begin{proof} Assume  that (i)  does not hold. We intend to show that Table~\ref{exotic} is complete.  Because of Lemma~\ref{subfield}, to do this we investigate the cases itemized in Theorem~\ref{twoZsig} (iii), (iv) and (v). Let $x,y \in H$ have order $r$ and $s$ respectively.
Notice that for each of these novel examples, at least one of $r$ and $s$ is small. In particular, we have that $p\in \{2,3,5\}$ and $a=1$ from Theorem~\ref{LZ}.
We consider first the groups listed in Table~\ref{tot}.
Since $H$ acts faithfully on $V$, we need $p$ odd for row one in Table~\ref{tot}. So in the first section of Table~\ref{tot}, we only have to consider $\Alt(7)$. This is listed as row (1) of Table~\ref{exotic}.  In the second sector of Table~\ref{tot}, we find the sporadic simple groups and these are then ruled out by Lemma~\ref{3.2}.

Consider the final division  of Table~\ref{tot}. Then both $r$ and $s$ are small. So $p \in \{2,3\}$ by Theorem~\ref{LZ}. Then, because of the described congruences for $p$ in Table~\ref{tot}, we have  $p=3$. The first case arises in $\Omega_7(3)$ coming from the Weyl group embedding and is listed in  row (2) of Table~\ref{exotic} (recall an absolutely irreducible subgroup of $\GL_n(q)$ fixes a unique type of form). The group $2\udot \PSL_3(4)$ is a subgroup of $\Omega^-_6(3)$ and so this is listed in row (3).

Now consider the possibility that $X$ is an alternating group and $V$ is the natural permutation module. Then both $r=\zeta_{ea,p}= e+1$ and $s=\zeta_{fa,p}= f+1$ (and $\zeta_{6,2}$ is not defined).  Furthermore, the square of neither $r$ nor $s$ divides $|H|$. We have that $a=1$, $p\in \{2,3\}$ and both $e$ and $f$ are even.
If $p=3$, then we have that $e=4$ and $f=6$ ($r= 5$ and $s=7$). Since $d/2< e< f\le d<2e$, we have $6 \le d \le 7$. This gives rows (4) and (5). Now we may assume that $p=2$. We have that $\{e,f\} \subset \{4,6,10,12,18\}$ by Theorem~\ref{LZ}. Consideration of these possibilities yield the examples in rows (6) to (17). Note that the case that $e=4$ and $d=6$ falls as $\Alt(7)$ has no elements of order 9. Furthermore, the described containments follow from the information presented in \cite[Theorem 16.7]{SymplecticAmalgams}.

Finally for the linear groups which appear in Theorem~\ref{twoZsig} (v), we have that $r= e+1= f =\frac{s-1}{2}$ is small and that $s=2f+1$. Thus $r\in \{5,7,11,13,19\}$ by Theorem~\ref{LZ}. It follows that  $s= 11$ or $23$. If $s=11$, then $f=5$ and $p\in \{2,3\}$; however,  $11$ does not divide $2^5-1$ and $11^2$ divides $3^5-1$. So these configurations are ruled out. If $s=23$, then $r=11$ and we have $p=2$. But then $\lambda_{11,2}=89$ and this possibility is also ruled  out.

This completes the proof of Theorem~\ref{exotictwolarge}.

\end{proof}

\begin{theorem}\label{maingeneration} Suppose that $q= p^a$, $d \ge 4$, $H \le G=\GL_d(q)$ is irreducible and  $d/2< e<f\le d$. Let $r= \lambda_{ea,p}$ and $s= \lambda_{fa,p}$ and assume that $H$ has elements of order $r$ and elements of order $s$.

\begin{enumerate}
\item Suppose $\gcd(d,e,f)=1$ and at least one of $e$ and $f$ is odd. Then either $H \ge \SL_d(q)$ or $d=4$, $q=2$  and $H \cong \Alt(7)$.

\item If $H \le K=\GU_d(p^{a/2}) \le G$, $e$ and $f$ are both odd  and $\gcd(d,e,f)=1$, then $H    \ge F^*(K) \cong \SU_d(p^{a/2})$.

\item
 If $q$ is odd, $H \le K\le G$ with $F^*(K)\cong\Sp_{d}(q)$, $d \ge 6$,  and  $\gcd (d,e,f)=2$. Then $H \ge F^*(K)$.

\item Assume that $H \le K$ where $F^*(K) \cong \Omega_d^\epsilon(q)$, $d \ge 7$ and  $\gcd (d,e,f)\le 2$. Then  either
 $H \ge F^*(K) \cong \Omega_d^\epsilon(q)$ or $F^*(X)$, $d$, $q$, $(r,e)$ and $(s,f)$ are as presented in  Table~\ref{exotic}.

 %
%

\end{enumerate}\end{theorem}

\begin{proof} We consider the possibilities given in Theorem~\ref{exotictwolarge}. Let $X = F^*(H)$.
We first of all observe that by Lemma~\ref{subfield} we have that $X$ is not contained in a proper subfield subgroup.

Suppose that the hypothesis of (i) holds. By Lemma~\ref{imprimitive}, as at least one of $e$ and $f$ is odd, we have that $H$ is primitive. If Theorem~\ref{exotictwolarge} (i) holds, then either $X$ is either a classical group defined over $\GF(q)$ or an extension field subgroup. Now note that $r$ is a Zsigmondy prime for $\langle e,p^a\rangle$ and $s$ is a Zsigmondy prime for   $\langle f,p^a\rangle$. Hence if $X$ is an extension field subgroup define over $\GF(q^b)$, then we have $b$ divides  $\gcd(d,e,f)=1$ by assumption. Hence $b=1$ and this is impossible. Therefore $X$ is one of $\SL_d(q)$, $\Sp_{d}(q)$, $\Omega_d^\varepsilon(q)$ or $\SU_d(q^{1/2})$. But the Zsigmondy primes appearing in the orders of the latter three groups are all Zsigmondy primes for $\langle 2j, q\rangle$. But at least one of $e$ and $f$ is odd.  Hence, if Theorem~\ref{exotictwolarge} (i) holds, then $X \cong \SL_d(q)$. Now assume that  Theorem~\ref{exotictwolarge} (ii) holds. Then we have one of the configurations listed in Table~\ref{exotic}. The only possibility is in line (1) and this is the final possibility of part (i).

Suppose that $K \cong \GU_d(p^{a/2})$. Then $|K|= p^{ad(d+1)/4}\prod _{i=1}^d (p^{ai/2}-(-1)^i)$. In particular, we note that  $\lambda_{ga,p}$ divides $|K|$ with $g > d/2$ if and only if $g$ is odd.  Thus  $K$ has elements of order $r$ and $s$ and these elements are in $F^*(K)$. Suppose that $H \le K$ contains such elements and is irreducible. Then by Lemma~\ref{imprimitive}, $H$ is primitive. If Theorem~\ref{exotictwolarge} (i) holds then either $F^*(H)= F^*(K)$ or $H$ is an extension field subgroup and so is defined over a field of order $q^b$. In the latter case, our choice of $e$ and $f$ means that $b=1$ which is impossible.  Thus if Theorem~\ref{exotictwolarge} (i) holds then $X= F^*(K)$. If Theorem~\ref{exotictwolarge} (ii)  holds, we see that there are no possibilities for $X$ listed in Table~\ref{exotic}.

(iii) This is just the same as the last case. Just note that the extension fields can only occur if $b=2$ and then $r$ and $s$ do not both divide the order of the group. There are no examples to consider in Table~\ref{exotic}.

(iv) Suppose now that $G=\Omega_{d}^\epsilon(q)$ and $H$ is a subgroup of $G$ with elements of order
$r$ and $s$. (Note if $G =\Omega_{d}^+(q)$, $f=d$ is not possible).  By Lemma~\ref{subfield} we have that $H$ is not a subfield subgroup. The choice of $r$ and $s$
indicate that at least one $e$ and $f$ is not divisible by $4$, it follows the $X$ is not an extension field subgroup. Now the exceptional possibilities for $X$ are listed in Table~\ref{exotic} and are as indicated in (iv).

\end{proof}

We recall from \cite{AB, Huppert} that a Singer element of a classical group has order as given in the following table.
\begin{table}[h]
\caption{The orders of Singer elements}
\vskip 5mm
\begin{tabular}{|cc|}
\hline$G$&Singer element order\\
\hline $\SL_d(q)$& $(q^d-1)/(q-1)$\\
$\SU_{2d+1}(q)$& $(q^{2d+1}+1)/(q+1)$\\
$\Sp_{2d}(q)$&$q^d+1$\\
$\Omega_{2d}^-(q)$&$(q^d+1)/\gcd(q+1,2)$\\
\hline
\end{tabular}
\end{table}

\begin{lemma}\label{primedegree} Assume that $d$ is a prime and that $G$ is either $\SL_d(q)$ or $\SU_d(q)$. Suppose that $H \le G$ and that $H$ contains a Singer element $x$. Let $Z = \langle x \rangle$.  Then there exist $z' \in Z$, $z \in z'^G$ and $y \in x^G$ such that $G=\langle x,y \rangle$ and $xy=z$.
\end{lemma}

\begin{proof} We have that $d= |N_G(Z)/C_G(Z)|$ and $x$ is a regular semisimple element of $G$. By Gow's Theorem, for $z' \in Z\setminus Z(G)$, the $(x^G,x^G,z'^G)$ structure constant is non-zero.  By \cite{AB}, as $d$ is a prime,  $N_G(Z)$ is the unique maximal subgroup of $G$ containing $Z$. Furthermore, as $Z=C_G(Z)$  and $Z$ is cyclic we have $x^G \cap N = x^N$.

If  $(q^d-1)/(q-1) = |Z|> \left ({d+1}\atop 2\right)$ when $G =\SL_d(q)$ or $(q^d+1)/(q+1) = |Z|> \left ({d+1}\atop 2\right)$ when $G = \SU_d(q)$ then we may apply Lemma~\ref{structureconstant} to obtain the desired result.  If  $G = \SL_d(q)$, then  $ (q^d-1)/(q-1) = |Z|\le  \left( {d+1}\atop 2\right)$ has no solutions from $d$ a prime. Whereas if  $G = \SU_d(q)$, then  $ (q^d+1)/(q+1) = |Z|\le  \left( {d+1}\atop 2\right)$ with $q>2$, then $(d,q) =(5,2)$ and this case has been checked by computer.
 This completes the  proof of the lemma.

\end{proof}

\begin{lemma}\label{Pdegree} Suppose that $q= p^a$, $d \ge 5$ is a prime $H \le G=\GL_d(q)$ is irreducible and  $d/2< e<d$. Let $r= \lambda_{ea,p}$  and assume that $H$ has elements of order $r$. Then either \begin{enumerate} \item $F^*(H) = \SL_d(q)$, $\SU_d(q)$ or $\Omega_{d}(q)$;
\item $q=p$, $H \le \GL_1(3) \wr \Sym(d)$, $d \le 11$ and $r$ is small;
\item $F^*(H)= \Alt(d+1)$ or $\Alt(d+2)$, $q\in\{2,3,5\}$, $d \le 37$ and $r$ is small;
\item $F^*(H)$ is a sporadic simple group and one of the following occur.
\begin{enumerate}
\item $F^*(H) \cong \M_{11}$, $d=5$ and $q=3$ or $9$.
\item $F^*(H) \cong \M_{12}$, $d=11$, $e=10$ and $q=p=2$.
\item $F^*(H) \cong \M_{23}$, $d=11$,  $e=10$,  $q=p=2$.
\item $F^*(H) \cong \M_{24}$, $d=11$,  $e=10$,  $q=p=2$.
\end{enumerate}
\item $d=7$, $e=6$ and $F^*(H) \cong \G_2(q)$ with $p$ odd, ${}^2\G_2(q)$ or $\PSU_3(q)$ with $p=3$.
\item $d=7$,  $F^*(H) \cong \Sp_6(2)$ and either $e \in \{4,6\}$ and $p=q=3$  or $e= 6$ and $q=p=5$.
\item $F^*(H) = \PSL_3(2)$, $d=7$, $e=6$ and $q=p \in\{3,5\}$.
\item $F^*(H)=\PSU_3(3)$, $d=7$, $e=6$ and $p=q=5$.
\item $F^*(H) = \PSL_2(8)$, $d=7$, $e=6$ and $q=p\in \{3,5\}$.
\item $F^*(H) = \PSL_2(7)$, $d=7$, $e=6$ and $q=p\in \{3,5\}$.
\item $F^*(H) = \PSL_2(37)$, $d=e+1=19$ and $q=p=2$.
\item $F^*(H) = \PSL_2(27)$, $d=e+1=13$ and $q=p=2$.
\item $F^*(H) = \PSL_2(25)$, $d=e+1=13$ and $q=p=2$.
\item $F^*(H) = \PSL_2(23)$, $d=e+1=11$ and $q=p=2$.
\item $F^*(H) = \PSL_2(13)$, $d=e+1=7$ and $q=p=2$.
\item $F^*(H) = \PSL_2(11)$, $d=e+1=5$ and $q=p=2$.
\item $F^*(H) = \PSL_2(13)$, $d=e+1=7$ and $q=p=\{3,5\}$.
\item $F^*(H) = \PSL_2(11)$, $d=e+1=5$ and $q=p=\{3,5\}$.
\item $F^*(H) = \PSL_2(9)$, $d=e+1=5$ and $q=p=3$;
\item $F^*(H) = \PSL_2(s)$,  $d= \frac{1}{2}(s\pm 1)$ and $r= 2e+1=s$.
\end{enumerate}
In particular, if $\lambda_{ea,p}$ is not small, then either (i), (v) or (xx) holds.
\end{lemma}

\begin{proof} We again employ \cite{GPPS} and often use the Modular Atlas \cite{Moat} to rule out groups having  certain odd degree representations. In case of \cite[Example 2.1]{GPPS}, we once again apply Lemma~\ref{subfield} and obtain our result. By assumption \cite[Example 2.2]{GPPS} does not hold.  So suppose that $H \le \GL_1(q)\wr \Sym(d)$.  Then $r = ea+1$ which means that $q=p$ and that $r$ is small. So (ii) holds in this case.
Since $d$ is prime, both \cite[Examples 2.4 and 2.5]{GPPS} do not occur. If $F^*(H)$ is as in \cite[Example 2.6]{GPPS} then we either have (iii) or by \cite[Tables 2, 3 and 4]{GPPS}, $d=5$, $r$ is small and $p\ ge 7$ which is impossible. Moving onto \cite[Example 2.7]{GPPS}, suppose that $F^*(H) \cong \M_{11}$ with $d=5$. Then $e=4$.  Since $r \le 11$, we infer that $a \le 2$. Hence $q \in \{3,9\}$. Now in the case that $F^*(H) \cong \M_{11}$ and $d=11$, we have that $e=10$ and $r$ is small, but then $p =2$ by Theorem~\ref{LZ} which is impossible according to \cite[Table 5]{GPPS}. If $F^*(H) = \M_{12}$. Then $d=11$, $e=10$ and $r$ must be small. Again Theorem~\ref{LZ} eliminates this possibility. Suppose that $F^*(H) \cong \M_{23}$ then $r= e+1= 11=d$ or $2e+1$ which means that $a =1$ as if $a=2$, then $r=2e+1$ would be small. The situation for $F^*(H)= \M_{24}$ and $d=11$ is similar to that for $\M_{23}$. If $F^*(H)= \M_{24}$ and $d=23$, then $p \not \in \{2,3\}$ whereas $r$ must be small. Hence Theorem~\ref{LZ} eliminates this case. Similarly, we cannot have $d=23$ and $F^*(H) \cong \Co_3$ or $\Co_2$. Thus (v) holds.  Consider next \cite[Example 2.8]{GPPS}.
Then we have (vi) using the argument in Lemma~\ref{subfield} to eliminate smaller fields. Finally we contemplate the possibilities in \cite[Example 2.9]{GPPS} first dealing with Table 7 in which just two possibilities occur. The first bing $F^*(H) \cong \Sp_6(2)$ with $d=7$ In this case we see that $r$ is small with $e= 4$ or $6$. Since $p> 2$, we infer that $p=3$ or $p=5$ and $e= 6$ from Theorem~\ref{LZ}. This give (vii). The second possibility is than $d=5$ and that $p \ge 7$. Since also in this case $r$ is small we may use Theorem~\ref{LZ} to eliminate this case. So suppose that we have the examples in Table~8 of \cite{GPPS}. In dealing with these example we will repeatedly call on Theorem~\ref{LZ}.  In line one of Table~8 we find the groups $\PSL_n(s)$ with $n$ a prime and $r= (s^n-1)/(s-1)= e+1$. Thus $r$ is small and so we have $r \le 19$ which means first that $n =3$. If $p=2$, then $s$ is odd and so we get $s = 3$ and $H= \PSL_3(3)$ and $q=p$ which doesn't have any odd degree representations in characteristic $2$. If $p$ is odd, then we require $r \le 7$ and we obtain $H= \PSL_3(2)$, $d=r=e+1=7$ and $p=q\in \{3,5\}$. Next suppose that $H = \PSU_3(s)$. Then we have $r= (s^n+1)/(s+1)$ is small. If $p=2$, then $s$ is odd and $r$ is at most $19$. This gives us $\PSU_3(3)$ with $d=7$ and $e=6$. But $\PSU_3(3)$ has no irreducible representations in dimension $7$ over fields of characteristic $2$.  If $p$ is odd, then we require $r \le 7$ which again forces $H = \PSU_3(3)$, again with $d=7$, $e=6$ but with $p=q=5$. Moving on to line 3 of Table 8, we have $H= \PSp_{2n}(s)$ with $n= 2^b \ge 2$ and $s$ odd. Again if $p= 2$, we get $r= \frac {1}{2}(s^n+1) \le 19$ which means that $H = \PSp_4(5)$ with $d= 13= r$ and $q=2$ or $H= \PSp_4(3)$ both of which have no odd degree representations in characteristic $2$.  If $p> 2$, then we require $r \le 7$ which again forces us to consider $\PSp_4(3)$ but Theorem~\ref{LZ} implies $p=3$ and we have a contradiction.  We move on to line 4. Here $H= \PSp_{2n}(3)$ with $n$ an odd prime. Again $r$ is small. The only possibility is that we have $p=2$ and $n=3$ with $d=13$.  The remaining rows of Table $8$ all concern $\PSL_2(s)$ with $s \ge 7$. The first two cases to consider are $s= 2^c$. In this case we have $r= s\pm 1$ is small. Hence $s \pm 1 \le 7$ so $s= 8$ and we have $d= 7$, $e=6$. In the remaining case we have $s$ odd. In line $7$  we have $r=s=d$ is small and so $r \le 19$. Using \cite{Moat} we rule out $\PSL_2(19)$, $\PSL_2(13)$, $\PSL_2(11)$, $\PSL_2(7)$ and $\PSL_2(5)$ in characteristic $2$. If $p$ is odd, then we must have $s=7$ and this possibility is listed. So suppose that the last row of Table $8$ holds. Then we have $s$ odd and $r = \frac{1}{2} (s\pm 1)= d$. Supposing that $p=2$, we then get the possibilities listed in (xi) through (xvi). If $p$ is odd we require $(s \pm 1)/2 \le 7$and we get $s \in \{13,11,9,7\}$ with $d$ respectively $7,5,5,3$.
Finally then we have line 6. Here $s$ is a prime, $d= \frac{1}{2}(s\pm 1)$ and $r= 2e+1=s$.

\end{proof}

\begin{lemma}\label{altnorder} Suppose that $G= \Sym(d)$, $V$ is the deleted permutation module for $G$, $x \in G$ and $C_V(x)= 0$. Then either $\dim V = d-1$ and $x$ has order at most $d$ or $\dim V= d-2$, $d \ge 4$ and $x$ has order at most $(d^2-1)/4$.
\end{lemma}

\begin{proof} Let $W$ be the permutation module for $\Sym(d)$.  Then $\dim C_W(x)$ equals the number of cycles of $x$ in its cycle decomposition.  Since $C_V(x)= 0$, we have that either $\dim V= d-1$ and $\langle x\rangle $ is transitive or $\dim V= d-2$ and $\langle x\rangle$ has at most two orbits. In the first case $x$ has order $d$ and in the second case $\langle x\rangle$ has order at most $(d^2-1)/4$ provided $d \ge 4$.
\end{proof}

\begin{lemma}\label{OPlusetcAgain}
\begin{enumerate}
\item  Suppose that  $G\cong \Sp_{2d}(q)$  with $d \ge 3$. Let $x\in \GL_d(q)\le G$ have order $q^d-1$ fix  two opposite totally isotropic  subspaces,   $y \in \GL_{d-1}(q) \times \Sp_{2}(q)$ be such that $y$ projects to an element of order $q^{d-1}-1$ on the first factor and to an element of order $q+1$ on the second factor and let $z$ be a bireflection of order $q+1$. Then there exits a hyperbolic triple in $x^G\times y^G  \times z^G$.

\item Suppose that  $G\cong \SU_{2d}(q^{\frac 1 2 } )$  with $d \ge 3$. Let $x\in  G$ have order $(q^d-1)/(q^{\frac 1 2}+1)$ fix  two opposite totally isotropic  subspaces,   $y \in \GL_{d-1}(q) \times \SU_{2}(q^{\frac 1 2 })$ have determinant $1$  and be such that $y$ projects to an element of order $(q^{d-1}-1)/(q^{\frac 1 2}+1)$ on the first factor and to an element of order $q^{\frac 1 2}+1$ on the second factor and let $z$ be a bireflection of order $q^{\frac 1 2}+1$. Then there exits a hyperbolic triple in $x^G\times y^G  \times z^G$.
\end{enumerate}
\end{lemma}

\begin{proof}  Since the elements in $x^G$ and $y^G$ are regular semisimple elements,  we may suppose that $xy=z$ by Gow's Theorem.
Set $X = \langle x,y\rangle$. Then $\langle x\rangle$ preserves a decomposition of $V$ into two isotropic  subspaces of dimension $d$ and that $T=\langle x \rangle$ acts irreducibly on these two subspaces which are non-isomorphic as $T$- modules. Since $\langle y\rangle$ preserves no such isotropic subspace,   $X$ acts irreducibly on $V$. Considering the order of $x$ shows that $X$ acts tensor indecomposably on $V$.  We next show that $X$ acts primitively on $V$. Suppose that  $X \le \GL_d(q):2$. If $G= \Sp_{2d}(q)$,  we have $z$ has order $q+1$ and $z$ acts irreducibly on $[V,z]$ which has dimension $2$ so $[V,z]$ is non-degenerate and we cannot have $z \in \GL_d(q)$ fixing preserving just maximal isotropic spaces. Thus  $z\not \in \GL_d(q)$ and exchanges the two isotropic spaces, but then $d\ge 3 > 2=\dim [V,z] \ge d$ which is a contradiction. If $G= \SU_{2d}(q)$, then $z$ has order $q^{\frac 1 2 }+1$ and $[V,z]$ is a direct sum of two non-degenerate spaces which are not isomorphic as $\langle z \rangle$-modules. Hence again we have $z \not \in \GL_d(q)$ and we obtain a contradiction just as before.   Suppose then that $X \le I=\mathrm X_m \wr \Sym(2d/m)$ where $X_m$ is one of  $\Sp_{m}(q)$ or $\GU_m(q)$  in the respective cases and $m$ divides $2d$.  Since the only subspaces of $V$ preserved by $T$ are the two maximal isotropic spaces,   $T$ must act transitively on the blocks preserved by $I$. As any Zsigmondy prime dividing the $|T|$ is at least $d+1$, we have first that $m= 1$ and second that if $x_1= x^m$ is an element of order $\zeta_{da,p}$, then $C_V(x) \neq 0$ which is impossible.
 Hence we have shown that $G$ acts irreducibly,  primitively and tensor indecomposably on $V$. Since $z$ is a bireflection, we may apply the main result of \cite[Theorem 7.1]{GSa} to see that $X$ is a classical group in its natural representation, an alternating group or  one of the following holds:
\begin{enumerate}
\item[(a)]$2d=10$ and $X$ normalizes $\U_5(2)$ with $q$ odd.
\item [(b)] $2d=8$ and $X$ normalizes one of  ${}^3\mathrm D_4(q)$ or $\Omega_7(q)$  with $q$ odd or $\Sp_6(q)$ with $q$ even.
\item  [(c)] $2d=8$ and $X$ normalizes $2\udot \Omega_8^+(2)$ and $q$ is odd.
\item [(d)]$2d= 8$ and $X$ normalizes an extraspecial group $2^{1+8}_\pm$ and $q$ is odd.
\item [(e)] $2d= 8$ and $X$ preserves a tensor cube structure.
\item [(f)] $2d=6$ and $X$ normalizes either $\PSL_3(q)$, $\PSU_3(q)$ with $q$ odd or $\G_2(q)$ with $q$ even.
\item [(g)] $2d=6$ and  one $F^*(X)$ is one of 11 possibilities.
\end{enumerate}
By considering  $|T| $, if $X$ is a classical group, then, as $T$ decomposes $T$ into two subspaces of dimension $d$, we have that either $X=G$ or $G= \Sp_{2d}(q)$ with $q$ even and  $F^*(X) \cong \Omega_{2d}^+(q)$. However, in the later case  $\Omega_2^+(q)$ does not contain a conjugate $y$ (which is conjugate into $\mathrm{O}_{2d}^-(q)$). So this scenario cannot occur.

Suppose that $F^*(X)$ is an alternating group $\Alt(2d+1)$ or $\Alt(2d+2)$ and $V$ is the deleted permutation module. By Lemma~\ref{altnorder} we have that $q^d-1 \le (2d+2)^2/4= (d+1)^2$ if $G \cong \Sp_{2d}(q)$ and $(q^d-1)/(q^{1/2}+1) \le (d+1)^2$ if $G \cong \SU_{2d}(q)$ . Thus we have $q=2$ and $d \in \{3,4,5\}$.  Furthermore, in these cases we have that $F^*(X) = \Alt(2d+2)$ and thus $X$ does not have elements of the required orders.  Now we consider the possibilities in (a) to (g) above. If (a) holds, then the maximal element order in the normalizer of $\U_5(2)$ in $\Sp_{10}(q)$  is at most $48 < (3^5-1)/2$  and,  in $\SU_{10}(q)$ at most $240< (3^{10}-1)/4$ so (a) cannot  occur.  Suppose that $2d=8$. Since the embeddings   in (b) are into the orthogonal group $\Omega_8^+(q)$ (\cite{KL08}), (b) cannot occur. In (c) and (d) the maximal element orders are bounded by $8.2.30$ and so we must have $G= \Sp_8(3)$. This final case has been checked by computer.

Since $\GL_2(q) \wr \Sym(3)$ does not contain elements of order $\zeta_{4a,p}$ (which divides $(q^4-1)$),  we also have that (e) does not hold.

Suppose that $2d=6$. In case  (f),  the cases that $F^*(X)/Z(F^*(X))\cong \PSU_3(q)$ or $\PSL_3(q)$ and that $V$ is the module $V(2\lambda_1)$ fails as the module is not self-dual or unitary. The case that $F^*(X)= \G_2(q)$ can only occur if $q$ is a power of $2$. In this case we have $\G_2(q) \le\Sp_6(q)$. Now we note that $G_2(q)$ doesn't contain an element conjugate to $y$. Hence these cases do not happen.
Now assume that case (f) occurs.  Suppose first that $G = \Sp_6(q)$. Then by considering the order of $x$, $|Z(X)| \le 2$ and considering the values of $p$ itemized in \cite[Theorem 7.1 (d)]{GSa}, we see that all cases are eliminated.  So suppose that $G= \GU_{6}(q^{\frac 1 2 })$.  In particular $q$ is a square. Since $(9^3-1)/4$ is larger than any of the elements in the candidate groups, we must be have $G=\SU_6(2)$ and $x$ thus has order $21=(4^3-1)/3$.  Now we consider $y$. This element has order $3(4^2-1)/3= 15$ and the fifth power of this element is not central in $G$. Now the candidates for $X$ are $3 \udot \M_{22}$ and $3 \udot \U_4(3)$ neither of which have an element of order $15$ which does not power to $Z(G)$. This completes the proof of the lemma.
\end{proof}

\begin{lemma}\label{OPlusetc}   Suppose that  $G\cong \Omega_{2d}^+(q)$ with $d \ge 5$. Then $G$ has  a hyperbolic triple of type $((q^d-1)/\gcd(q-1,2),(q^d-1)/\gcd(q-1,2), (q+1)/\gcd(q+1,2))$ where the elements of order $(q^d-1)/\gcd(q-1,2)$ fix  two opposite totally singular  subspaces and the element of order $(q+1)/\gcd(2,q+1)$ is a bireflection.
\end{lemma}

\begin{proof}  We let $x$ be a Singer element of $\GL_d(q):2 \cap G$ which stabilizes  a
pair $(q+1)/\gcd(2,q+1)$ of opposite singular subspaces. Then $x$ is a regular semisimple element in $G$. Let $z\in \Omega_2^-(q)$ be a bireflection of order $(q+1)/\gcd(2,q+1)$ on $V$.  We use Gow's Theorem and suppose  $y$ is a conjugate of $x$  and all of $x$, $y$ and $xz$ are chosen to satisfy $xy=z$.

 We set $X = \langle x,y\rangle$. Then $\langle x\rangle$ preserves a decomposition of $V$ into two  singular subspace $W_1$ and $W_2$ of dimension $d$ and that $\langle x \rangle$ acts irreducibly on these two subspaces which are non-isomorphic as $\langle x \rangle$- modules. Since $W_3=[V,\langle z\rangle]$ is a 2-dimensional non-degenerate space containing only non-singular vectors, we have that $W_3 \cap W_1= 0= W_3 \cap W_2$. It follows that $X$ acts irreducibly on $V$. We argue that $X$ acts primitively and tensor indecomposably on $V$ just as we did in Lemma~\ref{OPlusetcAgain}. The application of \cite{GSa} is also much the same apart from we have assumed $d \ge 5$.
So, if $X$ is a classical group then $X= G$ by Lemma~\ref{subfield} and the other possibilities for $X$  melt away with the arguments from Lemma~\ref{OPlusetcAgain}.
\end{proof}

\section{Classical Beauville groups}\label{SS4}

In this section we prove that the classical groups are Beauville groups.

\subsection{Linear groups}

In this subsection $G=\SL_d(q) $ with $q=p^a$ and $V$ be the  natural $G$-module. Additionally,  we suppose that $Z \le Z(G)$.
\begin{lemma} Assume that $d \ge 7$.
\begin{enumerate}
\item $G/Z$ has a hyperbolic triple of type $(\lambda_{ad,p},\lambda_{ad,p},\lambda_{a(d-1),p})$.
\item  $G/Z$ has a hyperbolic triple of type $$(\lambda_{a(d-3),p}\lambda_{3a,p}, \lambda_{a(d-3),p}\lambda_{3a,p},\lambda_{a(d-2),p} ).$$
\end{enumerate}
In particular, $G/Z$ is a Beauville group.
\end{lemma}

\begin{proof} For (i), select  $x\in G$ of order $r= \lambda_{ad,p}$ and  $z\in G$ of order $s= \lambda_{a(d-1),p}$. Then $x$ is regular semisimple and consequently by Gow's Theorem~\ref{G}  there is a conjugate $y$ of $x$ such that $ xy$  is conjugate to $ z$ in $G$. But then Theorem~\ref{maingeneration}(i) implies that $(x,y,xy)$ is a hyperbolic triple for $G$ of type $(r,r,s)$. This then gives a hyperbolic triple with the same parameters in $G/Z$ and proves (i).

Now assume  $x \in \SL_{d-3}(q)\times \SL_3(q)\le \SL_d(q)$ has order $r=\lambda_{a(d-3),p}\lambda_{3a,p}$ and $z' \in \SL_{d-2}(q)\le \SL_d(q)$ has order $s= \lambda_{a(d-2)}$. Then $x$ is a regular semisimple element in $G$ and so there exists a conjugate, $y \in G$, of $x$  such that $z=xy$ is conjugate to $z'$.  Note that $H = \langle x,z\rangle$ acts irreducibly on $V$. Since $x^{\lambda_{3a,p}}$ has order $\lambda_{a(d-3),p}$ and $z$ has order $\lambda_{a(d-2)}$, it follows from Theorem~\ref{maingeneration}(i) that $(x,y,z)$ is a hyperbolic triple of type
$(r,r,s)$. Hence (ii) holds.

Finally as the hyperbolic triples we have found are coprime and intersect $Z$ trivially, we have that $G/Z$ is a Beauville group.
\end{proof}

\begin{lemma}\label{SL44}
Assume that  $3 \le d \le 6$, $q \ge 7$, $G = \SL_d(q)$ and $Z\le Z(G)$. Then

\begin{enumerate}\item $G/Z$ has a hyperbolic triple $(x,y,z)$ of type $(\frac{q^d-1}{(q-1)n}, \frac{q^d-1}{(q-1)n},\frac{q-1}{(2,q-1)})$ where $n = |Z|$.
\item If $q+1$ is a power of $2$, then $\SL_4(q)$ has a hyperbolic triple $(x,y,z)$ of type $(\lambda_{4a,p}, \lambda_{4a,p},\frac{q-1}{2})$.
\end{enumerate} In particular, we note that $x$ and $y$ are regular semisimple elements with cyclic centralizers and that $z$ is the image of  $\diag(\lambda^2, \lambda^{-2},1, \dots ,1)$  in (i) and an image of $\diag(\lambda^2, \lambda^2,\lambda^{-4}, 1)$ in (ii) where $\lambda $ is a generator for the multiplicative group of $\GF(q)$.
\end{lemma}

\begin{proof} (i) We work in $\SL_d(q)$ and let $V$ be its natural module. Let $x'$ be an element of order $(q^d-1)/(q-1)$. Then some power of $x$ has order $\lambda_{da,p}$ and, by the choice of $q$, $\lambda_{da,p}$ is large. Furthermore $x$ is a regular semisimple element in $G$. Therefore from Gow's Theorem there are two conjugates, $x$ and $y$, of $x$ which product to give an element $z=\diag(\lambda^2, \lambda^{-2},1, \dots ,1)$ of order $\frac{q-1}{\gcd(2,q-1)}$. Let $H= \langle x,y\rangle$. We may now employ  the main theorem of \cite{GPPS} as $H$ contains $x$ which powers to $\lambda_{ad,p}$. Since $\lambda_{da,p}$ is large and $d \le 6$, we then see that $F^*(H)$ is either an extension field subgroup, a classical subgroup defined over $\GF(q)$,  or one of the following holds:
\begin{enumerate}
\item $d= 4$ and $F^*(H)$ is isomorphic to a subgroup of $2^{1+4}_\pm .\mathrm O_4^\pm(2)$.
\item $d=4$ and $F^*(H)$ is isomorphic to a subgroup of $ 2\udot \Alt(7)$;
\item $d=6$, $p=2$  and $F^*(H) \cong \G_2(q)$;
\item $d=4$, $p=2$ and $F^*(H) \cong {}^2\B_2(q)$;
 or \item $F^*(H) = \PSL_2(s)$, $s \le 13$.
\end{enumerate}
However, $H$ has elements of order $(q^d-1)/(q-1)$  and $q \ge 7$ and so each of these possibilities fail.   Now suppose that $F^*(H)$ is an extension field subgroup. Then, as $H$ contains $z$, we first see that $H$ must be defined over a quadratic extension, and then, as the non-trivial  eigenvalues of $z$  are distinct, we have a contradiction. Therefore $H$ is a classical subgroup defined over $\GF(q)$. However no such subgroup contains elements of order $(q^d-1)/(q-1)$. Hence $G = H$ and $G$ has a hyperbolic triple of type $(\frac{q^d-1}{(q-1)n}, \frac{q^d-1}{(q-1)n},\frac{q-1}2)$ where $n = |Z|$.

Now suppose that $d=4$, $q \ge 5$ and $q+1$ is a power of 2. The  elements $x$ and $y$ of order $\lambda_{4a,p}$ act irreducibly on $V$ and so can be arranged, by Gow's Theorem, to product to $z$ where $\diag(\lambda^2, \lambda^2,\lambda^2, \lambda^{-6})$. Set $H= \langle x,y\rangle$. Since $q+1$ is a power of $2$ and $q>3$, we have that $\lambda_{4a,p}$ is large. In particular, $\lambda_{4a,p} = k4a+1 \ge 8a+1=9$ and so is at least 13. Now using \cite{GPPS} and Lemma~\ref{subfield}, we see that the only possibility is that $H$ is an extension field subgroup or is a classical group defined over $\GF(q)$. Since $z$ centralizes a $1$-dimensional subspace, $H$ is not contained in an extension field subgroup. Suppose that $H$ preserves some sesquilinear form $(\;,\;)$. Since $\lambda_{4a,p}$ does not divided $\GU_4(q)$, $(\;,\;)$ must be  bilinear. Let $v\in V$ be an eigenvector of $z$ with eigenvalue $\lambda^2$. Then $(v,v)=(v.z,v,z)= \lambda^4(v,v)$. Since $\lambda^4\neq 1$, we must have $(v,v)=0$. So $v$ is singular. Now $v^\perp$ is preserved by $z$. Let $w\in V$ with $(w,v)=1$. Then $w.z= \mu w+x$ for $\mu \in \{\lambda^2,\lambda^{-4}, 1\}$. Now $1=(w,v)= (w.z,v.z)= \mu\lambda^2 \in \{\lambda^4, \lambda^{-2}, \lambda^2\}$ which is impossible. It follows that $H = \SL_4(q)$ as claimed.
\end{proof}

\begin{lemma}
\begin{enumerate}
\item If $5 \le d \le 6$ and $q \ge 7$, then $G/Z$ has a hyperbolic triple of type $(\frac{q+1}{\gcd(2,q+1)}\lambda_{(d-2)a,p},\frac{q+1}{\gcd(2,q+1)}\lambda_{(d-2)a,p},\lambda_{(d-1)a,p})$.
\item Let $G = \SL_4(q)$. If $r$ is an odd prime divisor of  $q+1$, then $G/Z$ has a hyperbolic triple of type $(q^2+q+1, q^2+q+1,r)$.
\end{enumerate}
\end{lemma}

\begin{proof}
 Assume that $5 \le d \le 6$. Let $x'$ have order $\frac{q+1}{\gcd(2,q+1)}\lambda_{(d-2)a,p}\in \SL_2(q)\times \SL_{d-2}(q)$ and $z$ have order $\lambda_{(d-1)a,p} \in \SL_{d-1}(q)$. Using Gow's Theorem we  arrange that two conjugates $x$ and $y$ of $x'$ product to $xy=z$. Letting $H=\langle x,y\rangle$, we have that $H$ is irreducible on $V$ and so, when $d \in \{5,6\}$, Theorem~\ref{maingeneration} (i) implies that $H= G$. Thus $G/Z$ has a hyperbolic triple of type $(\frac{q+1}{\gcd(2,q+1)}\lambda_{(d-2)a,p},\frac{q+1}{\gcd(2,q+1)}\lambda_{(d-2)a,p},\lambda_{(d-1)a,p})$.

Assume that $d=4$. Let $r$ be an odd prime divisor of $q+1$ and $z \in \SL_2(q)\times \SL_2(q)$ act so that $V$ restricted to $z$ is not homogeneous. Now choose and element $x'$ of order $q^2+q+1$ contained in the subgroup of $G$ isomorphic to $\SL_3(q)$. Then, as usual, we find $x$ and $y$ conjugate to $x'$ such that $xy=z$. Let $H = \langle x,y\rangle$. Since $z$ preserves only $2$-dimensional subspaces  while $x$ preserves a $1$-spaces and a $3$-space, we infer that $H$ acts irreducibly on $V$. Now we note that $x$ powers to an element of order $\lambda_{3a,p}$ and that $\lambda_{3a,p}$ is large.

   Once again, from \cite{GPPS} we get that $G = \langle x,y\rangle $.

\end{proof}

\begin{lemma}\label{lineardim3} Let $G = \SL_3(q)$, $q>3$. Then if $p$ is odd, then $G/Z$ has a hyperbolic triple of type $(p,p,\frac{(q^2-1)}{\gcd(3,q-1)})$ and if $p=2$, then $G/Z$ has a hyperbolic triple of type $(4,2,\frac{q^2-1}{\gcd (3,q-1)})$.
\end{lemma}

\begin{proof} Suppose that $q>3$. Let
$x=\left(\begin{array}{ccc}1&0&0\cr a&1&0\cr b&1&1
\end{array}\right)$  with $a, b \in \GF(q)^\#$
and
$y= \left(\begin{array}{ccc}1&0&1\cr 0&1&0\cr 0&0&1
\end{array}\right).$

Then $xy$ has characteristic polynomial $1-(b+3-a)x+(3+b)x^2-x^3$. Now let $\lambda \in \GF(q)$ be a generator of the multiplicative group and choose $m \not\in \{\mu+\mu^{-1}\mid \mu \in \GF(q)\}$. This is possible as $q\neq 3$. Then $x^2-\lambda m x+ \lambda^2$ is an irreducible polynomial over $\GF(q)$. Now the matrix $z=\left(\begin{array}{ccc}0&-\lambda&0\cr \lambda&\lambda m&0\cr 0&0&\lambda^{-2}\end{array}\right)$ has characteristic polynomial $(\lambda^{-2}-x)(x^2-\lambda m x+ \lambda^2)= -x^3+(\lambda m +\lambda^{-2})x^2-(\lambda^{-1}m +\lambda^2)x +1$. Since $\lambda$ is a generator of $\GF(q)$ and $q>3$, $\lambda^2\neq 1$, hence $\lambda m +\lambda^{-2} \neq \lambda^{-1}m +\lambda^2$. It follows that we may select $a$ and $b$ so that $xy$ and $z$ have the same characteristic polynomial and we do this. Thus $xy$ has order $q^2-1$.

Let $H = \langle x,y\rangle$. Then $H$ acts irreducibly on the natural module for $G$.  Since $H$ has a cyclic subgroup of order $q^2-1$ we conclude that $G=H$.
\end{proof}

\begin{lemma}\begin{enumerate}
\item $\SL_3(3)$ and $\SL_3(2)$ are Beauville groups.
\item $\SL_4(q)$ is a Beauville group for $q\le 16$.
\item $\SL_5(q)$ and $\SL_6(q)$ are Beauville groups for $q\le 7$.
\end{enumerate}
\end{lemma}

\begin{proof} This was easily checked using computer. The types $(l,m,n)$ and $(l_1,m_2,n_1)$ of the triples discovered are listed in Table~\ref{small}.

\begin{table}[h]\caption{Beauville triple types in $\SL_d(q)$ for $d$ and $q$ small.}\label{small}\vskip 5mm
\begin{tabular}{|c|cc|}
\hline
$(d,q)$&$(l,m,n)$&$(l_1,m_1,n_1)$\cr
\hline

$ (3,2) $&$( 4 , 4 , 4 )$&$( 3 , 3 , 7 )$\cr
$ (3,3) $&$( 4 , 4 , 8 )$&$( 3 , 3 , 13 )$\cr
$ (4,2) $&$( 4 , 4 , 4 )$&$( 3 , 3 , 15 )$\cr
$(4, 3) $&$( 8 , 8 , 8 )$&$( 9 , 9 , 13 )$\cr
$(4, 4) $&$( 4 , 4 , 17 )$&$( 3 , 3 , 15 )$\cr
$(4, 5) $&$( 4 , 4 , 31 )$&$( 3 , 3 , 13 )$\cr
$(4, 7) $&$( 16 , 16 , 14 )$&$( 9 , 9 , 57 )$\cr
$ (4,8)$&$( 4 , 4 , 511 )$&$( 9 , 9 , 195 )$\cr
$(4, 9 )$&$( 4 , 4 , 205 )$&$( 9 , 9 , 91 )$\cr
$(4, 11) $&$( 8 , 8 , 10 )$&$( 3 , 3 , 183 )$\cr
$(4, 13) $&$( 8 , 8 , 244 )$&$( 9 , 9 , 5 )$\cr
$ (4,16) $&$( 4 , 4 , 455 )$&$( 3 , 3 , 4369 )$\cr
$(5, 2 )$&$( 4 , 4 , 14 )$&$( 3 , 3 , 15 )$\cr
$(5, 3 )$&$( 8 , 8 , 80 )$&$( 9 , 9 , 121 )$\cr
$(5, 4 )$&$( 4 , 4 , 14 )$&$( 9 , 9 , 255 )$\cr
$(5, 5 )$&$( 4 , 4 , 781 )$&$( 3 , 3 , 93 )$\cr
$(5, 7 )$&$( 16 , 16 , 2801 )$&$( 9 , 9 , 1197 )$\cr
$(6, 2 )$&$( 4 , 4 , 28 )$&$( 9 , 9 , 15 )$\cr
$(6, 3 )$&$( 8 , 8 , 80 )$&$( 9 , 9 , 121 )$\cr
$(6, 4 )$&$( 4 , 4 , 91 )$&$( 3 , 3 , 33 )$\cr
$(6, 5 )$&$( 8 , 8 , 10 )$&$( 3 , 3 , 651 )$\cr
$(6, 7 )$&$( 8 , 8 , 2800 )$&$( 9 , 9 , 8403 )$\cr\hline
\end{tabular}\end{table}

\end{proof}

\begin{theorem}\label{SLThm}  Suppose that $d \ge 3$, $q=p^a$, $G = \SL_d(q)$ and $Z \le Z(G)$. Then  $G = \SL_d(q)/Z$ is a Beauville group.
\end{theorem}

\begin{proof}
This is the cumulative outcome of combining the results of this subsection.
\end{proof}

\subsection{Unitary groups}

In this subsection we assume that $G = \SU_d(q)$ with $d \ge 3$, $q= p^{a/2}$. We view $G$ as a subgroup of $\GL_d(q^2)$ acting on the natural unitary space $V$. Further we let $Z \le Z(G)$. We note that when $d$ is odd, our choice of notation implies that $\lambda_{da,p}$ divides $|G|$ as the Singer elements in $G$ have order $q^d+1$. . 

\begin{lemma}\label{unitary1} Assume   $d \ge 8$.  Let $k$ be odd with  $ d/2< k< d-2$. Then $G$ has a  hyperbolic triple of type $(\lambda_{ka,p}\lambda_{(d-k)a,p}, \lambda_{ka,p}\lambda_{(d-k)a,p}, \lambda_{(k+2)a,p})$.
\end{lemma}

\begin{proof} 
Since $k$ is odd, and $k+2$ are coprime. Let $r = \lambda_{ka,p}$ and $s = \lambda_{(k+2)a,p}$. Let $x_1$ be an element of $\SU_k(q)$ of order $r$, $z' \in \SU_{k+2}(q)$ have order $s$ and identify these as elements of $\SU_d(q)$ using the obvious embedding. If $d$ is even, then $d-k \ge 3$ and we select an element $x_2$ of order $\lambda_{(d-k)a,p}$  in $\SU_{d-k}(q)$. We then define $x=x_1x_2 \in \SU_{k}(q)\times \SU_{d-k}(q)$. If $d$ is odd, then $d-k \ge 4$ and choose an element $x_2$ in $\SL_{(d-k)/2}(q^2) \le \SU_{d-k}(q)$ of order $\lambda_{(d-k)a,p}$ which acts on the $d-k$ unitary space  such that the restriction to $\langle x_2\rangle$ is a direct sum of two non-isomorphic $\langle x_2\rangle$-modules. Then set $x= x_1x_2 \in \SU_k(q)\times \SU_{d-k}(q)$.

Since $x$ is a regular semisimple element of $\SU_d(q)$ we may find a conjugate $y$ of $x'$ in $\SU_d(q)$ such that $z=xy$ is conjugate to $z'$.  By the choice of $x$ and $z'$, $H = \langle x,z\rangle$ acts irreducibly on $V$ and $H$ contains elements of order $r$ and $s$. Thus Theorem~\ref{maingeneration} (ii) implies that $(x,y,z)$ is a hyperbolic triple of type $(\lambda_{ka,p}\zeta_{(d-k)a,p}, \lambda_{ka,p}\zeta_{(d-k)a,p}, \lambda_{(k+2)a,p})$. Thus the lemma holds.
\end{proof}

We recall that a \emph{bireflection} is an element of $G$ with $\dim [V,g]=2$.

\begin{lemma}\label{unitary2}  Suppose that $G=\SU_{2d+1}(q)$ and let $z\in G$ be semisimple bireflection. Then there exist Singer elements $x$ and $y$ such that $xy=z$ and $G= \langle x,y\rangle$.
\end{lemma}

\begin{proof} Let $z$ be a semisimple element with $\dim [V,z]= 2$.  Then by Gow's Theorem~\ref{G} we can find Singer elements $x$ and $y$ such that $xy=z$.   Setting $H = \langle x,y\rangle$ we have that $z \in H$. Then $H$ has elements which have commutator of dimension $2$ on the natural module. It follows that $ H$ is contained in a quadratic extension group. In particular $d$ is even and we have a contradiction.
\end{proof}

\begin{lemma}\label{unitary} Assume that $d\ge 8$ and $q=p^{a/2}$. Then $G/Z$ is a Beauville group.
\end{lemma}

\begin{proof} Suppose first that $d \ge 9$ is odd. Then  Lemma~\ref{unitary2} implies that $G$ has a hyperbolic triple of type $((q^{d}+1)/(q+1),(q^{d}+1)/(q+1),q+1)$. On the other hand Lemma~\ref{unitary1} gives us a triple of type
$(\lambda_{(d-4)a,p}\zeta_{4a,p}, \lambda_{(d-4)a,p}\zeta_{4a,p}, \lambda_{(d-2)a,p})
$.  It follows that $\SU_d(q)$ is a Beauville group when $d$ is odd.

Suppose that $d\ge 8$ is even. Lemma~\ref{OPlusetcAgain} (ii) implies that $G$ has a hyperbolic triple of type $((q^{2d}-1)/(q+1),(q^{2d}-1)/(q+1), q+1)$ and Lemma~\ref{unitary1} gives us a triple of type $(\lambda_{(d-3)a,p}\zeta_{3a,p}, \lambda_{(d-3)a,p}\zeta_{3a,p}, \lambda_{(d-2)a,p})$. Hence $G$ is a Beauville group in this case also. Finally we note that the hyperbolic triple in Lemma~\ref{unitary1} consists of elements of odd order and so $G/Z$ is also a Beauville group. 
\end{proof}

\begin{lemma}\label{SU5 and 7} For $d \in \{5,7\}$,  $G/Z$ is a Beauville group.
\end{lemma}

\begin{proof} Let $d \in \{5,7\}$. By Lemma~\ref{primedegree}, $\SU_d(q)$ has a hyperbolic triple of type $(\frac{q^d+1}{q+1},\frac{q^d+1}{q+1}, \lambda)$ where $\lambda$ divides $(q^{d}+1)/(q+1)$ and where the first elements of the triple are Singer elements and the element of order $\lambda$ is a power of a Singer element.  If $d = 5$ let $e=3$ and
if $d=7$ let $e=5$. Then set  $r = \lambda_{ea,p}$ and let $x$ be a product of an element of order $r$ and an element of order $q-1$ preserving an orthogonal decomposition of $V$ into an $e$-space and a $2$-space.  Let $y$ have order $q^{d-1}-1$ be an element which preserves a decomposition of $V$ into a sum of two isotropic spaces of dimension $(d-1)/2$ and a $1$-dimensional non-degenerate space. Then $x$ and $y$ are regular semisimple elements of $G$.  By Gow's Theorem we may suppose that the product of $x$ and $y$ is conjugate to $x$.  The choice of $x$ and $y$ means that $H =\langle x,y \rangle$ is an irreducible subgroup of $G$. Now we may apply Theorem~\ref{Pdegree} to see that $G = \SU_d(q)$. Thus $G$ is a Beauville group. The choice of the triples also shows that $G/Z$ is a Beauville group.
\end{proof}

We also need an additional lemma for $\PSU_6(q)$.
\begin{lemma}\label{U6}  Suppose that $q\neq 2$ and that $G= \SU_6(q)$. Then $G$ has  hyperbolic triples of the following types. 
\begin{enumerate}
\item $((q^6-1)/(q+1),(q^6-1)/(q+1),q+1)$ where the elements of order $(q^6-1)/(q+1)$ preserve opposite  isotropic subspaces and the element of order $q+1$ is a bireflection.
\item $(\lambda_{5a,p}, \lambda_{3a,p}, \lambda_{5a,p})$
\end{enumerate}
In particular $G/Z$ is a Beauville group.
\end{lemma}

\begin{proof} Part (i) is a restatement of Lemma~\ref{OPlusetcAgain} (ii) in the special case $d=3$.

So consider (ii). Let $x $ have order $\lambda_{5a,p}$ be a Singer element of $\SU_5(q)$ and let $y\in \SU_3(q)\times \SU_3(q)$ be a regular semisimple element. Then as usual we can suppose that $xy$ is a bireflection of order $\lambda_{5a,p}$ as described.   Obviously $X = \langle x,y \rangle $ operates irreducibly on $V$. We now once again exploit \cite{GPPS}.   As usual if Example 2.1 of \cite{GPPS} holds, then we are done. As $X$ acts irreducibly, we have that Example 2.2 of \cite{GPPS} does not occur. As $\lambda_{5a,p}$ is  large Zsigmondy  we cannot have \cite[Example 2.3]{GPPS}.  Since $5$ and $6$ are coprime, \cite[Example 2.4]{GPPS} is impossible. \cite[Example 2.5]{GPPS} cannot occur as $6$ is not a power of $2$. In Example 2.6 of \cite{GPPS}, we note that the primitive prime divisor must be small  and so these cases do not come up. Moving on to \cite[Example 2.7]{GPPS} the only possibilities are that $X $ normalizes one of $2\udot \M_{12}$, $3 \udot \M_{22}$ or $2 \udot \J_2$. In the first and last case we have that $X$ can have no element of order $\lambda_{5a,p}$. Thus we have $X = 3 \udot \M_{22}$, $q= p=2$ and $G= \SU_6(2)$ which contradicts $q\neq 2$.  In \cite[Example 2.8]{GPPS} the only possibility is that $X $ normalizes $\G_2(q)$ but this also has no element of order $\lambda_{5a,p}$. From \cite[Example 2.9]{GPPS}, as $\lambda_{5a,p}$ is large, the only possibility arises in Table 8, with $F^*(X)/Z(F^*(X)) = \L_2(s)$ and $\lambda_{5a,p}=s$ and $6 = \frac 1 2 (s \pm 1)$ which means that $s=13$.  But then $13$ is not a possibility for $\lambda_{5a,p}$. This completes the proof of (ii).

It immediately follows from (i) and (ii) that $G/Z$ is a Beauville group.
\end{proof}

\begin{lemma}\label{U41} Let $q= p^{a/2}$ with $q>2$. Then $\SU_4(q)$ has a hyperbolic triple of type $(p,p,\lambda_{2a,p})$ if $p$ is odd and, if $p=2$, $G$ has a hyperbolic triple   of type $(4,2,\lambda_{2a,2})$.
\end{lemma}

\begin{proof} We let   $G=\SU_4(q)$ be the group of matrices which preserve the hermitian form with associated matrix $$J=\left(\begin{array}{cccc}
0&0&0&1\cr
0&0&1&0\cr
0&1&0&0\cr
1&0&0&0
\end{array}\right).$$
Let $F = \GF(q^2)$, $F_0= \GF(q)$ and $F_1 = \{\lambda \in F \mid \lambda^q+\lambda = 0\}$. Fix $e \in F_1$. Then $y =\left(\begin{array}{cccc}
1&0&0&e\cr
1&1&0&0\cr
0&0&1&0\cr
0&0&1&1
\end{array}\right)$ is an element of $G$. Now let $x=\left(\begin{array}{cccc}
1&0&0&0\\ 1&1&0&0
\\ b+c&{c}^{q}&1&0\\ -{b}^{q}-c+{c
}^{q}&-{b}^{q}-{c}^{q}+c&-1&1
\end{array}\right)$  where $b \in F$ and $c \in F_1$. Then $x \in G$. The characteristic polynomial of $xy$ is
$$w^4+(2ce+eb^q-4)w^3+(eb-eb^q+6-5ce)w^2+(2ce-eb-4)w+1.$$
A typical element $z$ of $G$ has characteristic polynomial $w^4+tw^3+fw^2+t^qw+1$ where $t \in F$ and $f \in F_0$. Our immediate aim is to show that we can select $c$ and $b$ to obtain any desired minimal polynomial. To do this we may take $c= -(f+t+t^q+2)/e$ and $b= -(3t^q-2t-2f-8)/e$.  Since we want $G= \langle x,y\rangle$, we further require that $x$ is not an involution when $p=2$.  So we require $c \neq 0$.  On the other hand, if $c=0$, we see that the minimal polynomial of $xy$ becomes
$$w^4+tw^3-(t+t^q-2)w^2+t^qw+1= (w-1)(w^3+(t+1)w^2-(t^q+1)w)-1)$$ and thus in this case we have that $z$ fixes a vector in the natural module for $G$. Now we select $z$ so that it has order $\lambda_{2a,p}$  with $z \in \SL_2(p^a)= \SL_2(q^2) \le \SU_4(q)$. Then $z$ has no non-zero fixed vectors on $V$ and hence its characteristic polynomial does not have $1$ as a root.  Now we select $b$ and $c$ so that $xy$ has the same  characteristic polynomial as $z$. Then $xy$ is conjugate to $z$ and $c \neq 0$. Set $X = \langle x,y \rangle$. Our choice of $x$ and $y$ implies that $X$ acts irreducibly on $V$. We intend to use the list of maximal subgroups of $G$  listed in \cite{BrayHoltRoney}. As $X$ is generated by $p$-elements,  $X= O^{p'}(X)$. Because $X$ acts irreducibly on $V$, $X$ is not contained in a parabolic subgroup or in a subgroup $\GU_3(q)$. Notice that as $z$ has order $\lambda_{2a,p}$, $X$ is not contained in a subfield subgroup of $G$ by Lemma~\ref{subfield}. Furthermore, $\lambda_{2a,p}$ does not divide the order of $\GU_2(q)\wr 2$.
Since $x$ and $y$ are not conjugate in $G$,  the Sylow $p$-subgroups of $X$ have order at least $p^2$; furthermore, if $p=3$, then $x$ has order $9$. It follows that $X$ is not in a maximal subgroup of type $(q+1)^3:\Sym(4)$, $(2^{1+4}\circ 4):\Sym(6)$, $(2^{1+5}\circ 4): \Alt(6)$, $(q+1,4)\udot\PSL_2(7)$, $(q+1,4)\udot\Alt(7)$, $(q+1,4)\udot\PSU_4(2)$  (here $p\neq 3$) or $4_2.\PSL_3(4)$. Finally, we address the possibility that $X$ is contained in the determinant one subgroup of  $\GL_2(q^2):2 \le \GU_4(q)$. Since the Sylow $p$-subgroup of $\GL_2(q^2)$ acts quadratically on  $V$  and $x$ does not, we infer that $x \not \in \GL_2(q^2)$. Hence we must have $p=2$. Now $y$ acts as a transvection on $V$, so we deduce that $y \in \GL_2(q^2)$. It follows that $z=xy $ is  contained in $\GL_2(q^2):2 $ but is not contained  in $\GL_2(q^2)$, which contradicts $z$ having odd order. Therefore this case cannot arise either. This completes the demonstration of the lemma.
\end{proof}

\begin{lemma}\label{U42}
Assume that $G= \SU_4(q)$ with $q=p^{a/2}\not\in\{2,3,5\}$. Then $G$ has a hyperbolic triple of type $(\lambda_{3a,p}, \lambda_{3a,p}, (q+1))$.
\end{lemma}

\begin{proof}
We let $y \in (\GU_3(q)\times \GU_1(q)) \cap G \cong \GU_3(q)$ be a Singer element. Then $x$ is a regular semisimple element of $G$. Thus we arrange for the product of $x$ with a conjugate $y$  of $x$ to be $z$ where $z \in \SU_2(q)\times \SU_2(q)$ is an element  of order $q-1$ which fixes no non-zero vectors and does not leave a non-degenerate one-space invariant (here we use $q>3$). Thus the group $X= \langle x,y\rangle$ is irreducible on $V$.

We show that $X$ is not contained in a maximal subgroup of $G$.
The maximal subgroups of $\SU_4(q)$ are listed in \cite{BrayHoltRoney}.
As $x$ has  order $\lambda_{3a,p}\ge 7$, we have  $X$ is not contained in any of the groups from Aschbacher classes $\mathcal C_1$ to $\mathcal C_5$. The  extraspecial group normalizers in $\mathcal C_6$ have largest odd prime divisor equal to $5$ and so these are also ruled out.  Thus $X$ must be a group in class $\mathcal S$. The only primes figuring in the order of the groups listed there are $2$, $3$, $5$, and $7$. In particular, we must have that $\lambda_{3a,p}$ is small. But then $q \in \{2,3,5\}$ which is impossible.  This proves the lemma.
\end{proof}

\begin{lemma}\label{U3} Suppose that $G = \SU_3(q)$ where $q = p^{a/2}> 2$. Then $G$ has a hyperbolic triple of type $(p,p,q^2-1)$ when $p$ is odd and of type $(4,2,q^2-1)$ when $q$ is even.
\end{lemma}

\begin{proof} We assume  $G$ preserves the form which has matrix $ {\left( \begin{array}{ccc}0&0&1\cr0&1&0\\1&0&0 \end{array}\right)}$. A typical element of a Sylow $p$-subgroup of $G$ is
$ x={\left( \begin{array}{ccc}1&0&0\cr a&1&0\\b&-a^q&1 \end{array}\right)}$ where, $a,b, \in \GF(q^2)$ and $b + b^q + a a^q=0$. Now note that $y= {\left( \begin{array}{ccc}1&0&e\cr0&1&0\\0&0&1 \end{array}\right)}$ where $e$ is non-zero and $e+ e^q =0$ is also in $G$.
We manipulate $a$ and $b$ so that  $xy$ is conjugate to $z={\left( \begin{array}{ccc}\lambda&0&0\cr0&\lambda^{q-1}&0\\0&0&\lambda^{-q} \end{array}\right)}$ where $\lambda$ is a generator of the multiplicative group of $\GF(q^2)$. For this we just have to equate the characteristic polynomials. We have that the minimal polynomial of $xy$ is $-w^3+(be+3)w^2-(3+aa^q e+be)w+1$ whereas $z$ has minimal polynomial $
{w}^{3}- \left(\lambda +{\lambda }^{q-1}+  {\lambda }^{-q} \right) {w}
^{2}+ \left( {\lambda }^{q}+ \lambda^{1-q} + \lambda^{-1} \right) w-1$. We now select $b = \frac{1}{e}(\lambda+\lambda^{q-1}+ \lambda^{-q}-3)$. Then it is easy to check that by substituting $-a a^q = b + b^q$ and using $e^q = -e$ we get the two polynomials are equal.  Furthermore, as the norm map is onto, given that $b+b^q$ is non-zero, there exists $a$ satisfying $-aa^q= b+ b^q$. Now
\begin{eqnarray*}b+b^q &=&  \frac{1}{e}\left((\lambda+\lambda^{q-1}+ \lambda^{-q}) -( \lambda^q+\lambda^{1-q}+\lambda^{-1})\right)\cr&=& \frac 1 e (1-\lambda)(\lambda^{-q}-1)(1-\lambda^{q-1}) \neq 0.\end{eqnarray*} Thus we may arrange that $a\neq 0$. Set $X= \langle x,y\rangle$. Then, as $a\neq 0$, $X$ acts irreducibly on $V$. Since $x$ and $y$ are not conjugate in $G$, the Sylow $p$-subgroup of $X$ has order at least $p^2$ and contains two $G$-classes of non-identity elements and since $z\in  X$, $X$ has a cyclic subgroup of order $q^2-1$. Lemma~\ref{subfield} implies that $X$ is not contained in a subfield subgroup of $G$ all the other possibilities either do not have a cyclic subgroup of order $q^2-1$ or do not have Sylow $p$-subgroups of the correct structure.
 It follows that $X=G$ as claimed.
\end{proof}

\begin{lemma} \label{U32}  The group $\SU_3(q)$ has a hyperbolic triple of type $(q^2-q+1,q^2-q+1,\lambda)$ where $\lambda $ divides $q^2-q+1$.
\end{lemma}

\begin{proof}
This is a reiteration of  Lemma~\ref{primedegree}.
\end{proof}

\begin{theorem}\label{Uthm} Suppose that $d \ge 3$, $G = \SU_d(q)$ and $Z \le Z(G)$. Then $G/Z$ is a Beauville group.
\end{theorem}

\begin{proof} For almost all the groups in question this follows from Lemmas~\ref{unitary}, \ref{SU5 and 7}, \ref{U6}, \ref{U41},  \ref{U42}, \ref{U3} and \ref{U32}. The groups remaining are $\SU_4(2) \cong \PSp_4(3)$, $\SU_4(3)$ and $\SU_4(5)$ and $\SU_6(2)$. The first group is dealt with as a symplectic group in Theorem~\ref{Spthm} and  $\SU_4(3)$ and $\SU_6(2)$ are addressed in Theorem~\ref{excepcovers} as they have  exceptional multipliers. The remaining  group have been examined by computer and we have two hyperbolic triples one with elements of order $13$ the other with elements of order $63$.
\end{proof}

\subsection{Symplectic groups}

Having addressed the unitary groups in detail, the case of the symplectic groups are straightforward.

\begin{lemma}\label{symplectc}
Suppose that $d \ge 3$ and $q=p^a$. Then $\Sp_{2d}(q)$ and $\PSp_{2d}(q) $ are a Beauville group.
\end{lemma}

\begin{proof} Let $G =\Sp_{2d}(q)$. By Lemma~\ref{OPlusetcAgain} , $G$ has a hyperbolic triple of type $(q^d-1,\lcm((q^{d-1}-1),(q+1)),q+1)$.

Let $x$ be an element of order $\lambda_{2da,p}$ and $z \in \Sp_{2d-2}(q)\times \Sp_2(q)$ be an element of order $\lambda_{(2d-2)a,p}$. Then by Gow's Theorem, we may arrange that two conjugates of $x$, $x_1, x_2$, have  product  $x_1x_2=z$. Set $H = \langle x_1,x_2\rangle$. Suppose that $H < G$. Theorem~\ref{maingeneration} (ii) implies that $q$ is even.

If Theorem~\ref{exotictwolarge} (i)  holds, then by considering the order of $\Sp_{d}(q^2)$, we obtain a contradiction. It follows that $H$ must be one of the groups  $\Omega_d^\epsilon(q)$ with $q$ a power of $2$. However the presence of $x_1$ indicates that $\epsilon = -$ while $z$ indicates that $\epsilon = +$.  This contradiction shows that $H = G$ in the situation.  So suppose that Theorem~\ref{exotictwolarge} (ii) holds. Then $q=2$ and $F^*(H) \cong \Alt(13)$ or $\Alt(14)$.
Hence so long as $q\neq 2$ and $d \neq 12$, we have that  $G$ has a hyperbolic triple of type $(\lambda_{2da,p},\lambda_{2da,p},\lambda_{(2d-2)a,p})$. We have checked by computer that when $d= 12$ and $q=2$, $G$ also has a hyperbolic triple with the same designated type.
\end{proof}

\begin{lemma}\label{SP41}
Assume that $G= \Sp_4(q)$ with $q>7$. Then $G$ has a hyperbolic triple of type $$((q^2-1)/\gcd(2,q-1), (q^2-1)/\gcd(2,q-1)2,q+1).$$
\end{lemma}

\begin{proof}
We let $x$ be an element of order $(q^2-1)/\gcd(2,q-1)2$ containing in $\Sp_2(q)\times \Sp_2(q)\le \Sp_4(q)$ in such a way that $x$ acts on the first block as an element of order $(q-1)/2$ and on the second block as an element of order $(q+1)/2$. Notice that on at least one of these blocks, $x$ acts as an element of odd order. Furthermore, as $q> 5$, we have that $x$ is a regular semisimple element of $G$ with centralizer of order $q^2-1$. Now select an element $y \in \GL_2(q):2 \le \Sp_4(q)$ again of order $(q^2-1)/\gcd(2,q-1)$. Note that $y$, unlike $x$,  does not fix any one-dimensional subspaces. Thus $x$ and $y$ are not $G$-conjugate. Further we have $y$ is a regular semisimple element. Now by Gow's Theorem we may assume that $x$ and $y$ are chosen such that $xy$ has order $q+1$  and fixes a two-space vector wise.  Now let $X = \langle x,y\rangle$. Since $y$ fixes exactly two non-trivial subspaces each of which is isotropic and since $x$ fixes no two dimensional isotropic subspaces we infer that $X$ acts irreducibly on $V$.  Note that if we choose an odd prime $r$ divisor of $q\pm 1$  then $x$ contributes an element of order $r$ which fixes a non-zero vector where as $y$ contributes and element of order $r$ which fixes no non-zero vectors. It follows that $X$ contains an elementary abelian subgroup of order $r^2$.   Now using the maximal subgroups as presented in \cite{Mitchell} and \cite{BrayHoltRoney} we get that $X$ is contained in a conjugate of $\Sp_2(q)\wr 2$, $\GL_2(q):2$, $\GU_2(q):2$ ($p$ odd) or $\mathrm{SO}_4^+(q)$ ($p=2$) (we used that $(q^2-1)/2\ge 40$ to rule out some of the possibilities). In all the cases we see that the candidates have a unique conjugacy class of elements of order $(q^2-1)/\gcd(2,q-1)$.
\end{proof}

\begin{lemma}\label{oneclass} The group $G=\mathrm O_4^-(2^m)$ has exactly one conjugacy class of
 elements of order $4$.
\end{lemma}

\begin{proof} Let $S \in \Syl_2(G)$ and let $\sigma \in S$ be a reflection. Then $C_G(\sigma) \cong \Sp_2(2^m) \cong \mathrm O_3(2^m)$. Let $S_0 = S \cap \Omega_4^-(2^m)$. Then $S_0$ has order $q^2$ and we have $N_G(S)$ has order $2q^2(q-1)$. Now we calculate in this group that there is exactly one conjugacy class of elements of order $4$.
\end{proof}

\begin{lemma}\label{Sp42} The group $\Sp_4(q)$ has a hyperbolic triple of type $(p,p,(q^2+1)/2)$ if $p \not \in \{2,3\}$,   of type $(9,3,(q^2+1)/2)$ if $p=3$ and  of type $(4,4,q^2+1)$ if $p=2$.
\end{lemma}

\begin{proof} As usual for these small cases we give explicit matrices. We consider $G=\Sp_4(q)$ as the set of matrices which preserve the form with matrix $$J=\left(\begin{array}{cccc}
0&0&0&1\cr
0&0&1&0\cr
0&-1&0&0\cr
-1&0&0&0
\end{array}\right).$$ We let $\{e_1,e_2,f_2,f_1\}$ be the basis of  $V$ corresponding to $J$.
Then, for $a$ and $b$ in $\GF(q)$, the matrix $x=\left(\begin{array}{cccc}
1&0&0&0\cr
1&1&0&0\cr
0&a&1&0\cr
b&-a&1&1
\end{array}\right)$ is an element of $G$.

We first consider the situation with $p$ odd. Then $x$ has order $p$ if $p>3$ and has order $9$ if $p=3$ and $a \neq 0$.
Let $y=\left(\begin{array}{cccc}
1&0&0&1\cr
0&1&0&0\cr
0&0&1&0\cr
0&0&0&1
\end{array}\right)$ so that $y\in G$ is a transvection. Now the characteristic polynomial (in $w$) of $xy$ is $w^4-(4+a)w^3+(6+b+2a)w^2+(-4-a)w +1$.

Let $z$ be an element of $G$ of order $(q^2+1)/(2,q-1)$. Then, as $q> 3$, $z$ acts irreducibly on $V$ and its characteristic polynomial is of the form $w^4-tw^3+fw^2-tw+1$ where $t$ is the trace of $z$. Plainly we may select $a$ and $b$ so that $xy$ has characteristic polynomial  $w^4-tw^3+fw^2-tw+1$. (This can be done for $p=2$ as well.) Now setting $X = \langle x,y\rangle$, we see that $X$ contains an element of order $(q^2+1)/2$ and a transvection. The only maximal subgroups of $\Sp_4(q)$ for $q$ odd which contain transvections are $\Sp_2(q)\wr 2$ and $\Sp_4(q_0)$ where $\GF(q_0)$ is a subfield of $\GF(q)$. Plainly these  groups do not contain an element of order $(q^2+1)/2$. (In the case when $q$ is even, we can have $X = \Omega_4^-(q)$. In $G=\Sp_4(8)$, we have checked computationally that $G$ is not generated by $x$ and $y$ for any choice of $a$ and $b$.)

The plan for $q= 2^m$ is similar. We take  $y=\left(\begin{array}{cccc}
1&1&1&\lambda\cr
0&1&1&0\cr
0&0&1&1\cr
0&0&0&1
\end{array}\right)$ where $\lambda \in F \setminus \{\beta^2+\beta \mid \beta \in \GF(2^m)\}$. Then $y$ has order $4$.
Now, for $a$ and $b$ non-zero elements of $\GF(2^m)$, we let  $x=\left(\begin{array}{cccc}
1&0&0&0\cr
a&1&0&0\cr
0&b&1&0\cr
0&ab&a&1
\end{array}\right)$. Then $x$ also has order $4$. We have $C_V(x^2) = \langle e_1,e_2\rangle$ and $C_V(y^2) = \langle f_1,f_2\rangle$. Now note that $(f_1x,f_1)=0$ whereas if $(vy,v)=0$ we calculate that our choice of $\lambda$ implies that $v \in C_V(y^2)$. Hence $x$ and $y$ are not conjugate in $G$.

The characteristic polynomial of $xy$ is
\begin{eqnarray*}
&&w^4+bw^3+a^2(b\lambda+1)w^2+bw+1
\end{eqnarray*}
Again  the characteristic polynomial of an element of order $q^2+1$ has the form  $w^4+tw^3+fw^2+tw+1$ and is irreducible.  In particular, as elements in $\GF(q)$ have unique square roots (as the field has characteristic 2), we have that $t \neq 0$ and $f \neq 0$ as $1$ is not a root of the polynomial. Setting  $b= t$, we must choose $a$ so that  $a^2(b\lambda +1)= f$. This has a solution so long as $b\lambda \neq 1$. Since we have at least two choices for $\lambda$ this can be arranged.
Thus we may choose $a, b$ and $\lambda$ so that $xy$ has order $q^2+1$. Now set $X = \langle x,y\rangle$. If $X$ is a proper subgroup of $G$, then, as $X$ has elements of order $4$,  $X \cong \SL_2(q^2):2$ or  $\Omega_4^-(q):2$ by \cite{AB}. However these groups are isomorphic (by an outer automorphism of $G$) and by Lemma~\ref{oneclass} they have a unique conjugacy class of elements of order $4$. As $x$ and $y$ are not conjugate we deduce that $G=X$ and this concludes the proof.

\end{proof}
Finally we check computationally that $\Sp_4(q)$ with $q\le 7$ has hyperbolic triples as listed in Table~\ref{Sptab}.

\begin{table}[h] \label{Sptab}
\caption{Element Orders for Beauville Systems in small symplectic groups}
\vskip 0.5cm
\begin{center}
\renewcommand{\arraystretch}{1.24}
\begin{tabular}{|c|c|c|c|c|c|c|}
\hline
$G$ & $x_1$ & $x_2$ & $x_3$ & $ y_1$ & $ y_2 $ & $ y_3 $ \\
\hline
$\Sp_4(3)$ & 5 & 5 & 5 & 9 & 9 & 9 \\
\hline
$\Sp_4(5)$ & 8 & 8 & 8 & 13 & 13 & 13 \\
\hline
$\Sp_4(7)$ & 8 & 8 & 8 & 25 & 25 & 25 \\
\hline
\end{tabular}
\end{center}
\end{table}

We summarize the results of this subsection with the main result.
\begin{theorem}\label{Spthm} Suppose that $G = \Sp_d(q)$ with $d \ge 4$ and $Z\le Z(G)$. Then $G/Z$ is a Beauville group.\qed
\end{theorem}

\subsection{Spin groups}

The simply connected groups of Lie type $\B_n(q)$, $\mathrm D_n(q)$ and ${}^2\mathrm D_n(q)$ are the Spin groups. They have centre of order a power of $2$ and map onto the orthogonal groups $\Omega_{2n+1}(q)$,  $\Omega_{2n}^+(q)$, $\Omega_{2n}^-(q)$ respectively.  We often work in these images of the simply connected groups so that we can describe the elements in various hyperbolic triples via their action on the associated natural orthogonal space $V$. For this purpose whenever $G= \Spin^\varepsilon(q)$, $\ov G$ will denote the image $\Omega_d^\varepsilon (q)$. Note that as any maximal torus contains $Z(G)$ any preimage of a regular semisimple element of $\Omega_d^\varepsilon(q)$ in $\Spin_d^\varepsilon (q)$ is also regular semisimple.  Now suppose that $\ov x$ and $\ov y$ are regular semisimple elements of odd order in $\ov G$ and that $\ov z$ is a semisimple element of odd order in $\ov G$. Then there exist preimages  $x$, $y$ and $z$ of $\ov x$, $\ov y$ and $\ov z$ respectively which also have odd order. Furthermore, $x$ and $y$ are regular semisimple elements of $G$. Hence by Gow's Theorem~\ref{G}  we can conjugate $x$ and $y$ and suppose that $xy=z$ and consequently also $\ov x \ov y =\ov z$. In the arguments in this section, we will always argue in the orthogonal groups with the understanding that our conjugation happens in $G$ as just described when our elements have odd order. By abuse of notation, we will not use the ``bar" notation rather we will just describe the way  that elements of $G$ act on the orthogonal space.

\begin{lemma}\label{orthogonal+}
Suppose that $d \ge 10$ is even, $q=p^a$, $G= \Spin^+_d(q)$ and $Z \le Z(G)$.  Then $G/Z$ is a Beauville group.
\end{lemma}

\begin{proof} Let $G= \Spin_d^+(q)$ and $V$ be the orthogonal space for $G$ (so there is a kernel to the representation). The proof is similar to that of Lemma~\ref{unitary} but we have to take special care about the orthogonal type of the various subspaces involved.

Suppose that $d/2< k< k+2\le d-2$ with $k$ even. We let $x_1 \in \Omega_{d-4}^-(q)$ have order $\lambda_{(d-4)a,p}$ and $z \in  \Omega_{d-2}^-(q)$ of order $\lambda_{(d-2)a,p}$.  Now pick $x_2 \in \Omega_{4}^-(q)$ of order $\zeta_{4a,p}$ and set $x= x_1x_2 \in \Omega_{d-4}^-(q)\times \Omega^-_{4}(q) \le \Omega_d^+(q)$. Then in this case $x$ is regular semisimple in $G$ and by Gow's Theorem we can find a conjugate $y$ of $x$ such that $xy$ is conjugate to $z$. In addition we have that $H=\langle x,z\rangle$ acts irreducibly on $V$. Since $H$ contains elements of order $\lambda_{(d-4)a,p}$ and $\lambda_{(d-2)a,p}$, we may apply Theorem~\ref{maingeneration} (iv) to get that $H= \Omega^+_{d}(q)$ so long as $(d,q) \not = (14,2)$. Thus $G$ has a hyperbolic triple of type  $(\lambda_{(d-4)a,p}\zeta_{4a,p},\lambda_{(d-4)a,p}\zeta_{4a,p},\lambda_{(d-2)a,p})$ when $(d,q) \not = (14,2)$. In this exceptional case we make an alternative hyperbolic  triple by selecting $x$ of order $\lambda_{8,2}\zeta_{6,2}= 153 $ projecting in $\Omega_8^-(2)\times \Omega_{6}^-(2)$ and choose $y$ a conjugate of $x$ such that $xy$ is conjugate to $z$. Then $\langle x,y \rangle$ is irreducible on $V$ and we once again apply Theorem~\ref{maingeneration} (iv) to get that $G = \langle x,y\rangle$. Hence $\Omega_{14}^+(2)$ has a hyperbolic triple of type $(153,153,13)$.
 On the other hand, by Lemma~\ref{OPlusetc}, $G$ has a hyperbolic which projects to  a hyperbolic triple of type  $((q^{d/2}-1)/\gcd (q-1,2),(q^{d/2}-1)/\gcd (q-1,2), (q+1)/\gcd (q-1,2))$ in $\Omega_d^+(q)$. This shows that $G$ and the quotients $G/Z$ are Beauville groups.
\end{proof}

\begin{lemma}\label{O8}
Suppose that $H = \Spin_8^+(q)$ with $q=p^a\not\in \{2,3\}$. Then
\begin{enumerate}
\item $G$ has a hyperbolic triple  $(x,y,z)$  of type $((q^4-1)/\gcd(q-1,2),(q^4-1)/\gcd(q-1,2),\lambda_{6a,p})$ where the element $z$ centralizes a 2-dimensional $-$-space and  $x$ and $y$ are conjugate and act on a decomposition of $V$ in to a sum of  two totally singular subspaces preserved by $\GL_4(q):2$.
    \item  $G$ has a hyperbolic triple  $(x,y,z)$  of type $((q^2+1)/\gcd (q-1,2),(q^2+1)/\gcd(q-1,2),(q^3-1)/\gcd(q-1,2))$ where $x$ and $y$ are conjugate regular semisimple elements with centralizers of order $(q^2+1)^2/\gcd(q-1,2)$ and $z$ is an element of $\Omega_6^+(q)$ which preserves a decomposition of $V$ in to a sum of two singular $3$-spaces and a two space which it centralizes.
\end{enumerate}
In particular, if $H =\Spin^+_8(q)$ with $q \not \in\{2,3\}$ and $Z \le Z(H)$, $H/Z$ is a Beauville group.
\end{lemma}

\begin{proof} Let $H= \Spin_8^+(q)$ and  $G = \Omega_8^+(q)$ with $q\not\in\{2,3\}$ and $V$ be its natural orthogonal module. Again we identify elements of $G$ with their preimages of smallest order in $H$ and apply Gow's Theorem in $H$.  We first prove (i). Let $x$ be an element of the stabilizer of an opposite pair of singular subspaces $W_1$ and $W_2$ which acts as an element of order $(q^4-1)/\gcd(q-1,2)$ generating a subgroup of the Singer cycle subgroup on both subspaces. Then $x$ acts as the field element (of $\GF(q^4)$) $\lambda$ on $W$ and $\lambda^{-1}$ on $W_2$. In particular, as $x$ has order $(q^4-1)/\gcd(q-1,2)$, we have that $W_1$ and $W_2$ are the only subspaces of $V$ left invariant by $x$. Hence  $x$ is a regular semisimple element of $G$. Now let $z$ be an element of $\Omega_6^-(q) \times \Omega_2^-(q)$ acting just in the first factor as an element of order $(q^3+1)/\gcd(q-1,2)$ and centralizing the $2$-dimensional $-$-space. Then Gow's Theorem implies that we may find a conjugate $y$  of $x$ and assume that $xy=z$. Set $X = \langle x,y\rangle$. Note that $z$ does not leave any $4$-dimensional subspace invariant. Hence $X$ acts irreducibly on $V$.

 We use \cite{GPPS} again. Lemma~\ref{subfield} and the fact that $X$ contains   $z$  implies that $X$ is not contained in the groups listed in \cite[Examples 2.1, 2.2]{GPPS}.   The possibilities from \cite[Examples 2.3 and 2.5]{GPPS} do  not contain  an element of order $(q^4-1)/\gcd(q-1,2)$.
 The extension field subgroups $\mathrm O_4^\varepsilon (q^2)$ and $\GU_4(q)$ as the first do not contain a conjugate of $z$ and the second has no cyclic subgroup of order $(q^4-1)/\gcd(q-1,2)$ and thus possibility  \cite[Example 2.4]{GPPS} is  eliminated.  Since $q>3$ we have $(q^4-1)/\gcd(q-1,2) \ge 255$. Thus $x$ has high order and we see that $X$ is not a cover of $\Sym(10)$, $\Sym(9)$, $\Sym(8)$, $\Sym(7)$, $\Sp_6(2)$, $\Omega_8^+(2)$, ${}^2\B _2(8)$ or $\PSL_3(4)$.  This eliminates \cite[Example 2.6]{GPPS} and \cite[Table 7]{GPPS}. There are no sporadic examples in dimension 8 and so \cite[Example 2.7] {GPPS} has no groups for us.  The groups listed in   \cite[Table~8]{GPPS} are easily seen to be impossible.   This leaves $\SL_3^\varepsilon(q_0^3)$, $\Sp_6(q_0)$ ($p=2$) or $2\udot \Omega_7(q_0)$ ($p$ odd) on their spin modules   to consider from  \cite[Table 8]{GPPS}. However these groups  do not have cyclic subgroups of order $(q^4-1)/\gcd(q-1,2)$. This proves (i).

Now for (ii) we let $x$ be an element of the subgroup of $G$ fixing two perpendicular $-$-type spaces of dimension $4$, say $W_1$ and $W_2$ in such a way that  as $\langle x\rangle$-modules $W_1$ and $W_2$ are not isomorphic. Thus $x$ has order  $(q^2+1)/\gcd(q-1,2),$ and $\langle x\rangle$ is a diagonal subgroup in $(\Omega_4^-(q)\times \Omega_4^-(q)).2^2$.  Since $W_1$ and $W_2$ are non-isomorphic as $\langle x\rangle$-modules, we get that $x$ is a regular semisimple element of $G$  and that its centralizer has order $(q^2+1)^2/\gcd(q-1,2)$. Now select $z\in \Omega_6^+(q) \times \Omega_2^+(q)$ to project trivially on to the second factor and to have order $(q^3-1)/\gcd(q-1,2)$ leaving invariant two totally singular $3$-spaces. Thus $X$ acts irreducibly on $V$.

 We now survey the maximal subgroups of $\Omega_8^+(q)$ as described in \cite{KL08}. Before doing this, we observe that if $\alpha$ is an automorphism of $G$, then $\alpha$ maps regular semisimple elements of $G$ to regular semisimple elements of $G$. We refer explicitly to the 75 rows in Table 1 of \cite{KL08} where the maximal subgroups of $\Omega_8^+(q)$ are described. In particular, we note that in that table, the rows labeled with ``none" in column IV are not maximal subgroups of $\Omega_8^+(q)$. First because $X$ acts irreducibly on $V$, we have that $X$ is not contained in a parabolic subgroup of $G$. Hence $X$ is not contained in any of the groups listed in the first $8$ rows of \cite[Table 1]{KL08}. Now, in the subgroup $\Omega_7(q)$ of $G$ fixing a one space of $V$, we note that the elements of order $\lambda_{4a,p}$ commute with a subgroup $\Omega_3(q)$ and so are not regular semisimple. It follows now that the elements of all of the maximal subgroups of $G$ isomorphic to $\Omega_7(q)$ or $2 \udot \Omega_7(q)$ do not contain an element of order $\lambda_{pa,4}$ which is regular semisimple. Hence $X$ is not contained in such a subgroup. This eliminates the subgroups listed on rows 9 through 18 of  \cite[Table 1]{KL08}  as rows 15 to 18 are labeled none. Now, if $X \le \GL_4(q)$, we see that $x$ is centralized by a subgroup of order $(q^4-1)/\gcd(q-1,2)$ not $(q^2+1)^2/\gcd(q-1,2)$. This removes the maximal subgroups listed on rows 19 to 21 of \cite[Table 1]{KL08} as possible over groups of $X$. We do not need to consider rows $22$ or $26$. The argument adopted for the elimination of $\Omega_7(q)$ applies equally well to the groups listed in rows 23 to 25 and 27 to 32.  Rows 33 to 38 do not list  maximal subgroups of $G$ and 39 to 50 refer to conjugates of $2^6:\Sym(8)$. They require $q=p \equiv \pm 1 \pmod 8$ to be odd. However, as $\lambda_{4,p}$ is large, we must have $\lambda_{4,p}= 7$ which is absurd as $\lambda_{4,p} \equiv 1 \pmod 4$.
 The  rows 51 to 54 do not list maximal subgroups of $G$. Rows 55, 56 and 57 are eliminated as $\lambda_{4a,p}$  does not divide the order of the listed groups.  The subgroups of shape  $(\Omega_4^-(q) \times \Omega_4^-(q)).2^2$ do not have order divisible by $(q^3-1)/\gcd(q-1,2)$ and so rows 58, 59 and 60 are eliminated. Similarly  row $61$ is ruled out as a possible over group of $X$.

  The subfield type subgroups are eliminated as they have no cyclic subgroup of order $(q^4-1)/\gcd(q-1,2)$. This leaves the subgroups in rows 70 to 75 to consider. Rows 70 and 71 are ruled out as $\lambda_{4a,p}$ does not divide the order of the groups there. Finally, as $\lambda_{4a,p}$ is large, we cannot have $X$ contained in $\Omega_8^+(2)$, ${}^2\B_2(8)$, $\Alt(9)$ or $\Alt(10)$.  We conclude that $X = G$. Thus (ii) holds.

  Finally let $H= \Spin_8^+(q)$ and $Z \le Z(H)$.  If $p=2$, then as $q>2$ we have $Z(H)=1$. Hence $q$ is odd and $Z(H)$ is elementary abelian of order $4$ (see \cite[3.9.2]{WilsonBook}).  If $|Z|\ge 2$, then either $H/Z \cong \mathrm  P\Omega_8^+(q)$ or $H/Z \cong \Omega_8^+(q)$ as the triality automorphism acts transitively on the subgroups of $Z(H)$ of order $2$. In these cases (i) and (ii) show that $H/Z$ is a Beauville group.    Thus we just need to show that $H$ has a Beauville group.

  Let the non-identity elements of  $Z(H)$ be $v_1$, $v_2$ and $v_3$ and identify $G= H /\langle v_1\rangle$.
  Let $(\ov x,\ov y,\ov z)$ be the hyperbolic triple for $G$  identified in (i). Then $\ov z$ has odd order and $\ov x$ and $\ov y$ have even order, $k$ say, with $x^{k/2}\in Z(G)$.  As $Z(H)$ is elementary abelian, it follows that $x$, $y$ and $z$ have the same orders as $\ov x$, $\ov y$ and $\ov z$ and of course $xy=z$ as all along we were working in $H$.  In particular, the only non-trivial central elements of $H$ seen by the triple $(x,y,z)$  are in $\{v_2,v_3\}$.

Let $(\ov x,\ov y,\ov z)$ be the hyperbolic triple for $G$  from (ii). Then $x$ and $y$ have odd order and $\ov z$ may have either even or odd order. However no power of $\ov z$ is contained in $Z(G)$. It follows that if the triple $(x,y,z)$ has an element which generates a cyclic group which intersects $Z(H)$ non-trivially, then it is $z$ and the non-trivial element is $v_1$.  It follows that $H$ is a Beauville group.
\end{proof}

We include one of the two exception cases from Lemma~\ref{O8}.
\begin{lemma} \label{spin83} Both $\Spin_8^+(3)$ and $\Omega_8^+(3)$ are Beauville groups.
\end{lemma}

\begin{proof}  Matrices representing $\Spin_8^+(3)$ can be obtained from \cite{onlineatlas}. We have calculated by computer that  $\Spin_8^+(3)$  has hyperbolic triples of type $(7,7,7)$ and $(13,13,13)$.\end{proof}

Note that $\Omega_8^+(2)$ has an exceptional Schur cover and will therefore be dealt with in Theorem~\ref{excepcovers}.

\begin{lemma}\label{orthogonal-}
Suppose that  $d \ge 8$. Then $G=\Spin_d^-(q)$. Then the following hold.

\begin{enumerate}
\item  For $q \not \in \{2,3,5\}$, $G$ has a hyperbolic triple of type $$((q^{d/2}+1)/\gcd(q-1,2),(q^{d/2}+1)/\gcd(q-1,2),(q-1)/\gcd(q-1,2)).$$
\item For $q>3$ and $d \neq 8$, $G$ has a hyperbolic triple of type $$(\lambda_{(d-4)a,p} (q^2-1)/\gcd(q-1,2),\lambda_{(d-4)a,p} (q^2-1)/\gcd(q-1,2),\lambda_{(d-2)a,p}).$$
\item For $q \in \{2,3,5\}$ and $d \neq\{ 8, 12\}$, $G$ has  a hyperbolic triple of type $$((q^{d/2}+1)/\gcd(q-1,2),(q^{d/2}+1)/\gcd(q-1,2), (q^k+1)/\gcd(2,q-1))$$ where $k = d/2-1$ if $d/2$ is odd and $d/2-2$ if $d/2$ is even.
    \item  For $d \ge 14$, $G$ has a hyperbolic triple of type $$(\lambda_{(d-6)a,p}\lambda_{3a,p},\lambda_{(d-6)a,p}\lambda_{3a,p},\lambda_{(d-4)a,p}).$$
    \item For $d=8$ and $q\not \in \{2,3,5\}$, $G$ has a hyperbolic triple of type $(\lambda_{4a,p}(q^2-1)/\gcd(q-1,2),\lambda_{4a,p}(q^2-1)/\gcd(q-1,2), \lambda_{6a,p})$.
\end{enumerate}
In particular, $G$  and $G/Z(G)$ are Beauville groups so long as  $$(d,q) \not\in \{(8,2),(8,3),(8,5),(10,2),(10,3), (10,5), (12,2),(12,3),(12,5)\}.$$

\end{lemma}

\begin{proof} We let $x$ be an element of $G$ which projects in $\Omega_{d}^-(q)$ to an element  of order $(q^{d/2}+1)/\gcd(q-1,2)$. Then $x$ is a Singer element in $G$. Let $z$ be an arbitrary semisimple element of $G$. Then, as $x$ is a regular semisimple element of $G$, there is a $y \in x^G$ such that $xy$ is conjugate to $z$. Let $X= \langle x,y\rangle$.  Then,  by \cite{AB}, if $X \neq G$, then $X$ normalizes an extension field subgroup. Thus $X$ contains one of $\Omega_m^\eta(q^r)$ where $rm=n$ or $\GU_{d/2}(q)$ where $d/2$ is odd. We now specify $z$ more precisely. If $q\not\in\{2,3,5\}$, we let $z$ be a bi-reflection in $\Omega_2^+(q)$ of order $(q-1)/\gcd(q-1,2)$. Since $q> 3$, the eigenvalues of $z$ on $V$ are distinct. Therefore, in this case $X= G$. Suppose that $q \in \{2,3,5\}$. Then provided $d\ge 10$ and $d \neq 12$, we can select an even number $k < d/2$ such that $k$ does not divide $d$ (take $k=d/2-1$ when $d/2$ is odd or $k=d/2-2$ if  $d/2$ is even). Now we let $z$ be an element of order $(q^{k/2}+1) /\gcd(2,q-1)$. This element is obviously not in an extension field subgroup and so $X= G$ in these cases too. Thus (i) and (iii) hold.

Assume that $d>8$. Let $x=x_1x_2$ be such that $x_1$ is an element of order $\lambda_{(d-4)a,p}$ projecting in $\Omega_{d-4}^-(q)$ and $x_2$ is of order $(q^2-1)/\gcd(q-1,2)$ when projected into $\Omega_{4}^+(q)$ and lies in $\GL_{2}(q)$ (notice that this element has the same order in $G$ as it does in $\Omega_4^+(q)$). Let $z$ be an element of order $\lambda_{(d-2)a,p}$ projecting into $\Omega_{d-2}^-(q)$. Then, assuming that $q > 3$, we have that $x$ is a regular semisimple element. Applying Gow's Theorem,  we can find a conjugate $y$ of $x$ such that $xy$ is conjugate to $z$. So we suppose that $x, y$ and $z$ have been chosen in this way. Then $x$ is a regular semisimple element and $H = \langle x,y \rangle$ acts irreducibly on $V$. Now we use  Theorem~\ref{maingeneration} (iv) to get that $H= \Spin^-_{d}(q)$. Thus, so long as $q>3$, $(x,y,z)$ is a hyperbolic triple of type $(\lambda_{(d-4)a,p} (q^2-1)/\gcd(q-1,2),\lambda_{(d-4)a,p} (q^2-1)/\gcd(q-1,2),\lambda_{(d-2)a,p})$. This is (ii). Now suppose that  $d \ge 14$. We select elements $x$ of order $\lambda_{(d-6)a,p}\lambda_{3a,p}$ and $z$ of $\lambda_{(d-4)a,p}$ and again find  a conjugate $y$ of $x$ such that $xy$ is conjugate to $z$. Now  Theorem~\ref{maingeneration} (iv) implies that $H=G$. Thus we have a hyperbolic triple of type $(\lambda_{(d-6)a,p}\lambda_{3a,p},\lambda_{(d-6)a,p}\lambda_{3a,p},\lambda_{(d-4)a,p})$ under the specified conditions. This proves (iv).

Now suppose that $d=8$ and that either $q\not \in\{2,3,5\}$. Select $x$  to project to an element of $\Omega_4^-(q)\times \Omega_4^+(q)$ of order $\lambda_{4a,p}$ on the first factor and of order $(q^2-1)/\gcd(q-1,2)$ in the second factor. Then $x$ is a regular semisimple elements of $G$. Let $z$ have order $\lambda_{6a,p}$. As usual we may suppose that $xy=z$ where $y$ is some conjugate of $x$. Let $X=\langle x,y\rangle$ and identify $X$ with its image in $\Omega_8^-(q)$. Then $X$ acts irreducibly on $V$ and $X$ and as $q\neq 3$, Lemma~\ref{imprimitive} shows $X$ is primitive. Again we use \cite{GPPS}. Lemma~\ref{subfield} implies that $F^*(X)$ is not a proper subfield subgroup. If $F^*(X)$ is an extension field subgroup, then we must have $F^*(X) \le \Omega_4^-(q^2)\cong \PSL_2(q^2)$ but the presence of $z\in X$ rules this possibility out.  Using  $\lambda_{6a,p}$ is large and $d=8$, we also eliminate the possibility that $F^*(X)$ is a cover of an alternating
group or a sporadic simple group. From Example 2.8 of \cite{GPPS}  we see that $F^*(X)$ could be one of $\SL_2(q^3)$, $2\udot \Omega_7(q)$ with $q$ odd or $\Sp_6(q)$ with $q$ even, or $\PSU_3(q^{\frac 1 2})$ with $p \neq 3$. However the first three of these groups appear as subgroups of $\Omega_8^+(q)$ \cite{KL08} and so cannot be contained in $\Omega_8^-(q)$ while that last group has no elements compatible with $x$. Thus this case cannot occur. Finally we have to consider Example 2.9 from \cite{GPPS} from which we get $F^*(X) = \SL_2(17)$ but in this last case the large prime divisor is comes from $q^8-1$ and not $q^6-1$ and so this case is also impossible. Thus we infer that $G= X$. Hence (v) holds.

Suppose now that $d \ge 14$ and $q\not\in\{2,3,5\}$. Then the triples in (i) and (iv) show that $G$ is a Beauville group. Suppose that $d \ge 14$ and $q\in \{2,3,5\}$. Then (iii) and (iv) show that $G$ is a Beauville group. Suppose that $d \neq 8$. Then as the elements of  the triple from (ii) has the same order in $G$ as they do in $\Omega_{d}^-(q)$, the triples in (i) and (ii) satisfy the non-conjugacy condition. Thus, if $G$ is not a Beauville group and $d \neq 8$, then  $d= 10$ or $12$ with $q\in \{2,3,5\}$. If $d=8$ and $q\not\in\{2,3,5\}$, then (i) and (v) show that $G$ is a Beauville group.
\end{proof}

We now consider the exceptional cases in Lemma~\ref{orthogonal-}.

\begin{lemma}\label{orth-extra} The following hold.
\begin{enumerate}
\item$\Spin_8^-(5)$ has a hyperbolic triple of type $(313,313,3)$.
\item$\Spin_{10}^-(3)$ has a hyperbolic triple of type $(61,61,41)$.
\item$\Spin_{10}^-(5)$ has a hyperbolic triple of type $(1563,1563,13)$.
\item$\Spin_{12}^-(3)$ has a hyperbolic triple of type $(73,73,61)$.
\item $\Spin_{12}^-(5)$  has a hyperbolic triple of type $(7813,7813,3)$.
\end{enumerate}

\end{lemma}
\begin{proof}
Suppose that $G= \Spin_{8}^-(5)$ and let $x \in G$ have order  $(5^4+1)/2=313$. Select $z$ of order $3$ to project into $\Omega_2^-(5)$. Then $x$ is regular semisimple and so we may suppose that we have a triple $(x,y,z)$ with $xy=z$ and such that $y$ is conjugate to $x$. Now by \cite{AB}, we have that $X=\langle x,y \rangle$ normalizes an extension field subgroup. Assume that $X < G$. Since $z \in X$, we infer that $z$ centralizes a hyperplane when $V$ is considered over that extension field. As $d/2$ is even, we deduce that $X$ cannot be a unitary group. Thus we must have $x$ of order $2$. Hence $G= X$. This proves (i).  The same argument works to give (v).  Suppose that $G = \Spin_{10}^-(3)$. Then $\lambda_{10,3}= 61$. We let $x$ and $y$ be elements of order $61$ and arrange for their  product $z$ to have order $\lambda_{8,3}= 41$. Then, using Theorem~\ref{maingeneration} (iv) we have that $G=\langle x,y\rangle$. So (ii) holds.
The example in  (iii) is just a reiteration of the result in Lemma~\ref{orthogonal-}(iii) for the case $d=10$, $q=5$. So assume that $G= \Spin_{12}^-(3)$. This time we take our hyperbolic triple to be be other the form $(\lambda_{12,3}, \lambda_{12,3}, \lambda_{10,3}) = (73,73,61)$.
\end{proof}

\begin{lemma} The groups $\Spin_d^-(q)$  with $$(d,q)\in \{(8,2),(8,3),(8,5),(10,2),(10,3), (10,5), (12,2),(12,3),(12,5)\}$$ are Beauville groups.
\end{lemma}

\begin{proof} If $q$ is a power of $2$, then $\Spin_{d}^-(q) = \Omega_d^-(q)$ is simple so these groups are easily accessible by computer. The group $\Spin_8^-(3)$ is available as a matrix group.
from \cite{onlineatlas}. For the other groups under consideration we have provided a hyperbolic triple of odd order in Lemma~\ref{orth-extra} and so we may calculate in $\Omega_{d}^-(q)$ and lift the generators $x$ and $y$ without losing the conjugacy condition. Thus the remaining groups can be checked by computer. The hyperbolic triples discovered are as follows.

\begin{table}[h] \label{Sptab}
\caption{Hyperbolic triples in certain small orthogonal and spin groups of minus type}
\vskip 0.5cm
\begin{center}
\renewcommand{\arraystretch}{1.24}
\begin{tabular}{|c|c|c|c|c|c|c|}
\hline
$G$ & $x_1$ & $x_2$ & $x_3$ & $ y_1$ & $ y_2 $ & $ y_3 $ \\
\hline
$\Omega_8^-(2)$ & 5 & 5 & 5 & 17& 17& 17 \\
$\Spin_8^-(3)$ & 41& 41& 41&7&7&7 \\
$\Omega_8^-(5)$& 31 & 31 &31  & && \\
$\Omega_{10}^-(2)$&33&33&33&5&5&5\\
$\Omega_{10}^-(3)$&130&130&130&&&\\
$\Omega_{10}^-(5)$&620&620&620&&&\\
$\Omega_{12}^-(2)$&65&65&65&17&17&17\\
$\Omega_{12}^-(3)$&41&41&41&&&\\
$\Omega_{12}^-(5)$&313&313&313&&&\\

\hline
\end{tabular}
\end{center}
\end{table}

\end{proof}

\begin{lemma}\label{orthogonalodd}
Let $G= \Spin_d(q)$.
\begin{enumerate}
\item If $d \ge 15$, then $G$ has  a hyperbolic triple of type $$(\lambda_{(d-7)a,p}\lambda_{3a,p},\lambda_{(d-7)a,p}\lambda_{3a,p},\lambda_{(d-5)a,p}\lambda_{4a,p}).$$
\item If $d \ge 7$, then $G$ has a hyperbolic triple of type $$(\lambda_{(d-1)a,p}, \lambda_{(d-1)a/2,p}, \lambda_{ka,p})$$  where $ d/2< k \le d-3$ and $k$ is even.
\item If $d\in \{11,13\}$, then $G$ has a hyperbolic triple of type $$(\lambda_{(d-1)a,p},\lambda_{(d-1)a,p}, \lambda_{(d-1)a/2,p}).$$
\item If $q> 3$, then $G$ has a hyperbolic triple of type $$(\lambda_{(d-5)a,p}(q^2-1)/2,\lambda_{(d-5)a,p}(q^2-1)/2,\lambda_{(d-3)a,p}(q+1)/2).$$
\end{enumerate}
In particular, $G$  and $G/Z(G)$ are  Beauville group unless possibly $G \cong \Spin_{11}(3)$ or $\Spin_{13}(3)$.
\end{lemma}

\begin{proof} To see that we have a hyperbolic triple as in (i) we select $x$ of order $\lambda_{(d-7)a,p}\lambda_{3a,p}$ to project into $\Omega_{d-7}^-(q)\times \Omega_6^+(q)$ and let $y$ be a conjugate of $x$. Then we may suppose that $xy$ is conjugate to an element $z$ of order $\lambda_{(d-5)a,p}\lambda_{4a,p}$ projecting into $\Omega_{d-5}^-(q)\times \Omega_4^-(q)$. Then we see that $x$ and $y$ are regular semisimple elements of $G$ and $H= \langle x,y\rangle$ acts irreducibly on $G$. It follows from Theorem~\ref{maingeneration} (iv) that, so long as $d \ge 15$,  $G=H$.

We next consider (ii). We take $x$ to project to an element of order $\lambda_{(d-1)a,p}$ projecting into $\Omega_{d-1}^-(q)$ and $y$ to be an element of order $\lambda_{(d-1)a/2,p}$ projecting into $\Omega_{d-1}^+(q)$. Then both $x$ and $y$ are regular semisimple element of $G$. Thus we may arrange for $xy$ to have order $\lambda_{ka,p}$ where $d/2<k \le d-3$ is even. As usual Theorem~\ref{maingeneration} (iv) that, so long as $d \ge 15$,  $G=H$.

Suppose $d=11$ and $d=13$ we employ Gow's Theorem with  Lemma~\ref{Pdegree} to see that their is a hyperbolic triple of type $(\lambda_{(d-1)a,p},\lambda_{(d-1)a,p}, \lambda_{(d-1)a/2,p})$ where $x$ and $y$ project into  $\Spin _{d-1}^-(q)$ and $z$ projects to an element of $\Spin_{d-1}^+(q)$ preserves singular subspaces of dimension $(d-1)/2$. This proves (iii).

Finally, we let $x$ be an element of order $\lambda_{(d-5)a,p}(q^2-1)/2$ projecting in $\Omega_{d-5}^-(q)\times \Omega_4^+(p)$. Then $x$ is regular semisimple and we choose a conjugate $y$ of $x$ such that $xy$ is conjugate to $z$ which projects into $\Omega_{d-3}(q)^-\times \Omega_2^-(q)$ and has order $\lambda_{(d-3)a,p}(q+1)/2$. Then so long as $q\neq 3$, we have that $H = \langle x,y\rangle = \langle x,z\rangle$ must act irreducibly on $V$. Hence again Theorem~\ref{maingeneration} (iv) shows that $G$ has a hyperbolic triple as in (iv).

Finally, (i) and (ii) show that $G$ and $G/Z(G)$ are Beauville groups when $d\ge 15$ and (iii) and (iv) show that $G$ is a Beauville group if $d\in \{11,13\}$ and $q\neq 3$.

\end{proof}

\begin{lemma} \label{O1113}The group $\Omega_{11}(3)$ and $\Omega_{13}(3)$ have a hyperbolic triple of type $(41,41,41)$.
\end{lemma}
\begin{proof} We calculated this by computer with the assistance of Theorem~\ref{maingeneration} (iv).\end{proof}

\begin{lemma}\label{O9} Suppose that $G= \Spin_9(q)$ with $q=p^a$ odd. Then
\begin{enumerate}
\item $G$ has a hyperbolic triple  of type $(\lambda_{8a,p},\lambda_{8a,p},\lambda_{4a,p})$.
    \item  If $q>  3$, $G$ has a hyperbolic triple   of type $$(\lambda_{6a,p} (q-1),\lambda_{6a,p}(q-1),\lambda_{4a,p}).$$
\end{enumerate}
In particular, if $q> 3$, then $G$ and $G/Z(G)$ are Beauville groups.
\end{lemma}

\begin{proof} We deal with both triples at once. For the first triple we let $x_1 \in \Omega_8^-(q)$ have order $\lambda_{8a,p}$ and $z_1$ have order $\lambda_{4a,p}$ be a power of a Singer element from $\GL_4(q) \le \Omega_8^+(q)$. We have that $x_1$ is regular semisimple in $G$ and  so we may suppose that $y_1$ is a conjugate of $x_1$ and that $x_1y_1=z_1$. Set $H_1 = \langle x_1,y_1\rangle$.   For the second triple, we take $r_1 \in \Omega_6^-(q) \times \Omega_2^+(q) \le  \Omega_8^+(q)$ with projection to the first factor of order $\lambda_{6a,p}$ and projection to the second factor of order $q-1$. For $s_1$ we take an element of $\Omega_4^-(q) \times \Omega_4^-(q) $ of order $\lambda_{4a,p}$ chosen so that in its action on $V$, it has two $4$-dimensional factors on which it induces non-isomorphic modules and a one space of $+$-type.  This choice is possible because $\lambda_{4a,p} > 5$ due  to our restriction that  $q>3$. Set $H_2= \langle r_1,s_1\rangle$. Suppose for a moment that
$s_1$ is conjugate to an element of $\langle z_1\rangle$. Then $s_1$ would also preserve a pair of $4$-dimensional singular spaces. However, our choice of $s_1$  shows that this is impossible. Thus once we show that $G=H_1=H_2$, we will have proved that $G$ is a Beauville group.  Notice that both $H_1$ and $H_2$ both act irreducibly on $V$. We now work through the possibilities for $H_1$ and $H_2$ given in Lemma~\ref{3.2}. Obviously Lemma~\ref{3.2} (i) does not hold.   If $H_1$ and $H_2$ are as in Example 2.1 of \cite{GPPS}, then we have that they equal $G$ by Lemma~\ref{subfield}.   If $H_1$ or $H_2$ are contained in an extension field subgroup, then it would have to normalize $\Omega_3(q^3)\cong \PSL_2(q^3)$. This group is not normalized by elements of order $\lambda_{8a,p}$ or $\lambda_{4a,p}$.
From Example 2.8 of \cite{GPPS} we have only to consider the possibility that $H_1$ or $H_2$ is equal to $\SL_3(q^2)$ or $\PSL_3(q^2)$ acting on the module $W\otimes W^\sigma$ where $W$ is the natural $\SL_3(q^2)$ module and $\sigma$ is the field automorphism of order $2$. In particular, this module is not self-dual, and so these cases cannot arise.
From Lemma~\ref{3.2} (iii), we have only one possibility with $d=9$ and then $q$ is even. So this falls also. Since the triple
 in (i) consists of elements of odd order, we have that $G$ is a Beauville group.
\end{proof}

\begin{lemma} \label{093} The groups $\Spin_9(3)$ and $\Omega_9(3)$ are Beauville groups.
\end{lemma}

\begin{proof} We use the hyperbolic triple from Lemma~\ref{O9} (i) which has type $(41,41,5)$. Then we calculate by computer that $\Omega_9(3)$ has a triple of type $(7,7,7)$. It follows that $\Spin_9(3)$ and $\Omega_9(3)$ are Beauville groups.
\end{proof}

\begin{lemma}\label{spin7} Suppose that $G = \Spin_7(q)$ with $q=p^a$ odd and $q>3$. Then
\begin{enumerate}

\item   $G$ has a hyperbolic triple  $(x,y,z)$  of type $(\lambda_{6a,p},\lambda_{3a,p},(q-\theta)/2 )$ where $q-\theta \equiv 2 \pmod 4$ and $\theta = \pm 1$,   $x \in \Spin_6^-(q)$ and $y\in \Spin_6^+(q)$ centralize a one-dimensional $-$-space, respectively  a one-dimensional $+$-space on $V$ and are regular semisimple and $z$ is a bireflection of odd order.
\item  If $q > 17$ then $G$ has a hyperbolic triple  $(x,y,z)$  of type $((q-1)/2,(q+1)/2,\lambda_{4a,p})$  where   $x \in \Spin_6^+(q)$ and  $y\in \Spin_6^-(q)$ are regular semisimple and   $z \in \Spin_4^-(q)$ has order $\lambda_{4a,p}$.
\end{enumerate}
In particular, if $q> 17$, then $G$ and $G/Z(G)$ are Beauville groups.
\end{lemma}

\begin{proof}
(i)  Let $x$ have order $\lambda_{6a,p}$ be an element of $\Spin_{6}^-(q)$ and $y \in \Spin_{6}^+(q)$ has order $\lambda_{3a,p}$. Then $x$ and $y$ are regular semisimple in $G$ and so we may suppose  $xy=z$ with $z $ a bireflection of order $(q-\theta)/2$.
Note that $\lambda_{6a,p}$ is not small as $q> 5$. Let $X= \langle x,y\rangle$.  The action of $x$ and $y$ on $V$ shows that $X$ acts irreducibly on $V$.  Suppose that $X \neq G$.  Then as $X$ contains an element of order $\lambda_{6a,p}$ and $p^a>5$,  Lemma~\ref{Pdegree} implies  $F^*(X)/Z(X) = \PSL_2(13)$ or we have that $F^*(X)$ is one of $\G_2(q)$, ${}^2\G_2(3^{2n+1})$, $\U_3(3^n)$, $n \ge 2$.  However, $z$ is a semisimple bireflection and, though \cite{GSa} says that $\G_2(q)$ has bireflections they are in fact elements of order $p$. Thus in this case we have $G=X$.

(ii) Suppose that $q > 17$ then $G$ has elements $x$ of order $(q-1)/2$  and $y$ of order $(q+1)/2$ with  $x \in \Spin_6^+(q)$ and  $y\in \Spin_6^-(q)$ both regular semisimple. We suppose that  $xy=z \in \Spin_4^-(q)$ has order $\lambda_{4a,p}$. Set $X= \langle x,y\rangle$.

Then by considering the types of spaces left invariant  by $x$ and $y$ and using the fact that the elements are regular semisimple, we see that if $X$ is not irreducible, then $X  \le \Spin_4^+(q) \circ \Spin_3(q)$. Since $\lambda_{4a,p}$ doesn't divide the order of this group, we infer that $X$ acts irreducibly on $V$. Applying Lemma~\ref{Pdegree} and using the fact that $4\neq 6$, we see that $G= X$.
\end{proof}

\begin{lemma}\label{theothers} The following groups have hyperbolic triples of the types indicated.
\begin{enumerate}
\item $\Spin_7(3)$ has hyperbolic triples of type $(13,13,13)$ and $(15,15,15)$.
\item $\Spin_7(5)$ has hyperbolic triples of type $(7,7,7)$ and $(31,31,31).$
\item $\Omega_7(7)$  has a hyperbolic triple of type $(100,7,200)$.
\item $\Omega_7(9)$ has a hyperbolic triple of type $(8,328,328)$.
\item $\Omega_7(11)$  has a hyperbolic triple of type $(11,61,122)$.
\item $\Omega_7(13)$ has a hyperbolic triple of type $(17,580, 2465)$.
\item  $\Omega_7(17)$ has a hyperbolic triple of type $(34,85,340)$.
\end{enumerate}
\end{lemma}

\begin{proof}
These were all calculated by computer. We mention that the representation of $\Spin_7(5)$ was constructed by finding the subgroup in $\Omega_8^+(5)$.
\end{proof}

Again we summarize this subsection with our main result.

\begin{theorem}\label{Othm} Suppose that $G = \Spin_d^\varepsilon (q)$ with $d \ge 7$ and $Z \le Z(G)$. Then $G/Z$ is a Beauville group.
\end{theorem}

\begin{proof} Combining the lemmas of this section, we see that $\Spin_d^\varepsilon(q)$ is a Beauville group unless possibly  $G=\Omega_8^+(2)$. This group is examined in  Theorem~\ref{excepcovers} and it and its Schur covers are seen to be Beauville groups.
\end{proof}
\section{Exceptional groups}\label{SS5}

This section is devoted to the proof of our main theorem for the exceptional quasisimple Lie type groups. We recall that $\Phi_k(x)$ denotes the  $k$'th cyclotomic polynomial. Our strategy is to pick certain semisimple elements of large order and determine all the maximal subgroups that contain such a group. We do this for the large rank groups by appealing to the results from \cite{LSS} and \cite{Lies} while for the smaller rank groups we rely on the entire lists of maximal subgroups. We then select a further semisimple element which is not contained in any of the maximal overgroups we discovered for the first element. The hyperbolic triples are then assembled from these elements. Throughout we will take $q=p^a$ where $p$ is a prime.

\subsection{Exceptional groups of type $\E_8$}

Let $G= \E_8(q)$.
Pick  $\sigma_1 \in G$ of order $ \Phi_{30}(q)$ and $\sigma_2 \in G$  of order $\Phi_{14}(q)$.
For $i=1,2$, define conjugacy classes $C_i =\sigma_i^G$.
Let $\tau_1 \in G$ be an element of order $ \Phi_{15}(q)$, $\tau_2 \in G$ of order $\Phi_7(q)$ and, for $i=1,2$, define $D_i = \tau_i^G$.
 Note that the existence of elements of $G$  of  orders $\Phi_{30}(q)$ and $\Phi_{15}(q)$  follows from \cite[Table C]{LSS} while elements of order
$\Phi_7(q)$ and $\Phi_{14}(q)$ are contained in the maximal rank subgroup $\Spin_{16}^+(q)$.

\begin{lemma}\label{e8x} There exists a hyperbolic triple  $(x_1,x_2,x_3) \in C_1\times C_1 \times C_2$.
\end{lemma}
\begin{proof}  As the class $C_1$ consists of regular semisimple elements and $C_2$ consists of semisimple
elements, Gow's Theorem guarantees that there exist $(x_1,x_2,x_3) \in C_1\times C_1 \times C_2$ with  $x_1x_2x_3=1$.
Then by \cite[Section 4 (j)]{Weigel}, the only maximal subgroup of $G$ containing $x_1$ is $N_G(\langle x_1 \rangle)  $ and, from \cite{LSS}, $N_G(\langle x_1 \rangle)/\langle x_1 \rangle$ is cyclic of order $30$.
Clearly $C_2 \cap N_G(\langle x_1 \rangle) = \emptyset$ as $x_3$ has order greater than $  30$ and $\gcd(\Phi_{30}(q),\Phi_{14}(q)) = 1$ by Lemma~\ref{cyclo}. Thus
$G = \langle x_1,x_3 \rangle$ and the lemma follows.\end{proof}

\begin{lemma}\label{UniqueMax} The $N_G(\langle \tau_1\rangle)$ is the unique maximal subgroup of $G$ containing $\tau_1$.
\end{lemma}

\begin{proof} Set $T = \langle \tau_1\rangle$ and $M$ be a maximal subgroup of $G$ containing $T$. Then $T$ is a maximal torus of $G$. In particular, $C_G(T)=T$.  Let $\ov G$ be a simple algebraic group of type $\E_8$ and $\sigma $ be a Frobenius automorphism of $\ov G$ such that $G= \ov G_\sigma$ is the fixed points of $\sigma$ in $\ov G$.
In \cite[Theorem 8]{Lies} they  split up the maximal subgroups of $G$ into various classes. The first of these are the subgroups which arise as $M_\sigma$ where $M$ is a maximal $\sigma$-stable closed subgroup of $\ov G$. If $T \le M_\sigma$ with $M$ as just described, then   $M_\sigma= N_G(T)$ by \cite{LSS}. The second possibility is that $T \le M$ where $M$ has the same type is $G$. However, we would then have that $M= \E_8(q_0)$ with $\GF(q_0)$ a subfield of $\GF(q)$ and $T$ would be a maximal torus of $M$. Thus, by \cite[Proposition 2.4]{Lubeck}, $$|T| \le (q_0+1)^8 ,$$which is impossible. The next two possibilities are that $M$ is an exotic local subgroup or that $M\cong (\Alt(5)\times \Alt(6)):2^2$. Both of these are ruled out as they do not have cyclic subgroups as large as $T$. The fifth possibility is that $F^*(M)$ is simple and that the isomorphism type of $M$ is as detailed in \cite[Table 2]{Lies}. Now note that the prime divisors of $\Phi_{15}(q)$ are all congruent $1$ mod $15$. The simple groups listed in the first 4 sections of \cite[Table 2]{Lies} have all prime divisors less than $31$. For the other groups, we note that $\Phi_{15}(2) =151$ and $\Phi_{15}(q)\ge 4561$ and this comment eliminates all the remaining groups in the table. It now follows that $F^*(M)=M(q_0)$ is Lie type group defined in characteristic $p$. Here the situation is that either \begin{enumerate} \item $q_0 \le 9$ and $M(q_0)$ has rank at most $4$; \item $M (q_0)= \A_2^\epsilon(16)$; or \item $M(q_0) \cong \A_1(q_0)$, ${}^2\B_2(q_0)$ or ${}^2\G_2(q_0)$. \end{enumerate}
In the first of these cases we have that $T \le N_G(M)$ and $|N_G(M)/M|\le 5.6.3=90$  where the $5$ is the maximum contribution of diagonal automorphisms, the $6$ is the maximal possible size of a group of graph automorphisms and the $3$ is the maximum contribution from field automorphisms. Hence $|T \cap F^*(M)|\ge\Phi_{15}(q)/90$.  On the other hand, $|T\cap F^*(M)| \le (q_0+1)^4 \le 10000$ which means that $|T| \le 900000$. This shows that $q \le  7$. For $q=5 $ or $ 7$, we have $q_0=5$ and so Zsigmondy's theorem delivers a contradiction. So $q= 2,3$ or $4$.
 If $q=2$, then $|T| = 151$ is prime, if $q=3$, then $|T|= 4561$ is also prime and if $q=4$, then $|T|= 151.331$. In the first case we get that $q_0= 4,8$ and in the second case we have $q_0 = 9$ and in the last case we have $q_0=8$. Noting that all rank $4$ groups have over a field of order $r$ have order dividing $|\F_4(r)|$ or $|\SL_5(r)|$. We  check that the only possibility is that $q=2$ and $F^*(M) = \SL_5(8)$ (which does have order dividing the order of $\E_8(2)$). However, $\SL_5(8)$ contains an element of order $151.31$ whereas $\E_8(2)$ does not. Thus cases in  (i)  are ruled out.

Obviously case (ii) is not a serious candidate for an overgroup of $T$. So assume that we are in case (iii). Suppose that $F^*(M)= M(q_0)$. Let $z \in T$ be an element of order $\zeta_{15a,p}$. If $|M(q_0)|$ is coprime to $o(z)$, then $z$ induces a field automorphism of $F^*(M)$. This means that $z$ normalizes a $p$-subgroup of $G$ and consequently $o(z)$ divides the order of a parabolic subgroup of $G$, which is a contradiction. Therefore $F^*(M)$ contains a cyclic subgroup of the same order as $o(z)$. Since $G$ contains a unique conjugacy class of such subgroups, we have that $z \in F^*(M)$. Since $z$ does not normalize a $p$-subgroup, $o(z)$ does not divide $q_0-1$.  The centralizers of semisimple elements contained in $T$ appear in the first tranche of subgroups in \cite{Lies}, we see that in fact the $C_{F^*(M)}(z) \le T$.  Assume that $|C_{F^*(M)}(z)|$ divides $q_0+1$. Then $(q_0+1)/\gcd(2,p-1)$ divides $\Phi_{15}(q)$. In particular, $q_0 \le q^8$. Since $\zeta_{15a,p}$ divides $(q_0+1)/2$, we use Lemma~\ref{gcd}(i) to get that $p^{15a}-1$ divides $p^{2b}-1$ where $q_0= p^b$. If $a$ is odd, this mean that $q_0 \ge p^{15a}= q^{15}$ which is impossible. Hence $a$ is even and we have $q_0 = q^{15a/2}$ and $q_0= q^{1/2}$. Because  $\Phi_{15}(q)$ is odd, we now have $(q_0+1)/\gcd(2,p-1)$ is odd. Since $\gcd(q^{1/2}+1,\Phi_{15}(q))= 1$, we have that $q^{15/2}+1$ does not divide $\Phi_{15}(q)$. This contradiction shows that $|C_{F^*(M)}(z)|$ does not divide $q_0+1$. Therefore $F^*(M)$ is either a Suzuki group or   a Ree group. Now we have that $o(z)$ divides $q_0\pm \sqrt {pq_0} +1$. Then $T \le C_{M}(C_{F^*(M)}(z))$ which is equal to one of the tori in $F^*(M)$ of order
$q_0\pm \sqrt {pq_0} +1$.  Hence $q_0\pm \sqrt {pq_0} +1= \Phi_{15}(q)=q^8-q^7+q^5-q^4+q^3-q+1$. This implies that $q= \sqrt {pq_0}$ and then $(q-1)q^7 > q_0$ which is absurd.

\end{proof}

\begin{lemma}\label{e8y}
There exist a hyperbolic triple in $(y_1,y_2,y_3) \in D_1\times D_1 \times D_2$.
\end{lemma}

\begin{proof} As the class $D_1$ consists of regular semisimple elements and $D_2$ consists of semisimple
elements, Gow's Theorem gives a triple $ (y_1,y_2,y_3) \in D_1\times D_1 \times D_2$ with $y_1y_2y_3=1$. By Lemma~\ref{UniqueMax},
 $N_G(\langle y_1\rangle ) $ is the unique maximal subgroup containing $\langle y_1\rangle$ and by \cite[Table~5.2]{LSS}, $N_G( \langle y_1\rangle)/\langle y_1\rangle$ is cyclic of order $30$.
Clearly $D_2 \cap N_G(\langle y_1\rangle ) = \emptyset$ as $o(y_3) >  30$ and $\gcd(o(y_1),o(y_3)) = 1$. Thus
$G = \langle y_1,y_3\rangle $ and the result follows.\end{proof}

\begin{theorem} \label{beauE8}  $G= \E_8(q)$ is a Beauville group.
\end{theorem}

\begin{proof} Lemmas \ref{e8x} and \ref{e8y} provide two hyperbolic triples for $G$, whereas the fact that
$\gcd(\Phi_a(q),\Phi_b(q)) = 1$ whenever $a \in \{7, 15\}$ and $b \in \{ 15 , 30\}$ gives the non-conjugacy requirement. Hence $G$ is a Beauville group.\end{proof}

\subsection{Exceptional groups of type $\E_7$}

 Let $G = {\rm E}_7(q)$ be the universal group of type $E_7(q)$. Thus $Z(G)$ has order $(p-1,2)$. Let $\sigma_1 \in G$ of order $ \Phi_{18}(q)$, $\sigma_2 \in G$  of order $(q^7+1)/\gcd(q-1,2)$, $\tau_1 \in G$ be an element of order $ \Phi_{9}(q)$ and  $\tau_2 \in G$ of order $(q^7-1)/\gcd(q-1,2)$.
For $i=1,2$, define conjugacy classes $C_i =\sigma_i^G$ and $D_i = \tau_i^G$.

\begin{lemma} \label{maxE7}  Suppose that $M$ is a maximal subgroup of $G$.
\begin{enumerate}
\item If  $M \cap C_2$ is non-empty, then either $F^*(M) = \SL_2(q^7)$ or $\SL^\epsilon_ 8(q)/\gcd (2,q-1)$.
\item If  $M \cap D_2$ is non-empty, then either  $M$ is  an $\A_6(q)$-parabolic subgroup, $F^*(M) = \SL_2(q^7)$ or  $\SL_8(q)/\gcd (2,q-1)$.

\end{enumerate}
\end{lemma}
\begin{proof} Let $G = \E_7(q)$, and $M$ be a maximal subgroup of $G$ containing either $\sigma_2$ or $\tau _2$. If $M$ is a parabolic subgroup, then $M$ is as described in (ii).  Note that any closed subgroup  of the algebraic group $\E_7$ contains $\sigma_2$ or $\tau_2$ is of positive rank. Referring to \cite[Theorem 8(b)]{Lies},  we  consult the list in \cite[Table 5.1]{LSS} to see that once again (i) and (ii) hold.  The groups listed in \cite[Theorem 8(c)]{Lies} are not divisible by $q^7\pm 1$ and the same is true for the groups in \cite[Table~3]{Lies}.

There are no exotic local subgroups in $\E_7(q)$.

Suppose  that $F^*(M)$ is simple group. If $F^*(M)$ is not  a Lie type group in characteristic $p$, then we note that the automorphism groups of the groups listed in \cite[Table~2]{Lies} do not have elements of order $(q^7\pm 1)/\gcd(p-1,2)$. Hence $F^*(M)$ is a Lie pe group defined in characteristic $p$.

We work through the three possibilities itemized in \cite[Theorem 8]{Lies}.
Assume first that the rank of $M$ is at most $3$ and that $M$ is defined over $q_0$ which is at most $9$. These groups are all too small to contain elements of order $(q^7\pm 1)/\gcd(p-1,2)$.
It is also impossible that $F^*(M)=\A_2^\epsilon(16)$.

Assume that $F^*(M) = \A_1(q_0)$, ${}^2\B_2(q_0)$, or ${}^2\G_2(q_0)$ with $q_0 \le \gcd (2,p-1)388$.  Then we have $q_0= p^b$ where $b \le 8$.

Let $z$ be a power of $\sigma_2$ or $\tau_2$ element of order a $\lambda_{7a,p}$ or $\lambda_{14a,p}$. Then $o(z)> 7$ by Theorem~\ref{LZ}. Hence $z \in F^*(M)$. Writing $q=p^a$, we have $7a$ divides $2b$ by Lemma~\ref{gcd} and Theorem~\ref{Zig}. Thus $7$ divides $b$ and consequently $b=7$ and $p=2$. Since $4^7\pm 1$  exceeds the order of elements in $\A_1(2^7)$ we must have $q=2$ and again (ii) holds.
If  $F^*(M)$ is a Suzuki or a Ree group, then the bound on $q_0$ implies that either $q_0= 2^3$ or $2^5$, $2^7$  or $q_0=3, 3^3$ or $3^5$ respectively.  Considering element orders gives us that $F^*(M) = {}^2\B_2(2^7)$ and $q=2$. However this group has order divisible by $113$ whereas $\E_7(2)$ does not.
This proves the lemma.
\end{proof}

\begin{corollary}  \label{maxesE7} Suppose that $M$ is a subgroup of $G$.
\begin{enumerate}
\item If both $M \cap C_1$ and $M \cap C_2$ are non-empty, then $M = G$.
\item If both $M \cap D_1$ and $M \cap D_2$ are non-empty, then $M = G$.
\end{enumerate}
\end{corollary}
\begin{proof} We note that  the groups listed in the conclusions of Lemma~\ref{maxE7} are not divisible by $\Phi_{18}(q)$ or $\Phi_9(q)$.
Our claims  follow.
\end{proof}

\begin{lemma}\label{e7x}There exist a hyperbolic triple for $G$ in $C_2\times C_2 \times C_1$.
\end{lemma}

\begin{proof} As the class $C_2$ consists of regular semisimple elements and $C_1$ consists of semisimple
elements, Gow's Theorem shows that there is a triple in $C_2\times C_2 \times C_1$ with product 1. Corollary \ref{maxesE7} then gives
$G = \langle x_1,x_3\rangle $ and our claim follows.\end{proof}

\begin{lemma}\label{e7y}There exist a hyperbolic triple for $G$ in $D_2\times D_2 \times D_1$.
\end{lemma}

\begin{proof} As the class $D_2$ consists of regular semisimple elements and $D_1$ consists of semisimple
elements, Gow's Theorem guarantees that there is a triple in  $D_2\times D_2 \times D_1$ with product $1$. Now Corollary \ref{maxesE7} gives
$G = \langle y_1,y_3\rangle $. This proves the lemma.
\end{proof}

\begin{theorem} \label{beauE7}  Let  $G=\E_7(q)$. Then $G$ and $G/Z(G)$ are Beauville groups.
\end{theorem}

\begin{proof}  We have  $\sigma_1$ and $\tau_1$ have coprime orders by Lemma~\ref{cyclo}(iii). Also  $\gcd ((q^7+1)/\gcd(p-1,2), (q^7-1)/\gcd(p-1,2))=1$ so $\sigma_2$ and $\tau_2$ have coprime orders.  Noting that $q^7-1 = \Phi_1(q)\Phi_7(q)$ and $q^7+1 = \Phi_2(q)\Phi_{14}(q)$, we use Lemma~\ref{cyclo} (iii) again to see that $\sigma_1$ and $\tau_2$ and $\sigma_2$ and $\tau_1$ have coprime orders. Thus taking the hyperbolic triples provided by Lemmas \ref{e7x} and \ref{e7y}, we see that the non-conjugacy requirement condition for a Beauville system is satisfied. Therefore $G$ and $G/Z(G)$ are Beauville groups.
\end{proof}

\subsection{Exceptional groups of type $\E_6$}

In this subsection, we let $G= \E_6(q)$ be the universal group of type $\E_6(q)$. Thus $Z(G)$ has order $\gcd(q-1,3)$.  We let $\sigma_1$ be an element of order $\Phi_9(q)/\gcd(q-1,3)$  (which has order coprime to $3$) and $\sigma_2$ have order $\Phi_4(q)$ (seen for example in the subgroups of $G$ isomorphic to $\SL_3(q^3)$ and $\Spin_8^+(q)$ respectively).  For $i=1,2$, set $C_i =\sigma_i^G$.
 Since $G$ has a subgroup of shape ${}^3\mathrm D_4(q) \times (q^2+q+1)$,  $G$ has an element $\tau_1$ of order $\Phi_3(q) \Phi_{12}(q)$
 as $\Phi_3(q)$ and $\Phi_{12}(q)$ are coprime by Lemma~\ref{cyclo} (iii).  We then choose $\tau_2 \in G$ of order $\Phi_5(q)$ (contained in $\Spin_{10}^+(q)$) and, for $i=1,2$, define $D_i = \tau_i^G$.

\begin{lemma} \label{maxE6}  Suppose that $M$ is a maximal subgroup of $G$.
\begin{enumerate}
\item If $M \cap C_1$ is non-empty, then $F^*(M) \cong \SL_3(q^3)$.
\item If  $M \cap D_1$ is non-empty, then $M \cong (^3\mathrm D_4(q) \times (q^2+q+1)).3$.
\end{enumerate}
\end{lemma}

\begin{proof} We view $G=\E_6(q)$ as the fixed points of the appropriate endomorphism $\sigma$ of the algebraic group $\ov G$. Again we use \cite[Theorem 8]{Lies}.  We let $M$ be a maximal subgroup of $G$ and assume that $\sigma_1$ or $\tau _1 \in M$. If $M$ is the fixed points of an $\sigma$-stable subgroup of positive dimension in $G$, then, after noting that parabolic subgroups and the subgroups listed in \cite[Table~1]{Lies}  do not have elements of  orders $\Phi_9(q)/\gcd(q-1,2)$ or $\Phi_{12}(q)\Phi_3(q)$, we get  that $M$ is the fixed points of an $\sigma$-stable reductive, maximal rank subgroup of $\ov G$. Then we  apply \cite[Tables 5.1 and 5.2]{LSS}, to see that either (i) or (ii) holds.  If $M$ has the same type as $G$, then $F*(M) \cong \E_6(q_0)$ or ${}^2\E_6(q_0)$ where $\GF(q_0)$ is a proper subfield of $\GF(q)$. We adopt the argument from the proof of Theorem~\ref{UniqueMax} and note that a torus in $M$ has order at most $$(q-1)^6 < |T|\le 3(q_0+1)^6$$ where $T = \langle \sigma_1\rangle$ or $\langle \tau_1\rangle$. It follows that $q_0=2$ and $q=4$, Since $\Phi_9(q)/\gcd(q-1,3)$ is divisible by 19 whereas $|\E_6(2)|$ is not, we have a contradiction in the case that $T= \langle \sigma_1\rangle$.  Similarly, we eliminate $T = \langle \tau_1\rangle$ as $241$ divides $\Phi_{12}(4)$ but not $|\E_6(2)|$.

Since $\Phi_9(q)/\gcd(q-1,3) \ge \Phi_9(2)= 2^6+2^3+1= 73$ and $\Phi_{12}(q)\Phi_3(q) \ge \Phi_{12}(2)\Phi_3(2)=91$, we have that $M$ has cyclic subgroups of large order and furthermore Zsigmondy's Theorem~\ref{Zig} implies that $\Phi_9(q)$ is divisible by a prime greater than $10$ and $\Phi_{12}(q)$ by a prime at least $13$. These observations immediately show that $M$ is not an exotic local subgroup of $G$.

Suppose now that $F^*(M)$ is a simple group.
If $F^*(M)$ is not a Lie type group defined in characteristic $p$, then the possibilities for $F^*(M)$ are enumerated in Table~2 of \cite{Lies}. They all appear in the Atlas \cite{ATLAS} and it is elementary to check that they do not have element of sufficiently large order to be contenders for over groups of $\sigma_1$ or $\tau_1$.

Suppose that $F^*(M)$ is a Lie type group defined in characteristic $p$. We work through the three possibilities itemized in \cite{Lies}.
Assume first that the rank of $M$ is at most $3$ and that $M$ is defined over $q_0$ which is at most $9$. Then $F^*(M)$ has a projective representation of dimension at most $7$ over $\GF(q_0)$.  It follows that $q < q_0$ as Zsigmondy primes dividing $\Phi_9(q)$ or $\Phi_{12}(q)$ have no such projective representations for $q\ge q_0$  Hence $(q,q_0) = (2,4)$, $(2,8)$ or $(3,9)$.  If $q=2$, we confirm parts (i) and (ii) by looking at the maximal subgroups of $G$ listed  in \cite{KW}. So we have $(q,q_0) = (3,9)$. It is impossible for $\sigma_1 \in M$ as this requires a cubic extension of $\GF(3)$. If $\tau_1 \in M$, then  the smallest  projective representation of $M$ over $\GF(9)$ has dimension at least $6$, hence $F^*(M) \cong \PSp_6(9)$ or $\POmega_7(9)$. In these candidates for $M$, we have no elements of order $\Phi_3(3)\Phi_{12}(3)= 949$.

The next possibility is that $F^*(M)=\A_2^\epsilon(16)$. But such groups cannot contain either $\sigma_1$ or $\tau_1$.

Assume that $F^*(M) = \A_1(q_0)$, ${}^2\B_2(q_0)$, or ${}^2\G_2(q_0)$ with $q_0 \le \gcd (2,p-1)124$.  In particular, we have $q_0= p^b$ where $b \le 7$. Let $T = \langle \sigma_1 \rangle $ or $T=\langle \tau_1\rangle$. So
Let $z\in T$ be an element of order $\zeta_{9a,p}$ or $\zeta_{12a,p}$. Then $z$ has order at least $11$ and so $z \in F^*(M)$. Since $o(z)$ does not divide the order of a parabolic subgroup of $G$, we have $z$ does not lie in  split torus of $F^*(M)$. Suppose that $z$ is contained in a torus of $F^*(M)$ of order dividing $q_0+1$.
If $o(z)=\zeta_{9a,p}$, then $o(z)$ divides $q_0+1$ and $p^{9a}-1$ hence Lemma~\ref{gcd} (i) and Theorem~\ref{Zig} imply that $b \ge 9$ which is a contradiction.  So suppose that $z$ has order $\zeta_{12a,p}$. Then Lemma~\ref{gcd} (i) and Theorem~\ref{Zig} imply that $b$ is a multiple $6a$. Thus  $q_0 = q^6$ and  $q=2$. We now apply  \cite{KW} to get that $o(z)$ does not divide $q_0+1$. Thus $F^*(M)$ is a Suzuki or a Ree group. The bound on $q_0$ implies that either $q_0= 2^3$ or $2^5$ or $q_0=3, 3^3$ or $3^5$ respectively. For the four smallest groups, we refer to the Atlas \cite{ATLAS} to see that the element orders are incompatible. For $F^*(M) = {}^2\G_2(243)$, we easily get that $q=3$ and then note that $\Phi_9(3)$ and $\Phi_{12}(3)$ are coprime to $|\E_6(3)|$. This completes the proof of the lemma. \end{proof}

\begin{corollary}  \label{maxesE6} Suppose that $M$ a subgroup of $G$.
\begin{enumerate}
\item If both $M \cap C_1$ and $M \cap C_2$ are non-empty, then $M = G$.
\item If both $M \cap D_1$ and $M \cap D_2$ are non-empty, then $M = G$.
\end{enumerate}
\end{corollary}
\begin{proof}  We note that $ M = (^3\mathrm D_4(q) \times (q^2+q+1)).3$ contains no elements of order $\Phi_{5}(q)$ and that
$\SL_3(q^3)$ contains no elements of order $\Phi_{4}(q)$. Our claims now follow from Lemma \ref{maxE6}
\end{proof}

\begin{lemma}\label{e6x}
There exist a hyperbolic triple for $G$ in $C_1\times C_1 \times C_2$.
\end{lemma}

\begin{proof} As the class $C_1$ consists of regular semisimple elements and $C_2$ consists of semisimple
elements, Gow's Theorem guarantees that there exist $x_i$ with product $1$. Now Corollary \ref{maxesE6} gives
$G = \langle x_1,x_3\rangle $ and our claim follows.\end{proof}

\begin{lemma}\label{e6y}
There exist a hyperbolic triple for $G$ in $D_1\times D_1 \times D_2$.
\end{lemma}

\begin{proof} As the class $D_1$ consists of regular semisimple elements and $D_2$ consists of semisimple
elements, Gow's Theorem guarantees that there exist $y_i$ with product $1$. Now Corollary \ref{maxesE6} gives
$G = \langle y_1,y_3\rangle $ and our claim follows.\end{proof}

\begin{theorem} \label{beauE6} Let $G = \E_6(q)$. Then $G$ and $G/Z(G)$ are Beauville groups.
\end{theorem}

\begin{proof} We take the hyperbolic triples for $G$ provided by Lemmas \ref{e6x} and \ref{e6y}. Since $\sigma_1$ has order $\Phi_9(q)/\gcd(q-1,3)$ Lemma~\ref{cyclo} implies that the non-conjugacy conditions for a Beauville system are satisfied.\end{proof}

\subsection{Exceptional groups of type ${}^2\E_6$}

We consider $G={}^2\E_6(q)$. This time $Z(G)$ has order $\gcd(q+1,3)$.  Let $\sigma_1 \in G$ of order $ \Phi_{18}(q)/\gcd(q+1,3)$, $\sigma_2 \in G$  of order $\Phi_{4}(q)$.
Define $C_i =(\sigma_i)^G$. Let $\tau_1 \in G$ be an element of order $ \Phi_{12}(q)\Phi_6(q)$, $\tau_2 \in G$ of order $\Phi_{10}(q)$  and define $D_i = (\tau_i)^G$.

\begin{lemma} \label{max2E6} Suppose that  $q \ge 3$ and $M$ is a maximal subgroup of $G$. Then the following are true.
\begin{enumerate}
\item If $M \cap C_1$ is non-empty, then $ M = (^3\mathrm D_4(q) \times (q^2-q+1)).3$.
\item If $M \cap D_1$ is non-empty, then $F^*(M)$ is $\SU_3(q^3)$.
\end{enumerate}
\end{lemma}

\begin{proof}
We view $G={}^2\E_6(q)$ as the fixed points of the appropriate endomorphism $\sigma$ of the algebraic group $\ov G$ of type $E_6$. Let $M$ be a maximal subgroup of $G$ and assume that $T =\langle \sigma_1\rangle$ or $T= \langle \tau _1\rangle $ is contained in $M$. If $M$ is the fixed points of an $\sigma$-stable subgroup of positive dimension in $G$, then, after noting that parabolic subgroups and the subgroups listed in \cite[Table~1]{Lies}  do not have elements of  orders $\Phi_{18}(q)/\gcd(q+1,3)$ or $\Phi_{12}(q)\Phi_6(q)$, we get  that $M$ is the fixed points of an $\sigma$-stable reductive subgroup of positive rank in $\ov G$. Then we  apply \cite{LSS}, to see that either (i) or (ii) holds.   If $F^*(M) = {}^2\E_6(q_0)$ with $\GF(q_0)$ a proper subfield of $\GF(q)$,  then we note that the Zsigmondy primes appearing in $\Phi_{18}(q)$ and $\Phi_{12}(q)$ do not divide $|M|$.

We have that $\Phi_{18}(q)/\gcd(q+1,3) \ge \Phi_{18}(3)= 703$ and $\Phi_{12}(q)\Phi_6(q) \ge \Phi_{12}(3)\Phi_6(3)=511$. Hence the cyclic subgroups of $M$ have  large order. Furthermore Zsigmondy's Theorem implies that $\Phi_{18}(q)$ is divisible by a prime at $19$ and $\Phi_{12}(q)$ by a prime at least $13$. This information implies that $M$ is not an exotic local subgroup of $G$.

Suppose now that $F^*(M)$ is a simple group. If $F^*(M)$ is not a Lie type group defined in characteristic $p$, then the possibilities for $F^*(M)$ are listed in Table~2 of \cite{Lies}. They all appear in the Atlas \cite{ATLAS} and we see they do not have elements of sufficiently large order to be possibilities for overgroups of $T$.

So we suppose that $F^*(M)$ is a Lie type group defined in characteristic $p$. Again we work through the three possibilities itemized in \cite{Lies}.
Suppose first that the (untwisted) rank of $M$ is at most $3$ and that $M$ is defined over $q_0$ which is at most $9$. Then, as in the proof of Lemma~\ref{maxE6} we argue that $F^*(M)$ has a projective representation of dimension at most $7$ over $\GF(q_0)$.  Thus $q < q_0$. As $q \neq 2$, by hypothesis, we have    $(q,q_0)= (3,9)$.  It is now impossible for $\sigma_1 \in M$. If $\tau_1 \in M$, then the smallest  projective representation over $\GF(9)$ for $M$ has dimension at most $6$, hence $F^*(M) \cong \PSp_6(9)$ or $\POmega_7(9)$. In these candidates for $M$, we have no elements of order $\Phi_6(3)\Phi_{12}(3)= 511$.

As usual we cannot have  $F^*(M)=\A_2^\epsilon(16)$.

Assume that $F^*(M)$ is one of $\A_1(q_0)$, ${}^2\B_2(q_0)$, or ${}^2\G_2(q_0)$ with $q_0 \le \gcd (2,p-1)124$.  Then again $q_0= p^b$ where $b \le 7$.
Let $z\in T$ have order either $\zeta_{18a,p}$ or $\zeta_{12a,p}$. Then $z$ has order at least $13$ and so $z \in F^*(M)$. As before $z$ is not contained in a split torus or else it would be in a parabolic subgroup. Assume that $o(z)$ divides $q_0+1$. By considering the $\gcd(q_0+ 1, q^{12}-1)$ and noting that this number is divisible by $\zeta_{12a,p}$ or $\zeta_{18a,p}$, we get that $q_0= p^b$ where $b$ is a multiple $6a$ and $q=p^a$. This gives us a contradiction as $q>2$.

Thus $F^*(M)$ is a Suzuki or a Ree group. The bound on $q_0$ implies that either $q_0= 2^3$ or $2^5$ or $q_0=3, 3^3$ or $3^5$ respectively. Again the small groups are eliminated by looking in the  Atlas \cite{ATLAS}.  For $F^*(M) = {}^2\G_2(243)$, we easily get that $q=3$ and then we note that $\Phi_{18}(3)$ and $\Phi_{12}(3)$ are coprime to $|{}^2\E_6(3)|$.\end{proof}

\begin{corollary}  \label{maxes2E6} Suppose that $M$ a subgroup of $G$ and $q>2$.
\begin{enumerate}
\item If both $M \cap C_1$ and $M \cap C_2$ are non-empty, then $M = G$.
\item If both $M \cap D_1$ and $M \cap D_2$ are non-empty, then $M = G$.
\end{enumerate}
\end{corollary}
\begin{proof} We note that $ M = (^3\rm D_4(q) \times (q^2-q+1)).3$ contains no elements of order $\Phi_{10}(q)$ and that
$\PSU_3(q^3)$ contains no elements of order $\Phi_{4}(q)$. Our claims now follow from Lemma \ref{max2E6}.\end{proof}

\begin{lemma}\label{2e6x}For $q> 2$, there are hyperbolic triples in $C_1 \times C_1 \times C_2$ and $D_1\times D_1\times D_2$.
\end{lemma}

\begin{proof} This is just as in Lemmas~\ref{e6x} and \ref{e6y}.\end{proof}

\begin{theorem} \label{beau2E6} Let $G = {}^2{\rm E}_6(q)$. Then $G$ and $G/Z(G)$ are Beauville groups.
\end{theorem}

\begin{proof} For $q>2$, this is just as in Theorem~\ref{beauE6}. For $q=2$ we refer forward to Theorem~\ref{excepcovers}.
\end{proof}

\subsection{Exceptional groups of type $\F_4$}

 In this subsection  $G = {\rm F}_4(q)$. Using \cite{LSS}, we can pick
 $\sigma_1 \in G$ of order $ \Phi_{12}(q)$ and  $\sigma_2 \in G$  be an element of order $\Phi_{4}(q)/\gcd(p-1,2)$.  For $i= 1,2$, define $C_i =\sigma_i^G$. Let $\tau_1 \in G$ be an element of order $ \Phi_{8}(q)$ and $\tau_2 \in G$ a  regular semisimple element  of order $\Phi_3(q)$. For $i=1,2$, set  $D_i = \tau_i^G$.

\begin{lemma} \label{maxF4}Suppose $q>2$ and that $M$ is a maximal subgroup of $G$.
\begin{enumerate}
\item If $C_1 \cap M $is non-empty, then  $ M \cong  {}^3\rm D_4(q).3$.
\item If $D_1 \cap M$ is non-empty, then $M \cong  \Spin_9(q)$ or possibly $q=3$ and $F^*(M) \cong \PSp_4(9)$ or $\PSL_2(81)$.
\end{enumerate}
\end{lemma}

\begin{proof} We view $G$ as the fixed point of the endomorphism $\sigma$ an algebraic group of type $F_4$.  As is now familiar we follow the strategy dictated by  \cite{Lies}. Let $M$ be a maximal subgroup of $G$ which contains either $\sigma_1$ or $\tau_1$. If $M$ is the fixed points of a $\sigma$-stable subgroup of positive dimension in $G$ then we refer to \cite{LSS} to obtain (i) and (ii) as none of the parabolic subgroups of $G$ contain such elements and the groups listed in \cite[Table~3]{Lies} are too small. If $M$ has the same type as $G$, then $F^*(M) \cong \F_4(q_0)$ or ${}^2\F_4(q_0)$. In the first case we use Theorem~\ref{Zig} to see that $M$ cannot contain $T$ while in the second case we refer to \cite{malle} to get the same conclusion.

 The exotic local subgroups of $G$ and non-characteristic $p$ possibilities listed in \cite{Lies} are also too small just as in the case of $\E_6(q)$.  Thus we have $F^*(M)$ is a simple group defined in characteristic $p$ with rank at most two. In addition we may as well suppose that $\sigma_1 \not \in M$ as otherwise the more difficult arguments we employed for $\E_6(q)$ are valid. So $\tau_1\in M$ has order $q^4+1$. If $F^*(H)$ has rank $2$, then $q_0 \le 9$. Since $\zeta_{8a,p}$ must divide $|F^*(M)|$ we get $q=3$ and $q_0=9$. The only possibility is that $F^*(M) \cong \PSp_4(9)$ and we have included this case in our conclusion. Again the groups $\A_2^\epsilon(16)$ are impossible as $q>2$. Finally we have to consider the rank $1$ groups. This time $q_0 \le 68.\gcd(p-1,2)$ and $q_0= p^b $ where $b \le 6$.  As before we may assume that $\sigma_1\not \in M$. Further we get that $T= C_M(T \cap F^*(M))$ and that $C_{F^*(M)}(T \cap F^*(M))$ is a maximal torus of $F^*(M)$. Since $T$ does not normalize a $p$-subgroup, we get either $F^*(M) = \A_1(q_0)$ and $(q_0 +1)/\gcd(p-1,2)$ divides $q^4+1$ or $F^*(M)$ is a Ree or Suzuki group and $q_0\pm \sqrt{pq}+1= q^4+1$. The latter possibilities are clearly impossible. So $F^*(M) = A_1(q_0)$ and we get $q_0= q^4$. As $q> 2$, the bound on $q_0$ now forces $q=3$. So $F^*(M) = \A_1(81)$ and this is our final listed possibility.
\end{proof}

\begin{corollary}  \label{maxesF4} Suppose that $M$ is a subgroup of $G$ and $q>2$.
\begin{enumerate}
\item If $M \cap C_1$ and $M \cap C_2$ are both non-empty,  then $M = G$.
\item If $M \cap D_1$ and $M \cap D_2$ are both non-empty, then $M = G$.
\end{enumerate}
\end{corollary}
\begin{proof}  We note that $ M = {}^3\mathrm D_4(q).3$ contains no elements of maximal odd order dividing $\Phi_{4}(q)$ and that
$\Spin_9(q)$  ($\cong \Sp_8(q)$ when $q$ is even) contains no regular semisimple elements of order $\Phi_{3}(q)$. Our claims now both follow from Lemma \ref{maxF4} when $q> 3$. When $q=3$ and the exceptional case in Lemma~\ref{maxF4} occurs, we note further that $\Phi_3(3)$ does not divide $|\Aut(\PSp_4(9))|$ and we are done.   \end{proof}

\begin{theorem} \label{beauF4}Let $G= \F_4(q)$. Then \begin{enumerate}\item  there are hyperbolic triples for $G$ in $C_1 \times C_1 \times C_2$; and  \item there are hyperbolic triples for $G$ in $D_1 \times D_1 \times D_2$.\end{enumerate}
In particular, $G$ is a Beauville group.
\end{theorem}

\begin{proof} As the class $C_1$ consists of regular semisimple elements and $C_2$ consists of semisimple
elements, Gow's Theorem guarantees that there exist $(x_1,x_2,x_3) \in C_1 \times C_1 \times C_2$ with  $x_1x_2x_3=1$. Now Corollary \ref{maxesF4} (i) gives
$G = \langle x_1,x_3\rangle $ and (i) follows.

 As the class $D_1$ consists of regular semisimple elements and $D_2$ consists of semisimple
elements, Gow's Theorem shows that there exist $(y_1,y_2,y_3)\in D_1 \times D_1 \times D_2$ such that $y_1y_2y_3=1$. Using Corollary \ref{maxesF4} (ii) gives
$G = \langle y_1,y_3\rangle $ and (ii) holds.

Using Lemma~\ref{cyclo}(iii), we have
$\gcd (\Phi_a(q),\Phi_b(q)) = 1$ whenever
$a \in \{12,4\}$ and $b \in \{ 3, 8\}$ unless $a= 4$ and $b=8$. Since   $\gcd(\Phi_4(q)/\gcd(q-1,2),\Phi_8(q)) = 1$, we  have the required non-conjugacy condition. Hence $G$ is a Beauville group. \end{proof}

\subsection{Exceptional groups of type $^3\mathrm D_4$}

Set $G = {}^3\rm D_4(q)$. Pick $\sigma \in G$ of order $ \Phi_{12}(q)$ (see \cite{KL3D4}). Set $T= \langle \sigma \rangle$ and $N= N_G(T)$. Then, from \cite{KL3D4}, $N/T$ is cyclic of order $4$.
Set $C_1=\sigma^G$.

 Let $\tau_1 \in G$ be a regular semisimple element of order $\Phi_{3}(q)$ and $\tau_2 \in G$ regular semisimple of order $\Phi_{6}(q)$. For $i=1,2$, define $D_i = (\tau_i)^G$.

The following can be extracted from Kleidman \cite{KL3D4} where a complete list of maximal subgroups of ${}^3\mathrm D_4(q)$ is determined.

\begin{lemma} \label{max3D4}  Suppose that $M$ is a maximal  subgroup of $G$.
\begin{enumerate}
\item If $M \cap C$ is non-empty, then $ M = N$.
\item If $q > 2$, then either $M \cap D_1$ is empty or $M \cap D_2$ is empty.
\end{enumerate}
\end{lemma}\qed

\begin{lemma}\label{3d4x}
There exists an integer $k$ such that, setting $C_2 = (\sigma^k)^G$, there is a hyperbolic triple for $G$ in  $(x_1,x_2,x_3) \in C_1\times C_1 \times C_2$.
\end{lemma}

\begin{proof} As the class $C_1$ consists of regular semisimple elements and $C_2$ consists of semisimple
elements, Gow's Theorem guarantees that there exists $(x_1,x_2,x_3) \in C_1\times C_1 \times C_2$ with product $x_1x_2x_3=1$. We may suppose that $x_1= \sigma$. Then $C_1 \cap N \subseteq T$ and from Lemma~\ref{max3D4} (i), $N$ is the unique maximal subgroup of $G$ containing $T$. Therefore Lemma~\ref{structureconstant} applies with $k=4$. Since $\Phi_{12}(q) \ge 13$, we deduce that there is a hyperbolic triple in $ C_1\times C_1 \times C_2$.
\end{proof}

In the next lemma, $nX$ denotes a conjugacy class of elements  ${}^3\rm D_4(2)$ consisting of elements in \ATLAS conjugacy class $nX$ \cite{ATLAS}.

\begin{lemma}\label{3d4yy}
$^3{\rm D}_4(2)$  has a hyperbolic triple in  $7D\times 8A \times 9A$.
\end{lemma}

\begin{proof} The conjugacy class $7D$ is regular and contained only in the maximal
subgroups conjugate to $(7\times \PSL_2(7)):2$ or $7^2:2\Alt(4)$ (see \cite{ATLAS}). Evidently none
of these maximal subgroups contain elements of order $9$. The lemma follows as
the $(7D,8A,9A)$ structure constant is non-zero as can be checked form the character table or by using the computer.\end{proof}

\begin{lemma}\label{3d4y}
Suppose that $q > 2$. Then there exists a hyperbolic triple for $G$ in $D_1\times D_1 \times D_2$.
\end{lemma}

\begin{proof} As the classes $D_1$ and $D_2$ consist of regular semisimple elements, Gow's Theorem guarantees that there exist $(y_1,y_2,y_3) \in D_1\times D_1 \times D_2$ with product $1$. Now Corollary \ref{max3D4} (ii) gives
$G = \langle y_1,y_3\rangle $ if $q > 2$ and our claim follows.\end{proof}

\begin{theorem} \label{beau3D4} Let $G = {}^3{\rm D}_4(q)$. Then
 $G$ is a Beauville group.
\end{theorem}

\begin{proof}  Lemmas \ref{3d4x},  \ref{3d4y} and  \ref{3d4yy} show that we have hyperbolic triples in $G$, which, as
$\gcd(\Phi_{12}(q),\Phi_b(q)) = 1$ whenever $b \in \{  3 , 6  \}$, have members of pairwise coprime orders. This proves the theorem.\end{proof}

\subsection{Exceptional groups of type $\G_2$}

Suppose that $G= \G_2(q)$.  This time using Lemma~\ref{gcd} or just by writing down the polynomials  we have that $\gcd (\Phi_3(q),\Phi_1(q)) = 3$ or $1$ depending on whether or not $q \equiv 1 \pmod 3$
and  that $\gcd(\Phi_6(q),\Phi_2(q)) = 3$ or $1$ depending on whether or not $q \equiv -1 \pmod3$. Also
$\gcd(\Phi_6(q),\Phi_3(q)) = 1 $ and $\gcd(q+1,q-1) = 1 + j$ where $j = q \pmod  2$.

Assume that $q > 7$.
We pick a 4-tuple of numbers $(k_1,k_2,k_3,k_6)$ where, for $i \in \{1,2\}$,       $k_i$ is a   divisor of $\Phi_i(q)$  chosen so that
 $\gcd(k_1,k_2)=1$, $k_3 = \Phi_3(q)/\gcd(\Phi_3(q),3)$ and $k_6= \Phi_6(q)/\gcd(\Phi_6(q),3) $ (here we note that $9$ does not divide $\Phi_3(q)$ or $\Phi_6(q)$). Evidently, then our choices guarantee that
the numbers $k_1, k_2, k_3$ and $k_6$ are pairwise coprime.

We select $\sigma_1 \in G$ of order $ k_6$ and $\sigma_2 \in G$ regular semisimple
element of order $k_{1}$. For $i=1,2$ define $C_i =(\sigma_i)^G$. Let $\tau_1 \in G$ be an element of order $ k_{3}$, $\tau_2 \in G$ be regular semisimple of order $k_2$  and, for $i=1,2$, define $D_i = (\tau_i)^G$.

Since elements of order $k_6$ and $k_3$ are not contained in any proper subfield subgroup $\G_2(q_0)$ of $G$, the following lemma can be  deduced from \cite[Corollary 11]{As1} where the maximal subgroups of $\G_2(q)$ are listed.

\begin{lemma} \label{maxG2} Let $M$ be a maximal subgroup of $G$. The following are true:
\begin{enumerate}
\item If  $  C_1\cap M$ is non-empty, then $ M \cong  {\rm SU}_3(q).2$.
\item  If  $D_1 \cap M$ is non-empty, then $ M \cong  {\rm SU}_3(q).2$.
\end{enumerate}
\end{lemma}

\begin{corollary}  \label{maxesG2} Let  $M$ be a subgroup of $G$.
\begin{enumerate}
\item If  $C_1\cap M$ and $C_2\cap M$ are both non-empty, then $M = G$.
\item If  $D_1\cap M$ and $D_2\cap M$ are both non-empty, then $M = G$.
\end{enumerate}
\end{corollary}
\begin{proof} We note that if $M \cong  {\rm SU}_3(q).2$, then $M$ contains no regular semisimple elements of order $k_1$ and that
if $M \cong  {\rm SL}_3(q).2$, then $M$ contains no regular semisimple elements of order $k_2$. Our claims now follow from Lemma \ref{maxG2}.\end{proof}

\begin{theorem} \label{beauG2} Let $G = {\rm G}_2(q)$, $q > 7$. Then there exist hyperbolic triples for $G$ in $C_1 \times C_2\times C_2$ and in $D_1\times D_1 \times D_2$.
In particular,  $G$ is a Beauville group.
\end{theorem}

\begin{proof} As the class $C_1$ consists of regular semisimple elements and $C_2$ consists of semisimple
elements, Gow's Theorem guarantees that there exist $(x_1,x_2,x_3)\in C_1 \times C_1 \times C_3$ with product $x_1x_2x_3=1$. Now Corollary \ref{maxesG2} gives
$G = \langle x_1,x_3\rangle $ and our first claim follows. The same argument with $C_1$ replaced by $D_1$ and $C_2$ replaced by $D_2$ proves the second claim.
Since  $k_1, k_2, k_3$ and $k_6$ are pairwise coprime, we have the non-conjugacy condition required to show that $G$ is a Beauville group.\end{proof}

For small fields we have the following facts.

\begin{lemma}\label{G2calc}
\begin{enumerate}
\item $\G_2(7)$ is $(2,3,7)$ and $(43,43,19)$ generated.
\item $\G_2(5)$ is $(2,3,7)$ and $(31,31,5)$ generated.
\item $\G_2(4)$ is $(7,7,7)$ and $(13,13,13)$ generated.
\item $\G_2(3)$ is $(7,7,7)$ and $(13,13,13)$ generated.
\item $\G_2(2)' \cong \SU_3(3)$ is $(7,7,7)$ and $(8,8,8)$ generated.
\end{enumerate}
\end{lemma}

\begin{proof} Malle \cite{Malle} has proved the $(2,3,7)$ generation of $\G_2(7)$ and $\G_2(5)$ (they are Hurwitz groups).
The elements of order $43$ and the elements of order $19$ in $\G_2(7)$ are regular semisimple, so Gow's Theorem guarantees the existence of two elements of order $43$ whose product has order $19$. Now
we check using Aschbacher's list of maximal subgroups \cite{As1} that no maximal subgroup of
$\G_2(7)$ is divisible by both $43$ and $19$. This proves (i). Now from the character table of $\G_2(5)$, we get that for all classes
of elements of order $5$ in $\G_2(5)$ and all classes of elements of order $31$ we have
a nontrivial $(31,31,5)$ structure constant. The only maximal subgroup of $\G_2(5)$ which
contains elements of order $31$ is $\PSL_3(5)$, but $\PSL_3(5)$ only has two conjugacy classes of elements of order $5$.
 Thus (ii) is proved. Parts (iii), (iv) and (v)  were verified using the permutation
representations of degree 416,  351 and 28 respectively. The requisite systems were found by
random search in a matter of milliseconds.\end{proof}

This completes the investigation of $\G_2(q)$.
\subsection{Exceptional groups of type ${}^2\B_2$, ${}^2\G_2$, and  ${}^2\F_4$}

If $G$ is of type  ${}^2\B_2(q)$, ${}^2\G_2(q)$, or ${}^2\F_4(q)$, then complete lists of maximal subgroups are known and are conveniently listed in \cite[Theorem 4.1, Theorem 4.2 and Theorem 4.5]{WilsonBook}. In particular, so long as $q>p$ which we now assume, such $G$ contain two classes of maximal
subgroups $N_1=N_G(T_1)$ and $N_2=N_G(T_2)$ such that $T_1$ and $T_2$ are cyclic. Specifically, if $G = {}^2\B_2(q)$, then $T_1$ has order $q+ \sqrt {2q}+1$, $T_2 $ has order  $q- \sqrt {2q}+1$ and $N_G(T_1)/T_1 \cong N_G(T_2)/T_2$ is cyclic of order $4$, if  $G = {}^2\G_2(q)$, then $T_1$ has order $q+ \sqrt {3q}+1$, $T_2 $ has order  $q- \sqrt {3q}+1$ and $N_G(T_1)/T_1 \cong N_G(T_2)/T_2$ is cyclic of order $8$ and if $G = {}^2\F_4(q)$, then $|T_1|= q^2+q+1+\sqrt{2q}(q+1)$, $|T_2|= q^2+q+1-\sqrt{2q}(q+1)$ and $N_G(T_1)/T_1 \cong N_G(T_2)/T_2$ is cyclic of order $12$.

In all cases the lists of maximal subgroups of $G$ show that the normalizers of  $T_1$ and $T_2$ are the unique maximal subgroups of $G$ which contain $T_1$ and $T_2$ respectively.
Let $\sigma$ be a generator for $T_1$ and $\tau$ be a generator for $T_2$, and  set $C_1 = \sigma^G$   and $D_1 = \tau^G$.  If $y \in C_1\cap N$, then $\langle y \rangle = T_1$ and so $y$ and $\sigma$ are conjugate in $N_1$. Similarly, we have if $y \in D_1 \cap T_2$, then $y$ and $\tau $ are conjugate in $N_2$. Also, by Gow's Theorem, we have that the $(C_1,C_1,X_1)$ and the $(D_1,D_1,Y_1)$ structure constants are non-zero where $X_1$ and $Y_1$ are arbitrary semisimple conjugacy classes of $G$.  Therefore, so long as $q>2^3$, $q>3^3$ and $q>2^5$ for $G$ of type ${}^2\B_2(q)$, ${}^2\G_2(q)$, or ${}^2\F_4(q)$ respectively, we may apply Lemma~\ref{structureconstant}. This  shows, under the specified conditions on $q$, that there exists $k_1$ and $k_2$ such that setting  $C_2= (\sigma^{k_1})^G$ and
$D_2=(\tau^{k_2})^G$, we have hyperbolic triples in $C_1 \times C_1 \times C_2$ and $D_1\times D_1\times D_2$. Thus, as $|T_1|$ and $|T_2|$ are coprime, we have the following theorem.

\begin{theorem} \label{beauexsm} If $G$ of type ${}^2\B_2(q)$, ${}^2\G_2(q)$, or ${}^2\F_4(q)$ and $q>2^3$,$q>3^3$ and $q>2^5$ respectively, then
$G$ is a Beauville group.
\end{theorem}\qed

The remaining cases can either be checked with computer, or  using the facts above about $N_i$ and the existence
of a third semisimple element of requisite order.  We thus obtain the following lemma.

\begin{lemma}\label{exccalc}
\begin{enumerate}
\item ${}^2\B_2(8)$ is $(5,5,5)$ and $(13,13,13)$ generated.
\item ${}^2\G_2(3)' = \PSL_2(8)$ is $(7,7,7)$ and $(9,9,9)$ generated.
\item ${}^2\G_2(27) $ is $(37,37,7)$ and $(19,19,13)$ generated.
\item ${}^2\F_4(2)'$ is $(13,13,13)$ and $(16,16,16)$ generated.
\item ${}^2\F_4(8)$ is $(37,37,7)$ and $(101,109,13)$ generated.
\item ${}^2\F_4(32)$ is $(793,793,33)$ and $(1321,1321,31)$ generated.
\end{enumerate}
\end{lemma}\qed

\section{Double covers of $\Alt(n)$}\label{SS6}

The alternating groups for $n \ge 6$ have previously been shown to be Beauville groups in \cite{FG} however it is not clear how to lift their hyperbolic triples to the double cover of $\Alt(n)$ while maintaining the conjugacy criteria. Thus we present a further proof of their result.
We first present the following lemma.
\begin{lemma}\label{Alt(n)} The following hold:
\begin{enumerate}
\item If $n$ is odd with $n \ge 7$, $\Alt(n)$ contains a hyperbolic triple of type $(n-2,n-2,5)$.
\item If $n$ is even with $n \ge 6$, $\Alt(n)$ contains a hyperbolic triple of type $(\lcm(3,n-3),n-2,3)$.
\end{enumerate}
\end{lemma}

\begin{proof} Let $\Omega= \{1,2,\dots,n\}$. Suppose that $n$ is odd. Let $x = (1,2, \dots, n-2)$ and $y = (n,n-1,\dots, 3)$. Then $z=xy =(1,2,n,n-1,n-2)$ is a $5$-cycle. On the other hand $xy^{-1}=(1,2,4,\dots, n-1,n, 3, 5, \dots n-2)$ is an $n$-cycle. Let $H =\langle x,y\rangle$. Then  $H$ is transitive on $\Omega$. Since $ z^{x^2}= (2,3,4,n-1,n-2)$, we have that $H$ is $2$-transitive. As $n \ge 9$,  Jordan's Theorem \cite[Theorem 13.9]{Wielandt} implies that $H= \Alt(n)$. We check by hand that when $n=7$, then the same result holds. This proves (i).

Suppose that $n$ is even. Let $x = (1,2)(n, \dots, 3)$ and $y = (1,\dots,n-3)(n-2,n-1,n)$. Then $z=xy =(1,3,n-2)$.  Set $H= \langle x,y\rangle$. Then $H$ is transitive on $\Omega$.
Since $xyx^2$ is an $n-1$-cycle fixing $2$, we have that $H$ acts $2$-transitively on $\Omega$. Thus $H= \Alt(n)$ by  \cite[Theorem 13.9]{Wielandt}. This proves (ii).
\end{proof}

Now we have to consider the double cover of $\Alt(n)$.
We take as our standard copy of the double cover of $\Sym(n)$ the one described by the presentation

\begin{eqnarray*}\langle t_1, \dots, t_{n-1}, z &\mid &z= t_i^2=(t_it_j)^2= z, z^2, t_kt_{k+1}t_k = t_{k+1}t_kt_{k+1}, \cr &&1\le i,j \le n-1, |i-j| > 1, 1\le k \le n-2\rangle
\end{eqnarray*}
Here we have that the image in $\Sym(n)$ of $t_i$ is the transposition $(i,i+1)$ and $z$ is the involution in the centre. We denote the image of elements of $2\udot \Sym(n)$ in $\Sym(n)$ by using a tilde above the letter.

Using Jordan's Theorem \cite[Theorem 13.9]{Wielandt} when $n$ is odd, we have a hyperbolic triple $(\wt x, \wt y, \wt {xy})$ for $\Alt(n)$ of type $(n,3,n)$ where $\wt x= (1, \dots,n)$ and $\wt y=(1,2,n)$, we have $\wt x \wt y = (1,n,2, \dots, n-1)$.

Now we consider $x = t_{n-1} \dots t_1$, $y = t_1t_{1}^{t_{2}\dots t_{n-1}}z$. Then $x$ and $y$ project to $\wt x$ and $\wt y$ respectively.

\begin{lemma}\label{order3} We have that $o(y) = 6$ if $n$ is even and $o(y)= 3$ if $n$ is odd.
\end{lemma}

\begin{proof} If $n= 3$, then $y = t_1t_1^{t_2}z = t_1t_2^{t_1}z= (t_1t_2)^{t_1}z$. Now $$y^3=(t_1t_2)^3z = t_1t_2t_1t_2t_1t_2z= t_2t_1t_2t_2t_1t_2z= z^4= 1.$$ Hence $o(y)= 3$ in this case. Now suppose that $n\ge 4$ and that the result holds from $n-1$. Then, as $t_{1}^{t_{n-1}} = t_1z$,  we have

$$y= t_1t_{1}^{t_{2}\dots t_{n-1}}z= (t_1t_{1}^{t_{2}\dots t_{n-2}})^{t_{n-1}}.$$
Hence if $n$ is odd, we have that $n-1$ is even and $t_1t_{1}^{t_{2}\dots t_{n-2}}z$ has order $6$, which means that $y$ has order $3$ and, if $n$ is even, $t_1t_{1}^{t_{2}\dots t_{n-2}}z$ has order $3$ and $y$ has order $6$. This proves the lemma.
\end{proof}

\begin{lemma}\label{xsimz} The elements  $x$ and $xy$ are conjugate.
\end{lemma}

\begin{proof}
We have \begin{eqnarray*}
xy &=&(t_{n-1} \dots t_1) t_1 t_1^{t_2\dots t_{n-1}}z\\
&=& t_{n-1} \dots t_2t_1 t_1 t_{n-1}^{-1} \dots t_2^{-1} t_1 t_2 \dots t_{n-1}z\cr
&=& t_{n-1} \dots t_2 t_{n-1}^{-1} \dots t_2^{-1} t_1 t_2 \dots t_{n-1}\cr
&=& t_{n-1}^{-1} \dots t_2^{-1}t_{n-1} \dots t_2  t_1 t_2 \dots t_{n-1}\cr
&=& (t_{n-1}\dots t_1)^{t_2 \dots t_{n-1}}= x^{t_2 \dots t_{n-1}}.
\end{eqnarray*}
\end{proof}

\begin{lemma} \label{nodd}Suppose that $n$ is odd. Then either $(x,y, xy)$ is a hyperbolic triple of type $(n,3,n)$ for $2\udot \Alt(n)$  or $(xz,y,xyz)$ is a hyperbolic triple of type $(n,3,n)$ for $2 \udot \Alt(n)$.
\end{lemma}

\begin{proof} By Lemma~\ref{order3}, $y$ has order $3$ and by Lemma~\ref{xsimz} $x$ and $xy$ are conjugate. It follows that $xz$ and $xyz$ are also conjugate. Now, as $n$ is odd, either $x$ has order $n$ or $xz$ has order $n$, so as $\Alt(n)= \langle \wt x, \wt y\rangle$, we have the result.
\end{proof}

\begin{lemma}\label{neven} Suppose that $n \ge 6$ is  even. Then $2\udot \Alt(n)$ has a hyperbolic triple of type $(5,n-1,n-1)$. \end{lemma}

\begin{proof}
We first consider $2\udot \Alt(6)$. We take elements $\wt x= (1,2,3,4,5), \wt y=(2,3,5,4,6)$  and $\wt z= \wt x\wt y= (2,5,1,3,6)$ in $\Alt(6)$. These three elements generate $Alt(6)$ and so we have a $(5,5,5)$ hyperbolic triple for $\Alt(6)$ and we check using either Gap or Magma that in $2\udot \Alt(6)$, we have elements $a$, $b$ and $c=ab$ of order $5$ which project to these elements and product correctly. Thus we may suppose that $n \ge 8$.
We now take  elements $x$ and $y$ of $2\udot \Alt(n)$ whose images  in $\Alt(n)$ are $\wt x=(1,2,3,4,5)$ and $\wt y= (2,3,5,4,7,8, \dots, n, 6)= (2,3,5,4,6)(7, \dots, n,6)$.  Note that $H=\langle x,y\rangle$ is $2$-transitive and contains a $5$-cycle.  Thus $H= \Alt(n)$. Note that we have $\wt x \wt y = (2,5,1,3,6)(7,\dots, n,6)$. Now let $a$ and $b$ be the elements in $2\udot \Alt(6)$ (considered as a subgroup of $2\udot \Alt(n)$) of order $5$ as in the first paragraph. Let $\wt w = (7,\dots n, 6)$. Note that $a$ and $b$ are conjugate by any element from $2\udot \Alt(n)$ which projects to $(2,3)(1,5)$. It follows that if $w$ is any lift of $\wt w$, then we have that the elements  $yw$ and $zw$ are in fact conjugate and so have the same order. Since we can select  a lift of $\wt w$ such that $yw$ has order $n-1$, we now see that there is a triple of elements  $(xyw,(xyw)^{-1})$ of type $(5,n-1,n-1)$ as claimed.
\end{proof}

\begin{theorem}\label{AltThm} If $G = 2\udot\Alt(n)$ with $n \ge 6$, then $G/Z(G)$ and $G$ are Beauville groups.
\end{theorem}

\begin{proof} For $G$ we take the triples given by Lemma~\ref{nodd}, \ref{neven} and any preimages of the two triples given in Lemma~\ref{Alt(n)}. These triples demonstrate that $G$ is a Beauville group.
\end{proof}

Finally, for this section,  we note that the exceptional covers of $\Alt(6)$ and $\Alt(7)$ are dealt with in Section~\ref{exceptionalsec}.
\section{Sporadic Beauville groups}\label{SS7}

\begin{center}
\begin{table}[h]\caption{The types of the Beauville structures defined by the words given in Table \ref{sporadic2}, Lemma \ref{Baby} and \ref{Monster}.}
\label{sporadic1}
\begin{tabular}{|c|c|c|c|}
\hline
$G$&Types&$G$&Types\\
\hline
$\M_{11}$&((6,6,6),(11,11,11))&$3\dot\ON$&((33,33,33),(19,19,12))\\
$2\udot \M_{12}$&((11,11,11),(6,6,6))&$\Co_3$&((14,14,24),(5,5,22))\\
$\J_1$&((7,7,7),(19,19,19))&$\Co_2$&((5,5,23),(28,28,24))\\
$12\udot\M_{22}$&((5,5,6),(33,33,33))&$6\udot \Fi_{22}$&((13,13,33),(42,42,30))\\
$2\udot \J_2$&((7,7,7),(3,3,12))&$\HN$&((22,22,25),(19,19,15))\\
$\M_{23}$&((4,4,11),(23,23,23))&$\Ly$&((67,67,7),(37,37,24))\\
$2\udot \HS$&((5,5,12),(11,11,11))&$\Th$&((19,19,27),(31,31,24))\\
$3\udot \J_3$&((3,3,57),(9,9,9))&$\Fi_{23}$&((3,3,14),(5,5,5))\\
$\M_{24}$&((3,3,23),(12,12,12))&$2\udot \Co_1$&((15,15,15),(14,14,14))\\
$3\udot \McL$&((5,5,33),(7,7,12))&$\J_4$&((37,37,66),(43,43,23))\\
$\He$&((7,7,17),(12,12,12))&$3\udot \Fi'_{24}$&((3,3,29),(5,5,5))\\
$2\udot \Ru$&((4,4,20),(13,13,13))&$2\udot \mathbb{B}$&((31,23,3),(47,47,5))\\
$6\udot \Suz$&((3,3,26),(8,8,10))&$\mathbb{M}$&((2,3,7),(94,71,71))\\
\hline
\end{tabular}
\end{table}
\end{center}

\begin{center}
\begin{table}\caption{Words in terms of the standard generators $a$ and $b$ defining a Beauville structure for the full covering group of each of the sporadic simple groups of smaller order than that of the baby monster. The elements $g_i$ have the property that if we define the elements $y_i=x_i^{g_i}$ for $i=1,2$ then $\{\{x_1,y_1\},\{x_2,y_2\}\}$ is a Beauville structure for the given group.}
\label{sporadic2}
\begin{tabular}{|c|c|c|c|c|}
\hline
$G$&$x_1$&$g_1$&$x_2$&$g_2$\\
\hline
$\M_{11}$&$ab$&$ba$&$ab^2$&$b$\\
$2\udot \M_{12}$&$ab$&$bab^2$&$ab^2ab$&$b$\\ $
\J_1$&$ab$&$bab^2$&$ab^2ab$&$b$\\ $
12\udot \M_{22}$&$ab^2$&$b$&$ab$&$b^2$\\ $
2\udot \J_2$&$ab$&$(ab)^2b(ab^2)^2$&$b$&$(ab)^2a$\\ $
\M_{23}$&$b$&$aba$&$ab$&$ba$\\ $
2\udot \HS$&$b$&$ab^2a$&$ab$&$b^2ab^2$\\ $
3\udot \J_3$&$b$&$a$&$abab^2$&$b$\\ $
\M_{24}$&$b$&$a$&$ab^2ab$&$b$\\ $
3\udot \McL$&$b$&$a$&$(abab^2)^2$&$b$\\ $
\He$&$b$&$a$&$ab^2$&$b$\\ $
2\udot\Ru$&$b$&$aba$&$ab$&$b^2$\\ $
6\udot \Suz$&$b$&$a$&$(ab)^3(ba)^2b^2$&$ab$\\ $
3\udot \ON$&$ab$&$b^2$&$ab^2$&$b$\\ $
\Co_3$&$ab$&$a$&$(ab^3)^2$&$aba$\\ $
\Co_2$&$b$&$ab^3ab^2a$&$ab$&$b^2$\\ $
6\udot \Fi_{22}$&$b$&$a$&$ab^2$&$b$\\ $
\HN$&$(ab)^3$&$ab^2abab^2$&$(ab)^2(ba)^6(b^2a)^2(ba)^5b^2ab^2$&$ab^2ab$\\ $
\Ly$&$(ab)^2b^2$&$a$&$(ab)^3(ba)^2b^3$&$a$\\ $
\Th$&$ab$&$(ab^2)^3$&$(ab)^6a(b^2a)(ba)^4b(ba)^8b^2$&$b$\\ $
\Fi_{23}$&$b$&$a$&$ab^2ab$&$b$\\ $
2\udot \Co_1$&$(ab^2)^2(ab)^2$&$b$&$ab^2(ab)^3$&$(ab^2)^2$\\ $
\J_4$&$ab$&$bab^2$&$aba(b^2a)^2(ba)^2b^2$&$a$\\ $
3\udot \Fi'_{24}$&$b$&$a$&$ab^2ab$&$b$\\
\hline
\end{tabular}
\medskip

\end{table}
\end{center}

For the full covering groups of each of the sporadic simple groups of order smaller than that of the baby monster we give explicit words in terms of the standard generators in Table \ref{sporadic2} that define a Beauville structure for that group (see \cite{WilsonStandardGenerators} for background information on standard generators). These may be readily converted to explicit permutations and matrices using the data on the online {\sc Atlas} \cite{onlineatlas}. In some cases specific information on the conjugacy class is required--this is readily obtained in \cite{ATLAS,onlineatlas}. In some cases additional computation is required to check the non-conjugacy of certain elements. Beauville structures for other covering groups (for example, $6\udot \M_{22}$) may be obtained by taking images of these generators modulo subgroups of the centre. For the groups $2\udot \mathbb{B}$ and $\mathbb{M}$ we have the following.

\begin{lemma}\label{Baby}
The double cover of the baby monster group $2\udot \mathbb{B}$ has hyperbolic triples of type $(31,23,3)$ and $(47,47,5)$. In particular, $2\udot \mathbb B$ and $\mathbb B$ are Beauville groups.
\end{lemma}

\begin{proof} We use the maximal subgroup structure of $\mathbb B$ as given in \cite{baby}.
If $x_1\in 2\udot \mathbb{B}$ has order 47, then the only maximal subgroup of $2\udot \mathbb{B}$ containing $x_1$ is the normalizer of $\langle x_1\rangle$ which is isomorphic to the Frobenius group $2\times 47:23$. Therefore if $y_1\in 2\udot \mathbb{B}$ is an element of order 47 such that $o(x_1y_1)\not=47$ then $\langle x_1,y_1\rangle=2\udot \mathbb{B}$. Straightforward structure constant calculations using GAP show that we can choose $y_1$ so that $o(x_1y_1)=5$.

Let $x_2\in 2\udot\mathbb{B}$ have order 31 and let $y_2\in 2\udot \mathbb{B}$ have order 23. No maximal subgroup of $2\udot \mathbb{B}$ contains elements of both of these orders, thus $\langle x_2,y_2\rangle=\mathbb{B}$. Finally, a structure constant calculation shows that $x_2$ and $y_2$ may be chosen such that $x_2y_2$ has order $3$.
\end{proof}

\begin{lemma}\label{Monster}
The Monster group $\mathbb{M}$ has hyperbolic triples of type $ (2,3,7)$ and $(94,71,71)$. In particular, the monster is a Beauville group.
\end{lemma}

\begin{proof}
In \cite{Wilson} Wilson proves that $\mathbb{M}$ can be generated by an element $x_1$ of class 2B and an element $y_1$ of class 3B such that $x_1y_1$ is of class 7B.

Let $x_2\in\mathbb{M}$ be an element of order 94 (so $x_1^{47}$ is in class 2A). It is known that the only maximal subgroups of $\mathbb{M}$ containing elements of order 94 are copies of $2\udot \mathbb{B}$ (see \cite[Theorem 21]{AM}). Let $y_2\in\mathbb{M}$ be an element of order 71. Since $\mathbb{B}$ contains no elements of order 71 we have that $\langle x_2,y_2\rangle=\mathbb{M}$. Finally, using structure constants  we calculate that $y_2$ may be chosen so that $o(x_2y_2)=71$.
\end{proof}

\begin{theorem}\label{sporthm}
Suppose that $G$ is a quasisimple and $G/Z(G)$ is a sporadic simple group. Then $G$ is a Beauville group.
\end{theorem}
\begin{proof}
This follows from the data presented in  Table~\ref{TSp} and  Lemmas~\ref{Baby} and \ref{Monster}.
\end{proof}

\section{Exceptional covers of simple groups}\label{exceptionalsec}

In Table \ref{exceptional2} we give words in terms of the standard generators for hyperbolic triples  in the exceptional covers of the simple groups of Lie type. It follows that each of these groups is a Beauville group.  In Table \ref{exceptional2} we give the types of the Beauville structures defined by the words given in Table \ref{exceptional1}.

In each case it is straightforward to verify that the words given generate $G$ and that the elements have order as stated, with the exception of $(2^2\times3)\udot {}^2\E_6(2)$. Standard generators for  this group as $3462$ dimensions over $\GF(2)$ can be found on the website of the third author. In this case we see from the list of maximal subgroups given in the {\sc Atlas} \cite[pg. 191]{ATLAS} (noting the minor corrections given in \cite[pg. 304]{Moat}) that the only maximal subgroups of $^2$E$_6$(2) that has elements of order 19 are copies of $U_3(8):3$, which contains no elements of order 5 or 66. Direct computations show that $o(x_1y_1x_1^4y_1^4)=o(x_2y_2x_2y_2^{13})=19$.

\begin{center}
\begin{table}
\begin{tabular}{|c|c|}
\hline
$G$&Type\\
\hline
$6\udot \Alt(6)$&((5,5,5),(8,8,8))\\
$6\udot \Alt(7)$&((7,7,14),(5,5,3))\\
$2^2\udot {}^2\B_2(8)$&((5,5,5),(7,7,26))\\
$2\udot \Sp_6(2)$&((7,7,8),(9,9,15))\\
$(2^2\times3)\udot \U_6(2)$&((7,7,7),(33,33,33))\\
$2\udot \G_2(4)$&((5,5,10),(13,13,13))\\
$3\udot \G_2(3)$&((13,13,13),(8,8,7))\\
$2\udot \F_4(2)$&((3,3,17),(14,14,5))\\
$6\udot \Omega_7(3)$&((7,7,13),(120,120,120))\\
$2^2\udot \Omega_8^+(2)$&((5,5,15),(14,14,8))\\
$(3^2\times4)\udot \U_4(3)$&((6,6,14),(5,5,5))\\
$(2^2\times3)\udot {}^2\E_6(2)$&((5,5,5),(66,66,3))\\
\hline
\end{tabular}
\medskip

\caption{The types of the Beauville structures defined by the words given in Table \ref{exceptional2} for exceptional quasisimple groups.}
\label{exceptional1}
\end{table}
\end{center}

\begin{center}
\begin{table}
\caption{Words in terms of the standard generators $a$ and $b$ defining a Beauville structure for the full covering group for the simple groups of Lie type with exceptional Schur multipliers. The elements $g_i$ have the property that if we define the elements $y_i=x_i^{g_i}$ for $i=1,2$ then $(x_1,y_1, x_1,y_1) $ and $(x_2,y_2,x_2y_2)$ are hyperbolic triples which exhibit $G$ as a Beauville group.}
\label{exceptional2}
\begin{tabular}{|c|c|c|c|c|}
\hline
$G$&$x_1$&$g_1$&$x_2$&$g_2$\\
\hline
$6\udot \Alt(6)$&$bb^{aba}$&$ab$&$b$&$(abab^2a)^2$\\
$6\udot \Alt(7)$&$ab$&$b^2ab^3$&$b$&$(ab)^2a^2b^2a$\\
$2^2\udot {}^2\B_2(8)$&$ab$&$b^2$&$ab^2$&$b$\\
$2\udot \Sp_6(2)$&$b$&$a$&$ab$&$b^3ab^2$\\
$(2^2\times3)\udot \U_6(2)$&$b$&$aba$&$ab$&$b^2$\\
$2\udot \G_2(4)$&$b$&$(ab)^2a$&$ab$&$bab^2$\\
$3\udot \G_2(3)$&$ab$&$bab^2$&$ab^2(ab)^2$&$aba$\\
$2\udot \F_4(2)$&$b$&$a$&$ab^2(ab)^4$&$b$\\
$6\udot \Omega_7(3)$&$b$&$(ab)^4a$&$ab^2$&$b^3$\\
$2^2\udot \Omega_8^+(2)$&$b$&$ab(ba)^2$&$(ba)^3b^3a$&$ab^2a$\\
$(3^2\times4)\udot \U_4(3)$&$b$&$a$&$ab^3$&$b(ab^4)^2bab^4$\\
$(2^2\times3)\udot {}^2\E_6(2)$&$(ab)^2b$&$b$&$ab^2(ab)^2$&$b$\\
\hline
\end{tabular}
\medskip

\end{table}
\end{center}
We summarize the results of this section by stating the main result.
\begin{theorem}\label{excepcovers}
The exceptional covers of the groups of Lie type and their central quotients are all Beauville groups.\qed
\end{theorem}
\section{Proof of Theorem~\ref{main}}

Here we finally gather our partial theorems  to present the proof of our main theorem.

\begin{theorem}\label{main1}
With the exceptions of $\SL_2(5)$ and $\PSL_2(5)(\cong \Alt(5)\cong \SL_2(4)$), every non-abelian finite quasisimple group is a Beauville group.
\end{theorem}

\begin{proof}  We first suppose that $G$ is a quasisimple group of Lie type. If $G = \SL_2(q)$ with $q>5$, then we refer to the proof of \cite[Theorem 2.2]{FJ} where for $q \ge 11$, they present hyperbolic pairs of type $((q+1)/2, (q+1)/2,p)$ and $((q-1)/2,(q-1)/2,(q-1)/2)$ for $G$ which they also note provide hyperbolic triples for $\PSL_2(q)$. They also prove that $\PSL_2(8) \cong \SL_2(8)$ is a Beauville group.  By computer we determine $\SL_2(11)$ has hyperbolic triples of type $(5,5,11)$ and $(12,12,12)$, $\SL_2(7)$ has hyperbolic triples of type $(8,8,8)$ and $(7,7,3)$. Of course $\SL_2(9) \cong 2\udot \Alt(6)$ and so this case is covered by Theorem~\ref{AltThm}. The other linear groups are Beauville groups by Theorem~\ref{SLThm}. The unitary groups are Beauville groups  by Theorem~\ref{Uthm}, the symplectic groups are Beauville groups by Theorem~\ref{Spthm}. Finally the spin groups and their quotients are covered by Theorem~\ref{Othm}.  Theorems~\ref{beauE8}, \ref{beauE7}, \ref{beauE6}, \ref{beau2E6}, \ref{beauF4}, \ref{beau3D4}, \ref{beauG2} and \ref{beauexsm} together with Lemma~\ref{G2calc} and \ref{exccalc} prove that the exceptional quasisimple groups of Lie type are Beauville groups.
The alternating groups $\Alt(n)$ for $n \ge 6$ and their double covers are Beauville groups by
Theorem~\ref{AltThm}.  The sporadic simple groups and their covers are Beauville groups by Theorem~\ref{sporthm}. This leaves the exceptional covers of the alternating groups and the groups of Lie type and these are the subject of Theorem~\ref{excepcovers}.
\end{proof}

\end{document}